\documentclass{amsart}
\usepackage[utf8]{inputenc}
\usepackage{a4wide}
\usepackage{amssymb,amscd,hyperref,amsbsy}
\usepackage{mathtools}
\usepackage{lmodern}  
\usepackage[T1]{fontenc}
\usepackage[british]{babel}
\usepackage{verbatim} % \comment
\usepackage[shortlabels]{enumitem}
\usepackage{tikz,subfigure}
\usepackage{tikz-cd} %comm. diagrams
\usepackage{tabularx}
\newcolumntype{C}[1]{>{\centering}m{#1}} 
\usepackage{hhline}
\usepackage{longtable}
\usepackage{float} % Improved interface for floating objects
\usepackage{wasysym}
\usepackage[ruled,english]{algorithm2e}
\SetAlFnt{\scriptsize}
\SetAlCapFnt{\footnotesize}
\usepackage[all]{xy}
\usepackage{aligned-overset}
\usepackage{rotating}
\usepackage{multirow}
\usepackage{textcmds}

\newcommand{\GZ}{{\mathbb Z}}
\newcommand{\GR}{{\mathbb R}}
\newcommand{\GN}{{\mathbb N}}

\newcommand{\balpha}{\boldsymbol{\alpha}}
\newcommand{\bl}{\boldsymbol{\ell}}

\newcommand{\bp}{\mathbf{p}}

\newcommand{\bM}{\mathbf{M}}

\newcommand{\bx}{\mathbf{x}}
\newcommand{\be}{\mathbf{e}}

\newtheorem{lemma}{Lemma}[section]
\newtheorem{theorem}[lemma]{Theorem}

\newcommand{\thistheoremname}{}
\newtheorem*{genericthm}{\thistheoremname}

\theoremstyle{definition}

\theoremstyle{remark}

\newtheorem{definition}[lemma]{Definition}

\numberwithin{equation}{section}
\definecolor{americanrose}{rgb}{1.0, 0.01, 0.24}
\definecolor{awesome}{rgb}{1.0, 0.13, 0.32}
 \definecolor{ballblue}{rgb}{0.13, 0.67, 0.8}
 \definecolor{amber}{rgb}{1.0, 0.75, 0.0}
 \definecolor{pastelgreen}{rgb}{0.01, 0.75, 0.24}
 \definecolor{bananamania}{rgb}{0.98, 0.91, 0.71}
 \definecolor{jazzberryjam}{rgb}{0.65, 0.04, 0.37}
\definecolor{lava}{rgb}{0.81, 0.06, 0.13}
\definecolor{lavenderpurple}{rgb}{0.59, 0.48, 0.71}
\definecolor{lighttaupe}{rgb}{0.7, 0.55, 0.43}
\definecolor{pinegreen}{rgb}{0.0, 0.47, 0.44}
\definecolor{spirodiscoball}{rgb}{0.06, 0.75, 0.99}
\definecolor{tangelo}{rgb}{0.98, 0.3, 0.0}
\definecolor{turquoise}{rgb}{0.19, 0.84, 0.78}
\definecolor{teagreen}{rgb}{0.82, 0.94, 0.75}
\definecolor{powderblue(web)}{rgb}{0.69, 0.88, 0.9}
\definecolor{mauvelous}{rgb}{0.94, 0.6, 0.67}
\definecolor{deepsaffron}{rgb}{1.0, 0.6, 0.2}
\definecolor{darkgoldenrod}{rgb}{0.72, 0.53, 0.04}
\definecolor{bleudefrance}{rgb}{0.19, 0.55, 0.91}
\definecolor{caribbeangreen}{rgb}{0.0, 0.8, 0.6}
\definecolor{canaryyellow}{rgb}{1.0, 0.94, 0.0}

\newcommand{\R}{\mathbb{R}}

\newcommand{\tra}[1]{\vphantom{#1}^t #1}
\newcommand{\hide}[1]{{}}

\usepackage{listings}
\lstdefinelanguage{Sage}[]{Python}
{morekeywords={False,sage,True},sensitive=true}
\definecolor{dblackcolor}{rgb}{0.0,0.0,0.0}
\definecolor{dbluecolor}{rgb}{0.01,0.02,0.7}
\definecolor{dgreencolor}{rgb}{0.2,0.4,0.0}
\definecolor{dgraycolor}{rgb}{0.30,0.3,0.30}

\title[Rational approximations and  multidimensional continued fractions]{Rational approximations, multidimensional continued fractions and lattice reduction}
\author{V. Berth\'e} 
\address{Universit\'e  Paris Cit\'e, CNRS, IRIF, F-75006 Paris, France}
\email{berthe@irif.fr}
\author{K. Dajani}
\address{Department of Mathematics, Utrecht University, P.O. Box 80010, 3508TA Utrecht, The Netherlands}
\email{k.dajani1@uu.nl }
\author{C. Kalle}
\address{Mathematisch Instituut, Leiden University, Niels Bohrweg 1, 2333CA Leiden, The Netherlands}
\email{kallecccj@math.leidenuniv.nl }
\author{E. Krawczyk}
\address{Faculty of Mathematics and Computer Science, Institute of Mathematics,
Jagiellonian University, Stanis\l{}awa \L{}ojasiewicza 6, 30-348 Krak\'{o}w, Poland} 
\email{ela.krawczyk7@gmail.com}
\author{H. Kuru}
\address{
Faculty of Engineering and Natural Sciences, Sabancı University,
Tuzla, Istanbul 34956 Turkey }
\email{hamidekuru@sabanciuniv.edu}
\author{A. Thevis}
\address{Institut f\"ur Mathematik, Goethe-Universit\"at Frankfurt, Robert-Mayer-Str. 6-8, 60325 Frankfurt am Main, Germany}
\email{thevis@math.uni-frankfurt.de}

\date{\today}
 
\thanks{This work was supported by the Agence Nationale de la Recherche through the project CODYS (ANR-18-CE40-0007).}


\begin{document}

\begin{abstract}
 We first  survey  the current state of the art concerning  the  dynamical properties of  multidimensional continued fraction algorithms  defined dynamically as  piecewise fractional  maps and  compare them   with  algorithms  based on 
 lattice  reduction.  We discuss  their  convergence properties and the quality of the rational approximation,  and stress the interest  for these algorithms to  be  obtained by iterating dynamical systems. 
   We then focus on  an algorithm based on the  classical  Jacobi--Perron
   algorithm involving the nearest integer part. We  describe its Markov properties and  we  suggest a possible procedure for proving the existence of a finite ergodic invariant measure absolutely continuous with respect to Lebesgue measure.
\end{abstract}
\maketitle

% We first  survey  the current state of the art concerning  convergence   and  dynamical properties of  multidimensional continued fraction algorithms  defined dynamically as  piecewise fractional  maps and  compare them   with  algorithms  based on 
%  lattice  reduction,  concerning the rational approximations  that are produced. We discuss  their  convergence properties,  and stress the interest  for these algorithms to  be  obtained by iterating dynamical systems. 
%    We then focus on  an algorithm based on the  classical  Jacobi--Perron
%    algorithm involving the nearest integer part. We  describe its Markov properties and  we  suggest a possible procedure for proving the existence of a finite ergodic invariant measure absolutely continuous with respect to Lebesgue measure.

\tableofcontents

\section{Introduction}
Continued fraction type expansions  aim  (among other  properties) at providing increasingly good rational Diophantine approximations of real numbers. More precisely, a  multidimensional continued fraction is expected to produce  simultaneous better and better rational approximations with the same denominator  $\bp^{(n)}/q^{(n)}=(p_1^{(n)}/q^{(n)},\dots p_d^{(n)}/q^{(n)})_{n  \in {\mathbb N}}$ for  $d$-tuples $\balpha=(\alpha_1,\dots,\alpha_d)$ of real numbers,  with   the fractions $p_i^{(n)}/q^{(n)}$ converging to  $\alpha_i$ for  each $1\le i \le d$. 

The usual regular continued fractions  are known to provide   extremely good (and even   the best)  rational approximations   for positive real numbers \cite{Khintchine,Cassels}. The situation is more complicated in higher dimension. Indeed, there is no canonical extension of regular continued fractions to higher dimensions (see Section \ref{subsec:noncan}), and the zoology of existing algorithms is particularly rich (see Section \ref{subsec:zoo} as an illustration). 
The main advantage of most classical (unimodular) continued fractions is that they can be expressed as dynamical systems  whose  ergodic study has  already been well  understood (such as described in  Section \ref{subsec:dyn}). 
Ergodic theory allows  a precise description of the  long-range statistical properties of the expansions  that are produced e.g.\  their mean behavior. Indeed, ergodic theory  extends   basic   laws of large numbers in probability by dropping the assumption of intertemporal independence.
Thus,
it relates spatial averages $\int_X f\,  d\mu$    to time averages $\frac{1}{n}\sum_{0\leq i< n}f\circ T^{i}$
along trajectories, and the system   has the  same behavior when  averaged over time as averaged over  the whole  space. 
  
  However, the  main disadvantage  of these algorithms relies,   firstly,  in the fact that the behavior of the continued fraction  expansion    of a given $d$-tuple $\balpha$    can be difficult to grasp (it might not behave   in a  generic  way), and secondly, in the quality of  the rational approximations  that are produced. 
Indeed, the convergence 
of  multidimensional continued fractions is  governed by their (first and second)   Lyapounov exponents (see \cite{Lagarias:93}), which  describe the asymptotic behavior of the singular values of large products of matrices, under the ergodic hypothesis. More precisely, their  approximation  exponent     can be expressed as   $1- \frac{\lambda_2}{\lambda_1}$  according to \cite{Lagarias:93} ($\lambda_1$ 
and $\lambda_2$  being the two largest Lyapunov exponents of the associated  dynamical system). It   has 
to be  compared  with   Dirichlet's exponent $1+1/d$ (see Theorem \ref{thm:Dirichlet} in Section \ref{subsec:Dirichlet}). However, there is numerical evidence \cite{berth2021second} that the second Lyapunov exponent is not even negative in higher dimensions for  most classical algorithms such as the Jacobi--Perron \cite{Bernstein:71,Perron:07,Schweiger:73}, Brun  \cite{Brun19,Brun20,BRUN} or Selmer \cite{Selmer:61} algorithms, which prevents strong convergence of these algorithms.   In a nutshell, strong  (resp.  weak) convergence refers to  the convergence of quantities of the type  $ |||q^{(n)} \balpha|||$ (resp., 
 $ \vert\vert \alpha -\bp^{(n)}/q^{(n)}\vert \vert$). These algorithms converge weakly usually, but  they  fail to have strong convergence  (see Section \ref{subsec:conv} for the definitions of weak and strong convergence).

In terms of  the quality of rational approximations that are produced, there is  a second strategy which relies on lattice reduction, where 
rational approximations are  obtained by  exhibiting  short vectors in a  lattice attached to some  given $d$-tuple $\balpha$.  Lattice reduction  algorithms   aim to   find reduced basis of  Euclidean lattices, formed by short and almost orthogonal vectors.  The  most celebrated one is the LLL algorithm, designed by  Lenstra, Lenstra and Lov\'asz  in 1982 \cite{LLL:82}. It relies heavily on the use of Gram-Schmidt  orthogonalization.  
Its overall algorithmic
structure    is simple and yet,   its general probabilistic behavior is far from being  understood; this includes  the   gap between its    practical  performances  and 
its proved   worst-case  estimates.  Hence, although  finding rational approximations  work quite well in practice, the average behavior  of  such a strategy 
is not well understood. In particular, the lack of a description of reduction algorithms as dynamical systems prevents the use of tools from ergodic theory.

In the expository part of the  paper, we focus  on the  two  main   classes   of algorithms that produce rational approximations as discussed above.  
 The first type of algorithms can be expressed via dynamical systems  defined  on    a compact set (usually  of the form $[0,1]^d$).  Such  an algorithm 
associates with some  given vector  an infinite sequence of matrices, and one can consider the quality of convergence  of this product of matrices. The most classical examples of such algorithms are the  Jacobi--Perron \cite{Bernstein:71,Heine1868,Perron:07,Schweiger:73},  the Brun  \cite{Brun19,Brun20,BRUN}, or   the  Selmer algorithms (the last of which is  conjugate on the absorbing simplex to M\"onkemeyer’s algorithm  \cite{Mon:54, Panti:08}).  They are described e.g. in \cite{BRENTJES,SCHWEIGER,Lagarias:93}.
 The second type of algorithm is based on  lattice reduction algorithms, such as  the LLL algorithm (see Section \ref{sec:LLL}).  We focus  on these two families since they share as a common feature the fact  that they rely on a basis of the integer  lattice ${\mathbb Z}^{d}$.
 
 We then illustrate  in Section \ref{sec:ni} the  dynamical approach with the ergodic study of a version of the Jacobi--Perron algorithm based on the use  of the nearest integer part. One motivation for studying the nearest integer Jacobi--Perron  algorithm is to confirm the  idea  that  working with the nearest integer part improves the  quality
 of  continued fraction  algorithms, such as   indicated numerically by the  experimental results from  Steiner \cite{Steiner:pc} (see Section \ref{subsec:expnijpa}). The  partial quotients produced by the  usual Jacobi--Perron
algorithm  satisfy a simple Markovian rule. In the case of the nearest integer Jacobi--Perron algorithm the description of  the admissible sequences of digits is  much more involved. Hence, a simple modification -- such as changing the choice of the integer part -- leads to much more delicate
conditions for the description of the algorithm.
As a first step toward a theoretical confirmation of   the above-mentioned estimates, we prove the existence  of a Markov partition for the nearest integer Jacobi--Perron algorithm and suggest a possible procedure for proving the existence of a finite ergodic  invariant measure absolutely continuous with respect to Lebesgue measure.

% We then illustrate  in Section \ref{sec:ni} the  dynamical approach with the   ergodic study  of a version of the Jacobi--Perron algorithm  based on the use  of the nearest integer part.  One motivation  for  studying the nearest integer Jacobi--Perron  algorithm is to confirm the  idea  that  working with the nearest integer part improves the  quality
%  of  continued fraction  algorithms, such as   indicated numerically by the  experimental results from W. Steiner (see Section \ref{subsec:nijpa}). The  partial quotients produced by the  usual Jacobi--Perron
% algorithm  satisfy a simple Markovian rule. We see here that the description of  the admissible sequences of digits is  much more involved. Hence a simple modification  such as  changing the  choice of   the  integer part  leads to much more delicate
% conditions for the description of the algorithm.
% As a first step toward a theoretical confirmation of   the above-mentioned estimates, we prove the existence  of a finite ergodic  invariant measure with respect to Lebesgue measure.

\bigskip

Let us sketch the contents of this paper.    Section \ref{sec:rcf} recalls basic 
notions concerning classical continued fractions. We present their main properties  that we will use as a guideline for  possible generalizations to higher-dimensional case. We  then focus in Section \ref{sec:sa} on the two main strategies that
can be used for  producing rational approximations in an effective way. Section \ref{sec:cf} focuses on the classical dynamical  unimodular continued fraction algorithms. Algorithms based on the
lattice reduction algorithms  and homogeneous dynamics are considered  in Section \ref{sec:LLL}. 
We briefly  discuss  applications and   possible ways to improve   algorithms in Section  \ref{sec:applications}. Lastly, in Section \ref{sec:ni} we focus on the nearest integer Jacobi--Perron algorithm. We describe its associated Markov partition and provide a strategy for proving the existence of an absolutely continuous invariant measure.

% Let us sketch the contents of this paper.    Section \ref{sec:rcf} recalls basic 
% notions on classical continued fractions, by presenting  the main properties  that will be used as a guideline for  possible generalizations. We  then focus in Section \ref{sec:sa} on the two mains strategies that
% can be used for  producing in an effective way rational approximations. Section \ref{sec:cf} focuses on the classical dynamical  unimodular continued fraction algorithms. Algorithms based on 
% lattice reduction algorithms  and homogeneous dynamics are considered  in Section \ref{sec:LLL}. 
% We briefly  discuss  applications and   possible ways to improve   algorithms in Section  \ref{sec:applications}. Lastly we focus in Section \ref{sec:ni} on the nearest integer Jacobi--Perron algorithm. We describe its associated   Markov partition  and prove the existence of an absolutely continuous invariant measure.

\bigskip

{\bf Acknowledgements} We are  greatly  indebted to Wolfgang Steiner for  
the numerical data he provided us in Section \ref{subsec:expnijpa}. We would like to thank Women in Numbers Europe-4 (WINE4) for bringing us together and for giving us the opportunity to work on this project.

\section{Continued fractions} \label{sec:rcf}
 In this section we briefly recall   the main properties of the usual regular continued fractions.
 They will serve us a  guideline   for  the discussion on 
  the higher-dimensional case.
For general references on continued fractions, see e.g.  \cite{Bill:78,DaKr02,HW,Kraai02,Khintchine}.
For any positive  real number $\alpha \in [0,1]$,  its 
continued fraction expansion  is 

$$\alpha=
\cfrac{1}{a_1+
  \cfrac{1}{a_2+
    \cfrac{1}{a_3+\ddots}
}},$$
where  the   digits  $a_n$  are positive integers, called {\em partial quotients}. 
The rational  numbers  $p_n/q_n$, where  $p_n$, $q_n$  are coprime positive integers  defined as
    \[ \frac{p_n}{q_n}= \cfrac{1}{a_1 +\cfrac{1}{a_2 + \ddots + \cfrac1{a_n}
    }} ,\]
    are called   {\em convergents}. 
    %Let $$\theta_n= (-1)^n (q_n \alpha- p_n)=|q_n \alpha-p_n|$$ for all nonnegative  $ n$.
The  sequence of rational numbers $p_n /q_n $ approximates $\alpha$  up to an
error of order  $1/q_n^2 $:  one  has $$|\alpha - p_n/q_n| \leq  \frac{1}{q_n^2} \, \,  \mbox{ for all } n.$$

Dynamically, continued fraction expansions   can be   obtained  by applying  the Gauss map  $T _G\ :  [0,1] \rightarrow [0,1]$ defined by
$$ T_G(0)=0\quad \text{and}\quad
T_G(\alpha)= \{1/\alpha\}  \mbox{ if } \alpha \neq 0,$$ 
where $\{ \cdot \}$ is the fractional part of a real number.  If for an $\alpha \in (0,1]$ we write
$T_G(\alpha)= \{1/\alpha\}=\frac{1}{\alpha} -\lfloor \frac{1}{\alpha} \rfloor= \frac{1}{\alpha} - a_1$, then 
$\alpha= \frac{1}{a_1+ T_G(\alpha)}$. Now, by  setting
$a_n=\lfloor  \frac{1}{T^{n-1}_G(\alpha)} \rfloor$  for $ n \geq 1$,
one gets   the digits  in the the continued fraction expansion of $\alpha$.

% Dynamically, continued fraction expansions   can be   obtained  by applying  the Gauss map  
% $$T _G\ :  [0,1] \rightarrow [0,1], \ 
% \alpha \mapsto \{1/\alpha\}  \mbox{ if } \alpha \neq 0, \ T_G(0)=0.$$ 
% Observe that the fractional part $\{ \, \}$ is defined in terms of the  usual integer part, denoted as $\lfloor \, \rfloor$, whereas  the nearest integer Jacobi--Perron algorithm discussed in Section \ref{sec:ni}  is  based on the distance to the nearest integer. If 
% $\alpha_1= T_G(\alpha)= \{1/\alpha\}=\frac{1}{\alpha} -\lfloor \frac{1}{\alpha} \rfloor= \frac{1}{\alpha} - a_1$, then 
% $\alpha= \frac{1}{a_1+ \alpha_1}$. Now, by  setting
% $a_n=\lfloor  \frac{1}{T^{n-1}\alpha} \rfloor$  for $ n \geq 1$,
% one gets   the digits  in the the continued fraction expansion of $\alpha$. 

%One has $$x=\cfrac{1}{a_1+\cfrac{1}{a_2+\cfrac{1}{a_3+\cfrac{1}{a_4+\cdots}}}}.$$
%We also will use the notation $x=[0 ; a_1,\cdots,a_n, \cdots]$.

Note that the Gauss map is closely related to {Euclid's algorithm}:
starting with two (coprime)  positive integers   $\ell^{(0)}$ and $\ell^{(1)}$
Euclid's  algorithm   works by  subtracting  as often as  possible
 the smallest  of both numbers
  from  the largest one (that is, one performs   the     Euclidean division of  the largest one  by the smallest);  this yields
$\ell^{(0)}=\ell^{(1)}\lfloor  \frac{\ell^{(0)}}{\ell^{(1)}} \rfloor+\ell^{(2)}$, 
$\ell^{(1)}=\ell^{(2)}\lfloor \frac{\ell^{(1)}}{\ell^{(2)}}\rfloor+\ell^{(3)}$, {etc.},  until we reach 
$\ell^{(m+1)}=1=\mbox{gcd} (\ell^{(0)},\ell^{(1)})$.
By setting,  for $n \in \mathbb{N}$, $\alpha^{(n)}=\frac{\ell^{(n)}}{\ell^{(n+1)}}$  and  $a_n=\lfloor \alpha^{(n)} \rfloor$, one gets 
$\alpha^{(n-1)}=a_{n-1} + \frac{1}{\alpha^{(n)}}, $
and 
$$
\alpha^{(0)} =  \frac{\ell^{(0)}}{\ell^{(1)}}=a_0 +
\cfrac{1}{a_1+
  \cfrac{1}{a_2+
    \cfrac{1}{a_3+\ddots+\frac{1}{a_{m-1}+\frac{1}{a_m}}
}}}.$$

  Let us now revisit the action of the Gauss map  in  matricial terms.  Let $\alpha \in [0,1].$
For all $n$, we have
\begin{align*}
\left[\begin{array}{l}
\alpha\\
1
\end{array}\right] &=\alpha T_G(\alpha) \cdots T^{n-1}_G(\alpha)  \left [\begin{array}{ll}
0 & 1\\
1 & a_1 
\end{array}\right] \cdots \left [\begin{array}{ll}
0 & 1\\
1 & a_n 
\end{array}\right]   \begin{bmatrix}
T^{n}_G(\alpha)\\
1
\end{bmatrix} \\
& =  \alpha T_G(\alpha) \cdots T^{n-1}_G(\alpha)  \left [\begin{array}{ll}
p_{n-1} & p_{n}\\
q_{n-1}& q_{ n}
\end{array}\right]  \begin{bmatrix}
T^{n}_G(\alpha)\\
1
\end{bmatrix},
\end{align*}
by   using the classical relations between convergents and partial quotients, namely 
 $q_{-1}=0$, $p_{-1}=1$, $q_0=1$, $p_0=0$, and, for all $n$,
$$
q_{n +1}=a_{n +1} q_n +q_{n-1}, \  p_{n +1}=a_{n +1} p_n +p_{n-1}.$$
The matrix $\begin{bmatrix}
   p_{n-1} & p_n \\
   q_{n-1} & q_n
\end{bmatrix}$ 
is a square  matrix  with  integer entries  that has   determinant   of absolute value $1$, which is to say that it is {\em unimodular}. We denote the set of unimodular matrices by $\mathrm{GL}(d,{\mathbb Z})$. (We also use  the following standard notations: $\mathrm{GL}(d,{\mathbb R})$ stands for the  set of $d\times d$ invertible matrices with real entries, 
 $SL(d,\GN)$  stands for 
the set of $d\times d$ matrices of determinant $1$
with non-negative  integer coefficients.) Note that the entries of this matrix   are even  positive. To understand the convergence of such a sequence of matrices one can use  generalizations of the Perron--Frobenius theorem,  such as  Theorem \ref{thm:furs} below.

% It is said {\em unimodular} and  the set of unimodular matrices
% of determinant  $\pm1$  with values in ${\mathbb Z}$ is denoted as $\mathrm{GL}(d,{\mathbb Z})$. (We  use   in  the sequel   the following usual notation for square matrices of  size$f$: $\mathrm{GL}(d,{\mathbb R})$ stands for the  set of invertible matrices with real entries, 
%  $SL(d,\GN)$  stands for 
% the set of matrices of determinant $1$
% with non-negative   integer coefficients.) Note that the entries of this matrix   are even  positive. In terms  of convergence, one then   can use  generalizations of the Perron--Frobenius theorem,  such as  Theorem \ref{thm:furs} below.

Dynamically, the  Gauss map $T_G$  goes with  the map 
\begin{equation} \label{eq:gcocyle}
A_G : [0,1] \rightarrow  \mathrm{GL}(2, {\mathbb N}),  \ 
\alpha \mapsto \begin{bmatrix}
 0&  1\\
 1  & \lfloor 1/\alpha \rfloor
 \end{bmatrix}.
 \end{equation}
 Such a  map is called  a \emph{cocycle} in  the  terminology of (random) dynamical systems  (see e.g. \cite{Arnold95,Arnold98,Viana:book}).
 Let us set $A^{(n)} :=A(T_G ^{n}) (\alpha)$ for all positive  $n$.
The line directed by  the vector $(\alpha,1)$
in ${\mathbb R}^2$  belongs to the    sequence of nested cones
$A^{(1)} \cdots  A^{(n)} {\mathbb R}^2_+,$  i.e., 
$$(\alpha,1) \in \bigcap _n A^{(1)} \cdots  A^{(n)} {\mathbb R}^2_+= \bigcap _n  A(T_G)( \alpha)  \cdots A(T_G^n)(  \alpha)  {\mathbb R}^2_+ . $$
One has even more: this  sequence of nested cones converges to the line  directed by $(\alpha,1)$. 
The {\em convergence} is said to be {\em weak} if    the angles  between the (column) vectors  of the product matrices $A^{(1)} \cdots A^{(n)}$  tend to $0$ (as $n\to\infty$),  and {\em strong} if the distances between the  vectors tend to $0$.
Continued fractions can thus be seen as  producing dynamically  infinite convergent  sequences of unimodular  matrices. 
We revisit the notions of convergence of  infinite products of matrices  in Section \ref{subsec:conv} in more detail.

% Dynamically, the  Gauss map $T_G$  goes with  the map 
% \begin{equation} \label{eq:gcocyle}
% A_G : [0,1] \rightarrow  \mathrm{GL}(2, {\mathbb N}),  \ 
% \alpha \mapsto \begin{bmatrix}
%  0&  1\\
%  1  & \lfloor 1/x\rfloor
%  \end{bmatrix}.
%  \end{equation}
%  Such a  map is called  a \emph{cocycle} in  the  terminology of (random) dynamical systems  (see e.g. \cite{Arnold95,Arnold98,Viana:book}).
%  Let us set $A^{(n)} :=A(T_G ^{n}) (\alpha)$ for all positive  $n$.
% The line directed by  the vector $(\alpha,1)$
% in ${\mathbb R}^2$  belongs to the    sequence of nested cones
% $A^{(1)} \cdots  A^{(n)} {\mathbb R}^2_+,$  i.e., 
% $$(\alpha,1) \in \bigcap _n A^{(1)} \cdots  A^{(n)} {\mathbb R}^2_+= \bigcap _n  A(T_G \alpha)  \cdots A(T_G^n  \alpha)  {\mathbb R}^2_+ . $$
% One has even more: this  sequence of nested cones converges to the line  directed by $(\alpha,1)$. 
% The {\em convergence} is then said to be {\em weak} if    the angles  between the (column) vectors  of the product matrix $A^{(1)} \cdots A^{(n)}$  tends to $0$,  and {\em strong} if the distances between the  vectors tend to $0$.
% Continued fractions can thus be seen as  producing dynamically  infinite convergent  sequences of unimodular  matrices. 
% We revisit the notion of convergence of  infinite products of matrices  in Section \ref{subsec:conv} in more details.

\section{On  simultaneous  approximation}\label{sec:sa}
This section is  devoted to simultaneous  rational approximations. In Section \ref{subsec:Dirichlet} we first recall Dirichlet's theorem. We then  describe  in Section \ref{subsec:approx} the two main strategies  for producing such approximations that we consider in this survey.  We focus on   the quality  of approximations in Section \ref{subsec:conv}.

% This section is  devoted to simultaneous  rational approximations. We first recall Dirichlet's theorem in Section \ref{subsec:Dirichlet}. We then  describe the two main strategies under consideration in this survey   for producing such approximations in Section \ref{subsec:approx}.  We focus on   the quality 
% of approximations in Section \ref{subsec:conv}.

\subsection{Dirichlet's bound}\label{subsec:Dirichlet}
Let us recall Dirichlet's theorem; it can be 
obtained   as a direct application of the pigeonhole principle (see  e.g.  \cite{HW})  or of  Minkowski's first  theorem. 
 \begin{theorem}[Dirichlet's theorem]\label{thm:Dirichlet}
 For any $(\alpha_1, \cdots,\alpha_d)\in \mathbb{R}^d$ and any $Q$, there  exists a positive  integer $q$ with 
$q \leq Q^d$  and integers $p_i$ such that
$$\max_{ 1\leq i \leq d}|q \alpha_i - p_i| < \frac{1} {Q}.$$ 
\end{theorem} 
Theorem \ref{thm:Dirichlet}  immediately implies that for any   $(\alpha_1, \cdots,\alpha_d)\in \mathbb{R}^d$  the system of inequalities 
$$\left|\frac{p_i}{q}-\alpha_i \right|< \frac{1}{q^{1+\frac{1}{d}}},\ \mbox{ for }
i=1,2,\ldots ,d,$$
admits infinitely many  integer solutions $(p_1,\dots, p_d, q)$.

% Let us recall Dirichlet's theorem,  
% obtained   as a direct application of the pigeonhole principle (see  e.g.  \cite{HW})  or of  Minkowski's first  theorem. 
%  \begin{theorem}[Dirichlet's theorem]\label{thm:Dirichlet}
%  For any $(\alpha_1, \cdots,\alpha_d)\in \mathbb{R}^d$ and any $Q$, there  exists a positive  integer $q$ with 
% $q \leq Q^d$  and integers $p_i$ such that
% $$\max_{ 1\leq i \leq d}|q \alpha_i - p_i| < \frac{1} {Q}.$$ 
% \end{theorem} 
% One thus deduces  immediately  that the system of inequalities 
% $$\left|\frac{p_i}{q}-\alpha_i \right|< \frac{1}{q^{1+\frac{1}{d}}},\ \mbox{ for }
% i=1,2,\ldots ,d$$
% admits infinitely many  integer solutions.

The exponent $1+1/d$ is optimal as shown in \cite{Perron:1921}, see also \cite{Cassels,Schmidt:96}.
The Dirichlet's theorem    provides   the existence of ``good'' approximations.    One can thus  make an exhaustive search, though  this is not
an  efficient algorithmic method.  See \cite[Chapter 6]{LLLbook}  for a discussion on effective methods. See also   \cite{Lagarias:85}, where the computational complexity  concerning   simultaneous  Diophantine approximations is investigated. When the dimension  $d$ is fixed, \cite{Lagarias:85} gives algorithms  which,  for a  given $N$, find a  good  rational approximation with denominator $ 1 \leq q \leq N$  with respect to a specified accuracy, 
   or  which find all best approximations  in $[1,\cdots, N]$  in   polynomial-time (using methods  based on 
  the LLL algorithm).   
  Note that   the  following problem is proved to be NP-hard:
for a given  vector $\alpha \in \mathbb{Q}^d$,  positive   integer $N$ and   accuracy  $s_1/s_2$,
  is there an integer $Q$   with $1 \leq Q \leq N$ such that
  $|||Q \alpha |||  \leq s_1/s_2$?  (The distance to the nearest integer is expressed here with respect to the supremum norm.)  In a similar flavor, see also \cite{HJLS:89}   concerning   the problem of finding integer  relations, and
 \cite{BJM:88}.

% This exponent is optimal as shown in \cite{Perron:1921}, see also \cite{Cassels,Schmidt:96}.
% The proof  of Dirichlet's theorem           provides   the existence of ``good'' approximations.     We thus  can   make an exhaustive search, though  this is not
% an  efficient algorithmic method.  See \cite[Chapter 6]{LLLbook}  for a discussion on effective methods. See also   \cite{Lagarias:85} which 
% investigates  the  computational complexity  concerning   simultaneous  Diophantine approximation.
%    When the dimension  $d$ is fixed,   algorithms  are  given  which find a  good  approximation  $q$ with $ 1 \leq q \leq N$  for a  given $N$   with respect to a specified accuracy, 
%    or  which find    all best approximations  in $[1,\cdots, N]$  in   polynomial-time, with methods  based on 
%   the LLL algorithm.   
%   Note that   the  following problem of  decision is proved to be NP-hard:
%  we are  given a   vector $\alpha \in \mathbb{Q}^d$,  a positive   integer $N$ and  an accuracy  $s_1/s_2$;
%   is there an integer $Q$   with $1 \leq Q \leq N$ such that
%   $|||Q \alpha |||  \leq s_1/s_2$?  (the distance to the nearest integer is expressed here with respect to the sup norm).    See  also in the same flavor  \cite{HJLS:89}   concerning   the problem of finding integer  relations, and
%  \cite{BJM:88}.

Continued fractions   are known to   provide   good (and even the best) rational  approximations of  a   given real number $\alpha$  in $[0,1]$ (see e.g. \cite{Cassels,Khintchine}).
One would desire to have similar algorithms  yielding  good rational approximations  with the same denominator of   $d$-tuples of   positive real numbers. That is, for a given  $\balpha=(\alpha_1,\cdots, \alpha_d) \in [0,1]^d$,   one   looks  for  sequences of positive 
  integers $(q_n)_n$ and positive integer $d$-tuples $\bp^{(n)}= (p_1^{(n)}, \cdots, p_d^{(n)})_n$ such that 
  $$\lim p_i^{(n)}/q_n= \alpha_i,\quad   i=1, \cdots,d$$  with  a good  quality  of rational  approximation of
$\balpha$.  
Geometrically,  this corresponds to    looking  for     approximations of  a line in   $\mathbb{R}^{d+1}$
    by  points in  $\mathbb{Z}^{d+1}$. Dual  problems consist  of    looking for  small values of linear forms and  small linear relations.

%      The question is now to find similar algorithms simultaneous rational 
% approximations  with same denominator, and  of good quality,    Given  $\balpha=(\alpha_1,\cdots, \alpha_d) \in [0,1]^d$,   one then  looks  for  sequences of positive 
%   integers $\bp^{(n)}= (p_1^{(n)}, \cdots, p_d^{(n)})_n$ such that 
%   $$\lim p_i^{(n)}/q_n= \alpha_i,   i=1, \cdots,d$$  with  a good  quality  of rational  approximation of
% $\balpha$.  
% Geometrically,  this corresponds to    looking  for     approximations of  a line in   $\mathbb{R}^{d+1}$
%     by  points in  $\mathbb{Z}^{d+1}$. Dual  problems consist  in    looking for  small values of linear forms and  small linear relations.
%   Let us   try to  describe what is meant by  good quality here.
More precisely, given a norm $||\cdot ||$ on $\GR^{d}$, let  $|||\cdot  |||$ stand for the distance to the nearest integer. The usual norms that are considered  are
the supremum and the Euclidean norm. The quality of the
approximation 
is measured by
$
\frac{1}{q^{(n)}}||| q^{(n)}\balpha|||$, to be compared with  Dirichlet's bound, i.e.,
$|||q^{(n)} \balpha|||$   has to be compared  with  $(q ^{(n)})^{-1/d}$.  One can thus consider  the approximation exponent 
$$ \limsup_n - \frac{\log \Vert  \balpha-\frac{\mathbf{p}_i^{(n)}}{q_i^{(n)}} \Vert}{\log q^{(n)}}$$
and compare it  to the Dirchlet's bound $1+1/d$.

\subsection{How to produce  rational approximations}\label{subsec:approx}
We focus here on two main  approaches for  producing rational approximations. The first one is based on the  generation  
of infinite   convergent sequences  of  matrices obtained dynamically by iteration of a map acting  on a compact space, as illustrated with the Gauss map $T_G$ for usual continued fractions in Section \ref{sec:rcf}. 
This   will be  discussed further in Section \ref{sec:cf}. The second one is based on  the existence of small vectors picked in well chosen lattices; we will discuss it in Section \ref{sec:LLL}.

The first  strategy associates with  some  element $\balpha=(\alpha_1, \cdots, \alpha_d) \in \GR ^d$ 
   a sequence of  square matrices $(A^{(n)})_{ n \in \mathbb{N}}$  of size $d+1$  with integer entries.
   It can be produced for instance via a dynamical system   $(X,T)$  with  a map $A$ as follows (see also \ref{eq:gcocyle}):
 \begin{equation} \label{eq:MA} T \colon X \rightarrow X, \  A \colon  X \rightarrow   GL(d+1,{\mathbb Z}),  \mbox{ and }  A^{(n)}= A(T^n(x)).
 \end{equation}
  If the matrices belong to  $GL(d+1,\GZ)$, then  the  corresponding algorithm is called {\em unimodular}.
   Matrices   $A^{(n)}$ play the role of   partial quotients and the   product  matrices  $A^{(1)} \cdots A^{(n)} $ produce convergents.
Convergents   aim at    providing  Diophantine approximations (via their column vectors)
of   the direction $(\balpha,1)$. 
 We write
  \begin{equation} \label{eq:MAT}
  A^{(1)} \cdots A^{(n)} = 
 \left[\begin{matrix}
 p_{1,1}^{(n)} & \cdots   & p_{1,d+1}^{(n)}  \\
\vdots &\ddots & \vdots\\
 p_{d,1}^{(n)} & \cdots  & p_{d,d+1}^{(n)} \\
 q_{1}^{(n)}  & \cdots  & q_{d+1} ^{(n)}
 \end{matrix}
 \right].
 \end{equation}

The last element of each column  of $A^{(1)} \cdots A^{(n)} $   is  a   denominator for the associated  simultaneous rational approximation. The rows of the convergent matrices are meant to provide  the numerators  of the simultaneous   approximations, i.e.,  one  considers
$$\left(\frac{p_{j,1} ^{(n)}}{q_j ^{(n)}}, \cdots , \frac{p_{j,d} ^{(n)}}{q_j ^{(n)}}\right).$$
The  integers $q_j ^{(n)}$  play the role of  $n$th convergents, and  the 
vector $( p_{j,1} ^{(n)},\cdots, p_{j,d} ^{(n)}, q_{j} ^{(n)})$ is    called an  $n$th convergent vector.
The convergence of these matrices means  that they contract
in the direction of the vector $(\balpha,1)$. We discuss more precisely their convergence in Section \ref{subsec:conv} by 
 considering  products  $ A^{(1)} \cdots A^{(n)}$ as $n$ goes to infinity.

 We now describe  the second approach  based on  the existence of small  vectors in  well chosen lattices, as described in  the 
 seminal paper \cite{LLL:82}. This approach yields a  very fruitful compromise between the quality  of  approximation  (a good approximation is deduced from 
 a small vector) and the  efficiency  (this small vector is obtained in polynomial time). Let $\balpha=(\alpha_1,\cdots, \alpha_d) \in \mathbb{R}^d$ be a vector to approximate.  One  works  here again 
in a $d+1$-dimensional space, by introducing  a one-parameter family of lattices
 $(\Lambda_t)_{t >0}$  with positive   parameter $t$  tending  to $0$. More precisely, let $\Lambda_t$ be the lattice generated by the columns of the  matrix
$$M_t:=\left[ \begin{matrix}
  1&0& \cdots  & 0 & -\alpha_1\\
  0&1&\cdots &0& -\alpha_2\\
    \vdots&\vdots&\ddots&\vdots&\vdots\\
  0&0&\cdots&1&-\alpha_d\\
  0 & 0 &\cdots&0&  {t}
  \end{matrix} \right].$$
Note that   $\det (M_t)=t$, hence, 
 the   lattice   $\Lambda_t$  changes  at each step of the  algorithm. 
 Let us stress the fact that this strategy differs  from the previous one    where, in the unimodular case,   one works with    bases of the fixed   lattice
 $\mathbb{Z}^{d+1}$.
  We  will  take $t$ small, the  parameter $Q$ of   Dirichlet's theorem  being  connected to $t$ as  follows:
  $Q=t^{- \frac{1}{d+1}}$.

 One of the main features of the  LLL  algorithm  is that it produces  in    polynomial time  a  non-zero vector   ${\bf b}=(b_1, \cdots,  b_{d+1})$
  of the lattice $\Lambda_t$ 
 such that  
 \begin{equation}\label{eq:small}
  || {\bf b}||_2 \leq 2^{d/4} \det(M_t)^{1/{(d+1)}} =  2^{d/4} t^{1/{(d+1)}}.
 \end{equation}
 Note that  the geometry of numbers, and more precisely  Minkowski's first theorem,   
 guarantees the  existence of a  ``small'' non-zero   vector  ${\bf x}\in\Lambda_t$, i.e.,  such that 
 \begin{equation} \label{eq:min}
 ||{\bf x}|| \leq  \sqrt {(d+1)(d+5)/4}  \  ( \mbox{vol} (\Lambda_t))^{1/(d+1)}=  \sqrt {(d+1)(d+5)/4} \   t^{1/{(d+1)}}.
 \end{equation}

Let $({\bf e}_i)_{i=1,\ldots,d+1}$  stand for the canonical   basis of $\mathbb{Z}^{d+1}$.
There exist  integers $p_1,\ldots,p_d,q $    such that 
 $$\begin{array}{ll}
 {\bf b} &= p_1 {\bf e}_1 + p_2 {\bf e}_2 + \dots + p_d {\bf e}_d+q(-\alpha_1 {\bf e}_1 - \dots -\alpha_d {\bf e}_d + t {\bf e}_{d+1})\\
    &=
    (p_1 - q\alpha_1){\bf e}_1+ \cdots + (p_d - q\alpha_d) {\bf e}_d+ qt {\bf e}_{d+1}.
    \end{array}$$

One deduces  from (\ref{eq:small}) that for all $1\le i\le d$
$$    \quad |p_i - \alpha_i q| \leq 2^{d/4} t^{1/{(d+1)}}$$  
and 
$$qt \leq 2^{d/4}  t^{1/{(d+1)}}, \mbox{ i.e., }  \  t ^{\frac{1}{d+1}}\leq  \frac{2 ^{1/4}}{ q ^{1/d}}. $$

For all $i$, we deduce that
%One sets $Q =t^{1/{d+1}}$ and  gets $|p_i - \alpha_i q| \leq 2^{d/4} 1/Q$
$$|p_i - \alpha_i q| \leq   \frac{2^{  {(d+1)/4} }}{  q ^{  {1/d}}} ,$$
with  $$|q| \leq 2^{d/4} t^{-d/(d+1)}=2^{d/4}Q^d.$$

The quality of approximation is   the  quality that is   expected (according to Dirichlet's theorem)   up    to  a   multiplicative  factor   
$2^{  {(d+1)/4} }$  which depends  exponentially     on the  dimension. We could   have used the inequality (\ref{eq:min}) which would have given a different multiplicative  factor, but the same quality  ($q^{1/d}$). Nevertheless, the interest of a lattice reduction algorithm such as the LLL algorithm
is that  the small vector that is used is found in polynomial time.

For a  given $t>0$, the smallest vector of the lattice $\Lambda_t$ produces  the  best approximation.  One can ask whether it is possible to devise a  continued fraction algorithm from  this. Note that one   has     to   recompute  everything  from the beginning when one changes $t$. This will be discussed further in Section \ref{sec:LLL}.
Algorithms defined dynamically   with  maps such in (\ref{eq:MA})
are on the contrary called memory-less (see Section \ref{sec:cf}).

% For a  given $t>0$, the smallest vector of the lattice $\Lambda_t$ produces  the  best approximation.  NOte tha there is  no reason    a priori
% that we get all the best approximations.
% The question now  is to be able to devise  a   continued fraction algorithm from  this.
% One   has     to   recompute  everything  from the beginning when one changes $t$. This will be discussed further in Section \ref{sec:LLL}.
% Algorithms defined dynamically   with  maps such in (\ref{eq:MA})
% are on the contrary called memory-less (see Section \ref{sec:cf}).

%\subsection{Products of matrices}\label{subsec:product}

%Define elementary matrices.

%We focus on  algorithms that  produces    sequences of  matrices.

\subsection{Convergence and Lyapunov exponents}\label{subsec:conv}

%The dynamical algorithms have weak convergence (convergence of angle), but there is %numerical evidence that the algorithms fail to have strong convergence (convergence %of distance). Strong convergence can be measured by the second Lyapunov exponent.

One important feature of the first strategy (based on the generation  of infinite products of matrices)  from  Section \ref{subsec:approx} is that we are still  able to measure the quality of approximation that is produced in view of Dirichlet's theorem. To do this, we need to measure  the quality of  convergence of infinite products of matrices; this can be done  using Lyapunov exponents. One  can either measure the convergence of a given product of matrices or consider the  generic behavior of products of matrices generated by a dynamical system. Lyapunov exponents  allow,   among other things, the description of  the growth of the logarithm of the angles between  column vectors  of products of matrices $M_0 \cdots M_n$.

% One important feature of the   first strategy (based on the generation  of infinite products of matrices)  from  Section \ref{subsec:approx}    is to be able to measure the quality of approximation that is produced in view of Dirichlet's theorem, i.e., 
% to measure   the quality of  convergence of infinite products of matrices. This can be done  in terms of Lyapunov exponents: 
% one  can either measure the convergence of a given product of matrices or  consider the  generic behaviour of products of matrices generated by a dynamical system. Lyapunov exponents then allow,   among other things, the description of      the growth of the logarithm of the angles between  column vectors  of products of matrices $M_0 \cdots M_n$.

We state now the  definitions concerning convergence. The norm $\Vert \cdot \Vert$  refers to the Euclidean norm and the distance $\mathrm{d}$ below refers to the  usual associated distance of a point to a line. Let  $M=(M_n)_{n\in{\mathbb N}}$ be a sequence of  square  matrices of size $d$ and let $\bl\in {\mathbb R}^{d}$.  Let  $(\be_1, \cdots, \be_{d})$ stand for the canonical  basis of ${\mathbb R}^{d}$. We say that $M$ is \emph{weakly convergent } to $\bl$ if
%\begin{equation}\label{eq:weakconv}
\[\lim_{n\to+\infty} \mathrm{d}\left(\frac{M_0 \cdots M_{n-1} \be_i}{\lVert M_{0}\cdots M_{n-1} \be_i\rVert},\bl\right) = 0\quad \hbox{for all }i\in \{1,\ldots,d\}.
\]
We say that $M$ is \emph{strongly convergent } to $\bl$ if
%\begin{equation}\label{eq:strongconv}
\[\lim_{n\to+\infty} \mathrm{d}(M_{0} \cdots M_{n-1} \be_i,{\mathbb R} \bl) = 0\quad \hbox{for all }i\in \{1,\ldots,d\}.\]
Lastly, we say that $M$ is \emph{exponentially convergent } to $\bl$ 
if there exist $C,\gamma>0$ such that 
$$
\mathrm{d}(M_{0} \cdots M_{n-1}\be_i,{\mathbb R}\bl) < C e^{-\gamma} \quad \hbox{for all }n\in {\mathbb N}  \hbox{ and for all }i\in \{1,\ldots,d\}.
$$

Exponential convergence is a first step in the  direction  of good rational approximations. The constant $\gamma$ then has to be compared with  Dirichlet's bound.

Theorem \ref{thm:furs} below gives  a sufficient condition for the   sequence of cones $M_0 \cdots M_n  {\mathbb R}_+^d$ to nest down to a single  line  as $n$ tends to infinity  for  square matrices $M_i$  with non-negative entries.  It can be seen as a generalization of  the classical Perron--Frobenius theorem.  This  statement is  particularly useful in the present   context since   multidimenensional continued fraction algorithms often  generate
 non-negative matrices.   Hence, weak convergence is usually not an issue  for  multidimensional continued fraction algorithms. 
\begin{theorem}[{\cite[pp.~91--95]{Furstenberg:60}}]\label{thm:furs}
 Let $(M_n)_n$ be a sequence of  non-negative integer matrices  of  size $d$. Assume that there exist a strictly positive matrix~$B$ and indices $j_1 < k_1 \le j_2 < k_2 \le \cdots$ such that $B = M_{j_1} \cdots M_{k_1-1}= M_{j_2} \cdots M_{k_2-1} = \cdots$. Then,
\[
\bigcap_{n\in\mathbb{N}} M_0 \cdots M_{n-1} {\mathbb R}^d_+ = {\mathbb R}_+ \bl \quad \mbox{for some positive vector $\bl \in {\mathbb R}_+^d$.}
\]
\end{theorem}

% The following condition is  a sufficient condition for the  weak  sequence of cones $M_0 \cdots M_n  {\mathbb R}_+^d$ to nest down to a single  line  as $n$ tends to infinity  for  square matrices  with non-negative entries.  It can be considered as a generalization of  the classical Perron--Frobenius theorem.  This  statement is  particularly useful in the present   context since   multidimenensional continued fraction algorithms often  generate
%  non-negative matrices. 

Let us  now focus on strong convergence by  first  recalling  the definition of  the Lyapunov exponents of a   sequence of matrices. For a  matrix $M$ in $\mathrm{GL}(d,\R)$, the singular values $\delta_1,\ldots,\delta_d$ are the eigenvalues of the matrix $\left( \tra{M}M\right)^{1/2}$. Let us order  these  (positive and real)  values  as  $\delta_1 \ge \delta_2\ge \cdots \ge \delta_d$.
Given  a  sequence $\bM=(M_n)_{n\in {\mathbb N}}$ of matrices in $\mathrm{GL(d,\R)}$,   the \emph{$i$-th Lyapunov} exponent $\theta_i$ is  then  defined as the limit $$\theta_i:=\lim_{n\to\infty}\frac 1n\log(\delta_i(n)),$$ if this limit exists, with  $\delta_i(n)$ being  the $i$-th singular value of $M_0 \cdots M_{n-1}$. 
The Lyapunov exponents can also be defined recursively using exterior powers (see for example \cite[Proposition 3.2.7]{Arnold98})
by
\[  
\theta_1 + \cdots +\theta_k = \lim_{n\to\infty} \frac{1}{n}\log \lVert\wedge^k M_{0} \cdots M_{n-1}\rVert, \quad k=1,\dots,d, 
\]
provided that the limit exists.  One has  $\theta_1 \ge \theta_2 \ge \cdots \ge \theta_d$.
A sufficient condition for strong convergence can then be stated as  follows (see \cite[Proposition 3.4.2 (ii)]{Arnold98} and  \cite{ABMST:19}):
Let  $\bM = (M_n)_{n\in {\mathbb N}}$ be a sequence of non-negative matrices in $\mathrm{GL}(d,{\mathbb N})$ for which the Lyapunov exponents exist; 
if $\bM$ satisfies the  growth  condition $\limsup_{n\to\infty}\frac1n\log \lVert M_n\rVert \le 0$   together  with the condition $\theta_1 > 0> \theta_2$, then strong  convergence holds.

Lyapunov exponents can be also defined for random products of matrices, where the  randomness can be  provided by  putting some distribution on the set of  matrices or  by iterating   measurable dynamical  systems that produce matrices  with a cocyle map  such as in (\ref{eq:gcocyle}). The first results in this direction were stated  for  sequences of independent random matrices with a given distribution function, with 
the Furstenberg–Kesten theorem  (see \cite{FurKes:60,Fur:63});  these results have then  been refined  via  Kingman’s subadditive ergodic theorem  and lastly via Oseledets' multiplicative ergodic theorem  (see for instance \cite{Arnold95,Arnold98,Viana:book})   proving that the limits  (involved in the  definition of Lyapunov exponents)  exist
almost surely and take almost everywhere the same value.

Let us give a flavour of such results.  We will come back to  this also in Section \ref{subsec:dyn}.  We  first  recall a few elements from  ergodic theory. The (left)
{\em  shift}  $S$  acts on a sequence  $(M_n)_{n \in {\GN}}$ as 
$S((M_n)_{n})=(M_{n+1})_n$ (i.e.\ the  first term of the sequence $(M_n)_n$ is deleted).  Let  ${\mathcal M}$ be a finite set of matrices in $\mathrm{GL}(d,{\mathbb Z}).$  Let $D\subset \mathcal{M}^{\mathbb N }$  be a closed shift-invariant  subset of $\mathrm{GL}(d,{\mathbb Z})^{\mathbb N }$.
A probability measure $\nu$ on $D$ is
called invariant if $\nu(S^{-1} A) = \nu(A)$ for every measurable set $A \subset D$.
An invariant probability measure on $D$ is  \emph{ergodic} if any shift-invariant measurable set
has either measure $0$ or $1$.  We have seen  (see  \eqref{eq:MA}) that such a set of matrices can be obtained  by  considering a  measurable map $T$ acting on some compact metric space and  a measurable map 
$A:  X \rightarrow \mathrm{GL}(d,{\mathbb Z}).$
The  sequences of matrices are then  of the form 
$A(T^n(x))$ and one studies the existence  and the almost everywhere behavior of   limits   of the form 
$$\lim_{n\to\infty} \frac{1}{n} \log \Vert A(x) \cdots A(T^{n-1}(x))\Vert.$$

\begin{theorem}\cite[Theorem~3.4.11]{Arnold98}\label{theo:genly}
 Let  $D\subset \mathcal{M}^{\mathbb N }$  be a closed shift-invariant  subset of $\mathrm{GL}(d,{\mathbb Z})^{\mathbb N }$ together with a shift invariant measure $\nu$.  Assume that  $(D,S,\nu)$ is  ergodic.   Let $A:D\to \mathrm{GL}(d,{\mathbb Z})$,   $\bM=(M_n)_n \mapsto M_0$. Assume that $A$ is log-integrable, i.e.,  
\begin{equation}\label{eq:logint} 
\int_D \log \max\{\lVert A(\bM)\rVert_{\infty}, \lVert A(\bM)^{-1}\rVert_{\infty}\}d\nu(\bM) < \infty .
\end{equation}
Then the quantities $\theta_i$ which are recursively defined by
\begin{equation} \label{eq:LyGenericBothSides}
\theta_1 + \cdots +\theta_k = {\lim_{n\to \infty}} \frac{1}{n}\log \lVert\wedge^k M_0 \cdots M_{n-1} \rVert, \quad k=1,\ldots,d, 
\end{equation}
exist and do not depend on $\bM$ for almost all $\bM \in D$. 
Here,  $\wedge^k $ stands for the  $k$-fold wedge product.
\end{theorem}

\section{Higher-dimensional   dynamical continued fractions} \label{sec:cf}

We   now focus on  multidimensional  continued fraction algorithms.  Their non-canonicity is  discussed in Section \ref{subsec:noncan}.  Section \ref{subsec:expected} recalls what  is usually expected from  a continued fraction algorithm.  We then focus on  algorithms    obtained  by iteration of a  dynamical system, which yield the so-called memory-less  algorithms. The  interest of their dynamical description is highlighted in Section \ref{subsec:dyn}.
Lastly, we  give examples of such algorithms in Section \ref{subsec:zoo}. 

% Lastly, we  stress the  variety of such algorithms in Section \ref{subsec:zoo}. 

\subsection{Non-canonicity in higher dimension }\label{subsec:noncan}

The aim of this section is to   present   several  facts    sustaining the claim  that  there is no canonical   multidimensional continued fractions algorithm. 

Firstly,   (usual)  continued fractions rely on 
Euclid's algorithm:   starting with  two numbers, one subtracts   the smallest from the largest (see Section \ref{sec:cf}).
If we start with   at least three numbers, it is not  clear how to decide  which  operation  has to    be performed, hence the diversity  of existing   generalizations.  For instance,  Brun's algorithm can be described  as  subtracting  the second largest entry from the largest one. See Section \ref{subsec:zoo}  for  more details.

% Moreover,  the  specific    algebraic structure of  $SL(2,\mathbb{N})$ plays an important role 
% for  one-dimensional continued fraction algorithms.  We have seen indeed, in Section \ref{sec:rcf}
% that the matrices that are produced  in the case of  usual  continued fractions  are unimodular matrices
% with non-negative   integer coefficients.    The    algebraic   structure of $SL(2,\mathbb{N})$   is particularly simple:
%   $SL(2,\GN)$  is a free  and finitely generated monoid; it admits as generators 
% $ \begin{bmatrix}
% 1&0\\
% 1&1
%  \end{bmatrix}$, $\begin{bmatrix}
% 1&1\\
% 0&1
%  \end{bmatrix};$
% any matrix  in $SL(2,\GN)$ thus admits a unique decomposition
% in terms of these two  matrices.
% This decomposition is  a matricial translation
%  of  Euclid's
% algorithm,
%  and  the   beginning of  the    continued fraction expansion  of $\alpha$  can be  recovered from 
%  the unique   decomposition of  $ (-1)^n \begin{bmatrix}
% p_{n+1} & q_{n +1}\\
% p_n & q_n 
% \end{bmatrix}   $ in the free monoid $SL(2,\mathbb{N})$.  This explains why most one-dimensional continued fraction algorithms are  closely  related.

Moreover,  the  specific    algebraic structure of  $SL(2,\mathbb{N})$ plays an important role 
for  one-dimensional continued fractions algorithms. Indeed,  the matrices produced  in the case of  usual  continued fractions  are unimodular matrices
with non-negative   integer coefficients (see Section \ref{sec:rcf}).    The    algebraic   structure of $SL(2,\mathbb{N})$   is particularly simple:
  $SL(2,\GN)$  is a free  and finitely generated monoid; it admits 
\begin{equation}\label{eq:generators}\begin{bmatrix}
1&0\\
1&1
 \end{bmatrix} \quad \text{and}\quad  \begin{bmatrix}
1&1\\
0&1
 \end{bmatrix}
 \end{equation}
as generators; any matrix  in $SL(2,\GN)$ thus admits a unique decomposition
in terms of the matrices given in \eqref{eq:generators}.
This decomposition is  a matricial translation
 of  Euclid's
algorithm,
 and   the    continued fraction expansion  of $\alpha$  can be  recovered from 
 the unique   decomposition of  matrices 
 \[(-1)^n \begin{bmatrix}
p_{n+1} & q_{n +1}\\
p_n & q_n 
\end{bmatrix}, \quad n\geq 0\]
in the free monoid $SL(2,\mathbb{N})$.  This explains why most one-dimensional continued fraction algorithms are  closely  related.

The situation   is completely different for  $SL(3,\GN)$ which is  not  finitely generated.
 Consider  e.g.\ the family of matrices 
     $${ M}_n:= \begin{bmatrix}
     1 &0 &  n \\
     1 & n-1 & 0\\
     1 & 1   & n-1
     \end{bmatrix}  .$$
     According to    \cite[Chap.\ 12]{PytheasFogg:2002}  these  matrices  are       undecomposable  for $n \geq 3$:
  they are  not equal to  an even permutation matrix, and,  for any
pair of matrices  $A,B \in SL(3,\GN)$ such that 
${M}_n={ A}{ B}$,  ${A}$ or ${ B}$ is an even permutation matrix.

Another approach for generalizing  continued fractions could rely  on properties of best approximation. 

% Another approach for generalizing  continued fractions would be to rely  on properties of best approximation. 

\begin{definition} \label{def:ba}
A rational number  $p/q$  is said to be     a   {\em best  approximation}  of a real   number
$ \alpha$  if     every  $p'/q'$ with  $  1\leq q' \leq q$, $p/q \neq p'/q'$
satisfies 
 $$|q \alpha -p|   <|q' \alpha -p'|.$$
 \end{definition}
Convergents in the continued fraction expansion of $\alpha$  and  best approximations  are known to coincide \cite{Cassels,Khintchine}.
Nevertheless, this notion  is not  so satisfying for defining  continued fractions in higher dimensions     as stressed in  \cite{Lagarias:82a,Lagarias:82b}.
Firstly,    best approximations  depend on the  choice of  a norm \cite{Lagarias:82a}, and secondly,
the unimodularity property is lost.  
  More precisely, one   has the following.
  \begin{definition} \label{def:BAd}
      Let $\balpha \in [0,1]^d$. Let $||\cdot||$ be a given norm in 
$\mathbb{R}^d$ and let  $|||\cdot|||$ denote the distance to the nearest integer. The \textit{sequence of best approximations} of $\balpha$ with respect to the norm
$||\cdot||$  is  defined as the increasing sequence of   non-negative integers  $(q^{(n)})_{n \in \mathbb{N}} $
such that
$|||q^{(n )}\balpha||| < |||q \balpha|||$
for any  $q$  with  $1 \leq q  < q^{(n)}$.
  \end{definition}
  The existence of an infinite sequence of    best approximations can  be   derived in a classic way   from Dirichlet's theorem or from  Minkowski's first  theorem. Best approximations  fail to be unimodular    \cite{Lagarias:82b}.
More precisely, consider the   square  matrix  $M_n$ of size $d+1$ whose rows are given by successive best  approximations vectors ${\bf{v}}_n=(p^{(n)}_1, \cdots, p_d ^{(n)},q^{(n)})$   providing  $ |||q^{(n)} \balpha  |||$. Let $D_n$ stand for the determinant of this matrix.
It is proved in   \cite{Lagarias:82b} that for any norm  in dimension $d \geq 2$,  there
exists
$\balpha \in \mathbb {R} ^d$, with
$\dim_{\mathbb Q} [1,\alpha_1, \cdots, \alpha_d]= d+1$, such that  for any positive integer $N$,
 there exists an $n$ for which $D_n=D_{n+1}=\cdots=D_{n+N}=0$. 
 Arbitrarily large determinants can even occur in dimension $d=2$ with the supremum norm.  For more on   best approximations and  multidimensional continued fractions, see the survey \cite{CHEV}; see  also  \cite{Mosh:2007} and Section \ref{sec:LLL}.
 
% Note that  badly approximable   vectors  cannot exhibit this pathological behavior???

\subsection{What  is expected?}\label{subsec:expected}

We  briefly recall here  the main properties expected from a continued fraction expansion. 
  For a  general discussion on the   quest for  suitable  higher-dimensional continued fractions,  see \cite{BRENTJES}. See also  \cite{Grabiner}
  for a discussion on their limitations.
We already discussed the fact that a  continued fraction algorithm  is expected 
to  yield  simultaneous   better and better  rational approximations with  the same denominator       for    $d$-tuples   $\balpha=(\alpha_1,\cdots, \alpha_d)$  in $[0,1]^d$,  in an effective
way 
  and  with a  good  approximation quality. 
  More precisely, it  has   to   produce a  sequence of positive integers $(q^{(n)})_n$ such that the 
  distance to the nearest integer   $|||  q^{(n)}  \balpha|||$  converges exponentially fast to  
  $0$ with  respect to $q^{(n)}$, and ideally in $(q^{(n)} )^{- \frac{1}{d}}$ (as predicted by  Dirichlet's theorem).

 From an arithmetic viewpoint,  a  multidimensional    continued fraction algorithm is also   expected  to 
detect linear relations  between $1,\alpha_1,\cdots,\alpha_d$,  and  to  give  algebraic  characterizations of   periodic  expansions.
Furthermore, using such an algorithm, one could hope to 
determine   fundamental units  (e.g.\ of a cubic  number field), and  to solve  Diophantine equations (as the ones described in Section \ref{sec:applications}).
The  formalism  of 
 multidimensional continued fractions based on  the Klein polyhedra and sails   developed  in
     \cite{Arnold:89,Lachaud:93,Korkina:94,Arnold:98bis} is well-suited  for the  detection of    periodic expansions in terms of algebraic number fields by providing  generalized Lagrange's theorem.  For more on the subject and its history,
     see e.g.  \cite{Lachaud:98a,Lachaud:98},  the references  in \cite{Karpenkov:09}, and the book \cite{Karpenkov:13}, which also includes a review of various generalizations of  continued fractions.

% Concerning the properties of best approximations, 
% we have seen in Section \ref{subsec:noncan} that the sequence of best approximations (see Definition \ref{def:BAd}) depends heavily on the chosen norm  and that the associated transformations are no longer unimodular \cite{Lagarias:82b}. However, 
% it is possible to  use the  action of  the diagonal flow on the space of unimodular lattices   to better understand their behavior  \cite{CC:19,CC:23}.

%  From an arithmetic viewpoint,  a  multidimensional    continued fraction algorithm is also   expected  to 
% detect linear relations  between $1,\alpha_1,\cdots,\alpha_d$,  and  to  give  algebraic  characterizations of   periodic  expansions.
% Furthermore,  one could hope to 
% determine  thanks to such an algorithm 
%   fundamental units  (e.g., of a cubic  number field), and  to solve  Diophantine equations (such as evoked in Section \ref{sec:applications}).
% Note that the  formalism  of 
%  multidimensional continued fractions based on     Klein polyhedra and sails   developed  in
%      \cite{Arnold:89,Lachaud:93,Korkina:94,Arnold:98bis} is well-suited  for the  detection of    periodic expansions in terms of algebraic number fields by providing  generalized Lagrange's theorem.  For more on the subject and its history,
%      see e.g.  \cite{Lachaud:98a,Lachaud:98},  the references  in \cite{Karpenkov:09}, and the book \cite{Karpenkov:13}  which also includes a review of various generalizations of  continued fractions.

     Concerning the properties of best approximations, 
we have seen in Section \ref{subsec:noncan} that the sequence of best approximations (see Definition \ref{def:BAd}) depends heavily on the chosen norm  and that the associated transformations are no longer unimodular \cite{Lagarias:82b,Mosh:2007}. However, 
it is possible to  use the  action of  the diagonal flow on the space of unimodular lattices   to better understand their behavior  \cite{Lagarias:94,KLM,CC:23}. See also Section \ref{sec:LLL}.

  From  a dynamical viewpoint, continued fraction algorithms are also expected
  to   have   reasonable     ergodic  properties   such as the ones described in  Section \ref{subsec:dyn}. For instance, we would like to have  control  over the (almost sure) behaviors concerning the growth  of the convergents, the   distribution  of the  partial quotients, or  the  speed  of convergence via Lyapunov exponents.
 Famous examples  of    algorithms   expressed   dynamically as  piecewise linear fractional transformations are  the  Jacobi--Perron,  Brun or  Selmer algorithms. Their description is the object of   Sections \ref{subesec:dyn} and \ref{subsec:zoo} below.

 %  From  a dynamical viewpoint, continued fraction algorithm is also expected
 %  to   have   reasonable     ergodic  properties   such as described in  Section \ref{subsec:dyn}.
 %  to be able   among other things 
 % to  control   (almost sure) behaviors, concerning  for instance  the growth  of the convergents, the   distribution  of the  partial quotients, or else  the  speed  of convergence via Lyapunov exponents.
 % Famous examples  of    algorithms   expressed   dynamically as  piecewise linear fractional transformations are  the  Jacobi--Perron,  Brun or  Selmer algorithms. Their description is the object of  next Sections \ref{subesec:dyn} and \ref{subsec:zoo}.
\subsection{Dynamical continued fraction algorithms}\label{subesec:dyn}
We recall  here  the main concepts related to dynamical continued fraction algorithms, expanding  the brief description    from Section \ref{subsec:approx}.
  This dynamical formalism  covers the most classical  
multidimensional continued fraction algorithms discussed in  the classical  references \cite{Szekeres,BRENTJES,SCHWEIGER}, see \cite{Lagarias:93} for more details. See also \cite{Lagarias:93} for  a   description of  the most classical multidimensional continued fraction algorithms  together with the results concerning  their Lyapunov exponents.

We   consider here algorithms that produce the sequence of matrices  $(A^{(n)})_{ n \in  \mathbb{N}}$  in  a dynamical way. 
We take  mostly  a measure-theoretical  viewpoint: the algorithms will be defined  almost everywhere with respect to the Lebesgue measure on $[0,1]^{d}$. We focus  on the  unimodular case, since this allows  a geometric interpretation in terms of  bases of  the integer  lattice ${\mathbb Z}^{d+1}$.

% We recall  here  the main concepts describing  dynamical continued fraction algorithms by expanding  the brief description    from Section \ref{subsec:approx}.
%   This dynamical formalism, such as  detailed  in  \cite{Lagarias:93},  covers the most classical  
% multidimensional continued fraction algorithms discussed in  the classical  references \cite{Szekeres,BRENTJES,SCHWEIGER}. Moreover,  \cite{Lagarias:93} also contains  a   description of  the most classical multidimensional continued fraction algorithms  together with  results concerning  their Lyapunov exponents.

% We   consider an   algorithm producing  the sequence of matrices  $(A^{(n)})_{ n \in  \mathbb{N}}$  in  dynamical terms. 
% We will  have  here  mostly  a measure-theoretical  viewpoint: the algorithms will be defined  a.e. with respect to the Lebesgue measure on $[0,1]^{d}$. We focus here on the  unimodular case, since this allows  a geometric interpretation in terms of  basis of  the integer  lattice ${\mathbb Z}^{d+1}$.

Let $X\subset [0,1]^{d}$. (Usually $X$ is simply $[0,1]^{d}$ but some algorithms can  be also defined  on  sets  of the form $ \{x= (x_1, \cdots, x_{d})  \in [0,1]^{d} \mid  
0 \leq x_1 \leq \cdots \leq x_{ d} \leq 1\}.$)
A $d$-dimensional  unimodular \emph{continued fraction map} over $X$ is   given by   measurable  maps 
$$ T \colon X \rightarrow X, \quad  A \colon  X \rightarrow   GL(d+1,{\mathbb Z}), \quad \theta \colon X \rightarrow  {\mathbb R}$$
such that 
for a.e.   $\balpha\in X$:
\begin{equation}\label{eq:pm}
    \begin{bmatrix}\balpha\\
1
\end{bmatrix}
= \theta(\balpha)  A(\balpha) \begin{bmatrix}
    T(\balpha)\\1
\end{bmatrix}.
\end{equation}
The maps $A$ and $T$ play the main role in the algorithm; the  role played  by the map  $\theta$  is minor: it  serves as a renormalization.

The associated \emph{continued fraction algorithm} consists of iteratively applying the map $T$ to a vector ${\balpha}\in X$.
This yields a sequence of matrices $
(A(T^n({\balpha})))_{n\geq 1}
$, called the \emph{continued fraction expansion} of ${\balpha}$.
%Let
%\[
%A^{(n)}(\balpha) = A(\balpha ) A({ T(\balpha)}) \ldots %A({ T^{n-1}(\balpha)}), \quad  \theta^{(n)}(\balpha) = %\theta({\balpha} ) \theta({ T(\balpha)}) \ldots \theta({ %T^{n-1}(\balpha)}).
%\]
One has 
$$\begin{bmatrix}
 \balpha \\
1
\end{bmatrix}=  \theta({\balpha} ) \theta({ T(\balpha)}) \ldots \theta({ T^{n-1}(\balpha)})  A(\balpha ) A({ T(\balpha)}) \ldots A({ T^{n-1}(\balpha)})\begin{bmatrix}
T^n(\balpha)\\
1\end{bmatrix}.$$

Such an algorithm is said to be  `without memory'. Indeed,   the $(n+1)$-th step of the algorithm depends only on  the map $T$ and on the value $T^n(\balpha)$. This  is in contrast 
with the    algorithms based on  lattice reduction  that we will discuss
 in Section \ref{sec:LLL}.

In most classical examples of such algorithms the continued fraction map $T$ is piecewise continuous, or even a piecewise homography (see Section  \ref{subsec:zoo}). Further, many  of them are {\em  positive}, that is, all the  matrices appearing as the images of the map $A$ are non-negative.
 
% The most classical  examples of such  algorithms  (see  Section \ref{subsec:zoo} for examples)  are defined  with   continued fraction maps $T$ being  piecewise continuous, and even  piecewise  homographies.  Moreover an algorithm is said to be {\em  positive} if all the  matrices   are non-negative.

 % The most classical  examples of such  algorithms  (see  Section \ref{subsec:zoo} for examples)  are defined  with   continued fraction maps $T$ being  piecewise continuous, and even  piecewise  homographies.  Moreover an algorithm is said to be {\em  positive} if all the  matrices   are non-negative.
 
 We illustrate  this formalism with the regular  continued fraction case.  Considering 
 \begin{align*}
\left[\begin{array}{l}
\alpha\\
1
\end{array}\right] &=\alpha T_G(\alpha) \cdots T_G^{n-1}(\alpha)  \left [\begin{array}{ll}
0 & 1\\
1 & a_1 
\end{array}\right] \cdots \left [\begin{array}{ll}
0 & 1\\
1 & a_n 
\end{array}\right]   \begin{bmatrix}
T_G^{n}(\alpha)\\
1
\end{bmatrix} \\
& =  \alpha T_G(\alpha) \cdots T_G^{n-1}(\alpha)  \left [\begin{array}{ll}
p_{n-1} & p_{n}\\
q_{n-1}& q_{ n}
\end{array}\right]  \begin{bmatrix}
T_G^{n}(\alpha)\\
1
\end{bmatrix},
\end{align*}
  one gets
$$\alpha\,  T_G(\alpha) \cdots T_{G}^{n-1} (\alpha)= | q_{n-1} \alpha  - p_{n -1}|,$$ where $p_n/q_n$ stands for the $n$-th convergent of  $\alpha$,
and   $T_G$ stands for the Gauss map. 
 We see  with this example that the map $\theta$  which allows the renormalization  with respect to the last coordinate (set   to $1$) is of  an arithmetic nature.

We  now   revisit this   formalism  more geometrically  in terms of lattice bases following \cite{BRENTJES}.  One   approximates the   vectorial  line  in ${\mathbb R}^{d+1}$
 directed    by the non-zero vector 
 %${\mathbf \ell}=(\ell_1, \cdots, \ell_{d+1})$ in $\mathbb{R}^{d+1}$ (previously discussed as 
 $(\balpha,1)$
 by   a sequence   of integer  lattice bases $({\bf b}^{(n)})_{n \in  \mathbb{N}}$ of $\mathbb{Z} ^{d+1}$, namely the $d+1$  columns of the matrices  $A(\balpha) \cdots A(T^{n-1} (\balpha)$). 
 The lattice bases    generate  cones  that  are expected to  
converge  toward the  line  directed by $(\balpha,1)$.
 Moreover, if the algorithm is  positive, then these cones are nested and  
$(\balpha,1)$  belongs to the   positive  cone  generated by the vectors  ${\bf b}_i ^{(n)}$,  
 $i =1,\ldots, d+1$, i.e., 
 $$(\balpha,1) \in \left \{\sum _{ 1 \leq i \leq d+1} \lambda _i  {\bf b}_i ^{(n)}\mid      \lambda_i \geq 0 \   \text{ for all } i =1,\cdots,d+1  \right\}.$$
 In fact, (\ref{eq:pm}) implies that  the coefficients 
 $(\lambda_i)_{1 \leq i \leq d+1}$ are  proportional to
 the vector $(T^n (\balpha), 1)$.
 %if the ${\bf b}_i ^{(n)}$
% stand for the column vectors of the product matrix $A(\balpha) \cdots A(T^{n-1} (\balpha)$, i.e., 
% $$(\balpha,1)= \theta (\balpha)  \cdots   \theta (T^{n-1}(\balpha)) \left(\sum _{i =1} ^{d} 
 %T^{n } (\balpha)_i {\bf b}_i ^{(n)}+ {\bf b}_{d+1} ^{(n)}\right),$$
% where $T_n(\balpha)_i$ stands for the $i$-th corrdinate of the vector $T^n (\balpha)$.
 The algorithm  thus produces  a  sequence of  bases of 
 the lattice ${\mathbb Z}^{d+1}$  that all   determine  a homogeneous cone  in ${\mathbb R}^{d+1}$  that contains the ray 
$\{ \lambda (\balpha,1) \mid  \lambda \geq 0\}$.    
This is the viewpoint developed for instance in  \cite{BRENTJES}; in fact,  the algorithms there are designed in  this way.

We recall that the continued fraction map $T$ acts on parameters $\balpha$  living in an ambient $d$-dimensional space.  In view of what precedes,  it is also natural to  
start directly  with a general line $\bl=(\ell_1, \cdots, \ell_{d+1})$ in $\mathbb{R}^{d+1}$ and to  have an algorithm  with such a line as an input.  But then, there  is no canonical way 
 to    `projectivize' the algorithm, i.e.,
   to go from the line $\bl$ to a   $d$-dimensional vector 
   $\alpha$ working in a  compact set $X$  on which a dynamical system  $T$  acts. 
 One can set e.g.   $\alpha_i=\ell_i/\ell_{d+1}$ for $i=1,\cdots,d$ (e.g. if the line belongs to the positive cone of $\mathbb{R}^{d+1}$  and  if $\ell_{d+1}$ is the largest entry).
Usual ways  to go from  some ${\bl}\in \mathbb{R}_+^{d+1}$
to some  $\balpha \in  [0,1]^d$   
consists in  setting $\ell_{d+1}=1$, and  working with entries $\ell_i \in [0,1]^d$ for $1 \leq i \leq d$,
 or   else, in working with the simplex  $\sum_{ i=1}^{d+1} \ell_i =1$, with $\ell_i  \geq 0 $ for all $i$.  See for instance \cite{SCHWEIGER} for more    details.
 One can also choose to work  directly   on the projective space  ${\mathbb P} ({\mathbb R}^d)$
 by associating with  each element $ [y_1: y_2:  \cdots: y_{d-1}:y_d]$  the representative   defined by $\max y_i=1$ and by working with  projectivizations of  matrices in $GL(d, {\mathbb Z})$. This possibility  of having    different choices for a same piecewise $d+1$-dimensional linear  map   explains the abundance of  existing algorithms (as illustrated in Section \ref{subsec:zoo}).

 % Let us discuss   now the    steps, i.e.,  the allowed operations that  can be  performed  on the bases, or else, the   choice  allowed for  partial quotient matrices $A^{(n)}$. 
 % An algorithm is said to be {\em additive} if all the  matrices   belong to a  finite set, i.e., the  map $A$ from (\ref{eq:MA}) takes finitely  many values. 

 Let us discuss   now the   possible steps of the algorithms, i.e.,  the operations that  can be  performed  on the bases together with the  choices  allowed for  partial quotient matrices $A^{(n)}$. 
 An algorithm is said to be {\em additive} if all the  matrices   belong to a  finite set, i.e., the  map $A$ from (\ref{eq:MA}) takes finitely  many values. As an illustration, the additive version of the Gauss map is given by 
the Farey map
$$ x \mapsto \begin{cases}
\displaystyle \frac{x}{1-x} & \text{ for } 0 \leq x \leq 1/2, \\
\\
\displaystyle \frac{1-x}{x} & \text{ for }1/2 \leq x \leq 1.
 \end{cases}$$
Dynamically, this creates    non-trivial changes; indeed,  the   Gauss map is known  to have  a  finite ergodic invariant measure,
which is not the case for the Farey map.

% One convenient way to get an additive algorithm is to   restrict  the range  of   matrices  for the map $A$  to the set  of matrices in $\mathrm{GL}(d+1,{\mathbb Z})$  with  entries in $\{0,1\}$ made of  elementary matrices   and permutation  matrices.  An  elementary matrix  is  a matrix  having  $1$'s on the diagonal, one  entry  equal to $1$  elsewhere, and $0$'s otherwise.

One convenient way to get an additive algorithm is to   restrict  the range  of the map $A$  to the set  of elementary and   permutation matrices  with  entries in $\{0,1\}$.  A matrix is called elementary if it has $1$'s on the diagonal, one  entry  equal to $1$  elsewhere, and all other entries equal to $0$.
In geometric terms (still following  the  geometric formalism from \cite{BRENTJES})  the allowed   operations   on  the bases at each step $n$ are  of elementary types (they correspond
to integer  transvections):
 for  every $n$, there exist  $i\neq j$ (with $i,j$ depending on $n$)   and $c^{(n)} \in \mathbb{N}$ such  that

\begin{equation}\label{eq:brentjes}
{\bf b}_i ^{(n+1)}={\bf b}_i^{(n)}+  c^{(n)} {\bf b}_j ^{(n)},  \  {\bf b}_k ^{(n+1)}={\bf b}_k ^{(n)} \mbox{ for } k \neq i.
\end{equation} This restriction 
is not a  severe one and most of the algorithms discussed in the present  survey   enter this framework, by allowing also
 permutation rules  between the vectors.
 Algorithms  for which the choice of the coefficients $i,j$ and 
 $c^{(n)}$ depend  only on
  the cofactors  of $\bl$ with respect to  ${\bf b}^{(n)}$, i.e.,   the integers  $a_i^{(n)}$
   such that
   $$\bl = a _1 ^{(n)}  {\bf  b }_1^{(n)}+\cdots+ a _{d+1}  ^{(n)} {\bf  b}_{d+1} ^{(n)}, $$  are called {\em vectorial} in \cite{BRENTJES}. They are memory-less  algorithms. In particular,  with the notation of (\ref{eq:MA}), if, for some $\balpha$ the matrix $A(\balpha)$ is  equal  to  an elementary matrix,  then the entries of
   $T(\balpha,1)$ are obtained (modulo the  renormalization  produced by the map $\theta$)  by subtracting  an  entry from another one; it is obtained by performing subtractions.

The terminologies   additive vs.~multiplicative or  division vs.~subtractive  algorithm are commonly  used:  see e.g.~\cite{BRENTJES} where an algorithm is said to be 
subtractive if $c^{(n)}=1$ in (\ref{eq:brentjes}), and additive  if $c^{(n)}$ is   chosen as the maximal possible number  allowing
the line to stay within the positive cone generated by the convergent vectors ${\bf b}_i  ^{(n)}$ for $i=1,\dots,d$.  Additive  and multiplicative versions for a same type of rule can lead  to  very distinct behaviors. See for instance  \cite{BRENTJES}, where it is shown that the multiplicative form  of Selmer's algorithm is not able to be accelerated.
The underlying cause of this is the fact that  the group  of matrices generated  by positive transvections and permutations is not commutative.

%  See for instance   the example of Selmer's algorithm quoted in \cite{BRENTJES}
% which   is  shown  not to be  able to be accelerated  in a  multiplicative form. 
% This comes from the fact that  the group  of matrices generated  by positive transvections and permutations is not commutative.

%As underlined in   \cite{BRENTJES}, ``All continued fraction algorithms
% which have been proposed since the beginning (Jacobi, 1868), and up to about
% 1970 belong to this class. [...].  A  great disadvantage  is that the expansions of  vectorial algorithms often converge too slowly or not at all.'' 
% Nevertheless they are easier to study from  an ergodic  viewpoint for instance. 
% We recall that the convergence (weak or strong) has to  do with the fact that the  vectors of  the basis (the columns of the matrices) have  to converge (in angle or  in distance) to the   ray $(\balpha,1)$.  
 %   In particular,  as soon as one  has weak convergence, a continued fraction algorithm allows one  to approximate real vectors by sequences of rational vectors.  However 
  %  in order to reach a quality of approximation, at least exponential,    to be compared to Dirichlet's theorem, one needs strong convergence (almost everywhere  at least).

\subsection{The effectiveness of ergodic theory}\label{subsec:dyn}
Having an  underlying dynamical system offers a wide  range of  mathematical  tools, including   ergodic theory and   thermodynamic formalism via  transfer operators. 
Ergodic theorems  describe the limiting behavior of    ergodic sums of the form    $  \frac{1}{n} \sum_{k=0}^{n-1}
f \circ T^k  $.   In  probabilistic terms,
 the random variables  $ f \circ T^n $ 
satisfy the strong law of large numbers under the ergodic hypothesis on $T$.
Ergodic  sums  $\frac{1}{n} \sum_{k=0}^{n-1}
f \circ T^k  $    allow   the expression of a  wide  set of algorithmic  and arithmetic  parameters, and 
ergodic theorems  allow   the understanding of their   mean behaviors.
Besides ergodic theory,   transfer operators   provide  a description of  
 the evolution of  probability density functions, by   transporting  the action of the  map $T$  from a dynamical system to the  densities. If initial conditions  $(x_i)$   for trajectories are distributed according  to a probability density function, then 
the new  collection of points $(T(x_i))$   is  distributed according to a new  probability density function,  obtained by applying   a transfer  operator.

Ergodic theory has been quite an effective tool in understanding the typical behavior of various expansions of numbers and their quality of approximation. To name a few, we mention $\beta$-expansions, L\"uroth series and various one and multi-dimensional continued fractions.

We first illustrate it with the Gauss map. For general ergodic aspects of the Gauss map, see  e.g.  \cite{Bill:78,DaKr02,EW:2011}. The  statistics of occurrences of     partial quotients in   continued fraction expansions are deduced  from  the ergodic theorem applied to the Gauss map, and having an  ergodic absolutely continuous invariant measure   then provides metric results that hold  almost everywhere  with respect to the Lebesgue measure.

To explain this in more detail, let us recall  basic   ergodic properties of the  Gauss map. 
We endow the  dynamical  system $([0,1], T_G)$  with a structure
of   a  measure-theoretic dynamical 
system.   A {\em measure-theoretic dynamical 
system} is defined as  a system
 $(X, T, \mu,{\mathcal B})$,
where  $\mu$ is a probability measure defined on  the $\sigma$-algebra 
${\mathcal B}$
of subsets of  $X$,
and $T: X \rightarrow X$  is a measurable map  which
preserves the measure $\mu$, i.e.\ $\mu(T^{-1} (B))=\mu(B)$ for all $B \in {\mathcal B}$. In this case one also says that the measure $\mu$ is {\em $T$-invariant}.  Here,   we  endow $([0,1], T_G)$  with
the Gauss measure $\mu_G$ which  is   a  Borel probability measure   absolutely continuous 
with respect to the Lebesgue measure defined via the following density function
$$\mu_G = \frac{1}{\log 2} \int   \frac{1}{1+x} {dx}.$$
One checks that this measure is $T_G$-invariant. The  Gauss map  is   {\em ergodic} with respect to the Gauss measure, that is, 
every  Borel subset $B$ of $[0,1]$ such that
$T_G^{-1}(B)=B$ has  either zero  or full measure.
This implies that   almost all  orbits are dense in $[0,1]$.
%(almost    all means that the set of elements  $x$  whose orbit is not dense is    contained in a set of zero  measure). 
%More generally  a property  is said to  hold  {\em almost everywhere} (abbreviated as   a.e.)
 %if the set of elements for which the property does not hold is    contained in a set of zero  measure;  this property is  said  to be {\em generic}  (the points  that satisfy this property are then also  said to be  generic). This helps  us to give a meaning to the notion of   typical  behavior for a  dynamical system.
Ergodicity yields furthermore  the following striking   convergence result. 
% taht is usually sketched as ``
   Indeed, measure-theoretic ergodic dynamical systems satisfy  {\em Birkhoff's ergodic
 theorem},
also called {\em individual ergodic theorem}, which relates spatial means to temporal means.
\begin{theorem}[Birkhoff's ergodic theorem] \label{ch0:the:Birkhoff}
Let $(X, T,
    \mu,{\mathcal B})$ be an ergodic  measure-theoretic dynamical system.
Let $ f\in L^1(X,\mathbb{R}) $.   Then  the sequence 
$ (\frac{1}{n} \sum_{k=0}^{n-1}
f \circ T^k )_{n \ge 0}$ converges  a.e.  to  $ \int_X f\, d\mu$:

$$\forall f\in L^1(X,\mathbb{R}) \ , \ \ \ \frac{1}{n} \sum_{k=0}^{n-1}
f \circ T^k \xrightarrow[ n \rightarrow \infty]{\mu - a.e.}
\int_{X} f\, d\mu \ .$$
\end{theorem}
%$$\frac{1}{n} \sum_{j=0}^{n-1} f(T^j x)= \int  f  d\mu$$

For example in the case of continued fractions, using this theorem one is able to describe the distribution of the digits.

\begin{theorem}
For a.e.~$\alpha \in [0,1]$ the digit $j$ occurs  in the continued fraction expansion of $\alpha$ with density $\displaystyle\frac{1}{\log 2}  ( 2 \log (1+j) -\log j - \log (2+j))$.
\end{theorem}
One can even  show the digits  have the mixing property, as well as give a short and elegant proof of L\'evy's theorem which states that 
\begin{equation}\label{eq:levy}\lim_{n\to\infty}\frac{1}{n}\log q_n(x)=\frac{\pi^2}{12\log 2}
\end{equation}for Lebesgue almost every point $x$. With the help of a dynamical construction called the {\em natural extension} (which is basically a way to make the dynamics invertible), one can describe the asymptotic behavior of the approximation coefficients $$\Theta_n(x)=q_n^2|x-\frac{p_n}{q_n}|.$$  As an illustration,  we mention the ergodic proof of Bosma, Jager and Wiedijk (\cite{BJW}) of the famous Doeblin-Lenstra conjecture which says that for almost all $x$ the limit
\[
\lim_{n\rightarrow \infty} \frac{1}{n} \# \{ 1\leq j\leq n\, : \,
\Theta_j(x)\leq z \} \, ,\, {\mbox{ where }} 0\leq z\leq 1,
\]
exists and equals the distribution function $F(z)$ given by
\begin{equation}\label{distributionfunction}
F(z)=\begin{cases}
        \dfrac{z}{\log 2} & 0\leq z\leq \frac{1}{2},\\
         & \\
        \dfrac{1}{\log 2} (1-z+ \log 2z) & \frac{1}{2} \leq z \leq 1.
    \end{cases}
\end{equation}

In \cite{Jag}, Jager used ergodic tools to describe the simultaneous distribution of two consecutive $\Theta$'s by determining  for almost all $x$ the exact value of
\[
\lim_{n\rightarrow \infty} \frac{1}{n} \# \{ 1\leq j\leq n\, : \,
\Theta_{j-1}(x)\leq z_1 , \, \Theta_{j}(x)\leq z_2 \} ,\, {\mbox{ where }} 0\leq z_1,z_2\leq 1.
\]

As a third example we mention that it can be proven with the help of ergodic theory that for each real irrational number $x$ and each integer $n\geq 1$ one has
\begin{equation}\label{borel}
\min (\Theta_{n-1}, \Theta_n, \Theta_{n+1})\, <\, \frac{1}{\sqrt{a_{n+1}^2+4}}
\end{equation}
and
\begin{equation}\label{tong}
\max (\Theta_{n-1}, \Theta_n, \Theta_{n+1})\, >\, \frac{1}{\sqrt{a_{n+1}^2+4}}\, .
\end{equation}
For a generalization  of these results in the  setting  of best approximations, see \cite{CC:19,CC:23}, and also Section  \ref{sec:LLL}. Inequality~(\ref{borel}) is a generalization of a result by Borel~\cite{Bor}, which states
that
$$
\min (\Theta_{n-1}, \Theta_n, \Theta_{n+1})\, <\, \frac{1}{\sqrt{5}}\, .
$$

A great number of people independently found~(\ref{borel}), see for
example~\cite{Obr,BMcL,Sen}. Inequality~(\ref{tong}) is due to
Tong~\cite{Tong}. In fact, ergodic theoretic methods yield easy
proofs of generalizations of a great number of classical results by
Fujiwara, Segre and others like LeVeque, Sz\"{u}zs, and Segre;
see~\cite{JK,DaKr02} for more results and details.

In a different direction, ergodic theoretic methods can also be used to prove (generalizations of) Lochs' theorem which compares the amount of information given by the digits of different types of number expansions. The original statement of Lochs \cite{Lochs} compared the regular continued fraction digits to decimal digits. For an $x \in (0,1)$ let $(a_k)_{k \ge 1}$ denote its regular continued fraction digits and $(d_k)_{k \ge 1}$ its decimal digits and for each $n \ge 1$ let $m_n(x)$ denote the largest number of digits $a_k$ that can be determined from knowing $d_1, \ldots, d_n$. Then Lochs proved the following statement.

\begin{theorem}
For a.e.~$x \in (0,1)$, $\lim_{n \to \infty} \frac{m_n(x)}{n} = \frac{6\log 2 \log 10}{\pi^2}$.
\end{theorem}

See \cite{LiWu,BarreiraIommi} for similar results for other types of expansions. Lochs' theorem was placed in a dynamical setting in \cite{bosma99} and developed further in \cite{DajaniFieldsteel}. It became apparent that the number $\frac{6\log 2 \log 10}{\pi^2}$ is related to the measure theoretic entropies $h (T_G) = \frac{\pi^2}{6\log 2}$ and $h(T_{10}) = \log 10$ of the transformations $T_G$ and $T_{10}(x) = 10 x \pmod 1$ that generate regular continued fractions and decimal expansions, respectively. Later Lochs' theorem was extended and generalised in several directions, e.g.~in \cite{KVZ} for random dynamical systems and in \cite{ValerieLochs} to include systems with zero entropy.

\bigskip

Let us consider now  the case of  multidimensional continued fraction  algorithms. For a general   description of their   ergodic  properties, see  \cite{SCHWEIGER}. A first step in the ergodic study of a continued fraction algorithm   consists of proving the existence of the  so-called a.c.i.m (absolutely  continuous  invariant measure).    Note that a  very efficient way to find a.c.i.m.'s for continued fraction transformations is via the natural extensions; see e.g. 
   \cite{ArnouxNogueira93,AL18,ArnouxSchmidt:19,EINN19,KLMM20}. 
   
Let us now consider   more closely convergence properties from a dynamical viewpoint. 
 Lyapounov exponents of  classical algorithms such as  the  Jacobi--Perron  algorithm or the   Brun algorithm  have been thoroughly studied in       \cite{BROISEJP,Broise:01}; see also \cite{FS:21} for the simplicity of the spectrum.
In particular, the a.e.   exponential   (strong) convergence of the Brun   \cite{FUKO,Meester:99,Schratz:01} and  Jacobi--Perron algorithm \cite{Broise:01} (see also  \cite{Lagarias:93,Schweiger96}) holds in dimension $d=2$:
there exists  $\delta >0$  s.t. for a.e.   $(\alpha_1,\alpha_2)$,  there  exists
  $n_0=n_0 (\alpha, \beta)$  s.t.
    for all  $n \geq n_0$
  $$|\alpha_1- p^{(n)}_1/q^{(n)}| < \frac{1}{(q^{(n)}) ^{1+\delta}}, \  
  |\alpha_2- p^{(n)}_2/q^{(n)}| < \frac{1}{(q^{(n)}) ^{1+\delta}},$$
  where 
  $p^{(n)}_1,p^{(n)}_2, q^{(n)}$    are given by the Brun (resp. Jacobi--Perron) algorithm.

We have seen already in Section \ref{subsec:conv} that the quality of rational approximations provided by continued fraction algorithms obtained dynamically by an iteration of a  map  can be    expressed in terms  of the two first Lyapunov exponents.  There is, in fact,  a strong link   with the uniform approximation exponent as shown in  \cite{Lagarias:93}; see  also \cite[Theorem~1]{HK00}, \cite[Proposition~4]{Baldwin:92}
and   \cite[Section 2]{berth2021second}. One considers  as in (\ref{eq:MA}) the  maps $T: X \rightarrow X$, $A: X \rightarrow  \mathrm{GL} (d+1 ,{\mathbb Z}).$ 
For fixed $\balpha\in  X $ and $i\in\{1,\ldots,d+1\}$, we have 
\[
\eta_A^*(\balpha,i) = \sup
\bigg\{
\delta >0 \;:\; \exists\, n_0=n_0(\balpha,i,\delta) \in\mathbb{N} \hbox{ s.\ t.\ }\forall \, 
n\ge n_0,\; \bigg\Vert  \balpha-\frac{\mathbf{p}_i^{(n)}}{q_i^{(n)}} \bigg\Vert < (q_i^{(n)})^{-\delta}
\bigg\},
\]
where $\Vert\cdot\Vert$  is an arbitrary norm in $\mathbb{R}^d$. The quantity
\[
\eta^*_A(\bx) = \min_{0\le i\le d} \eta_A^*(\bx,i)
\]
is called the \emph{uniform approximation exponent} for $\balpha$ using the algorithm $A$. Now, following \cite[Theorem~4.1]{Lagarias:93},  we consider a  $d$-dimensional multidimensional continued fraction algorithm $A$  satisfying some mild ergodic  conditions (called (H1) to (H5) in \cite{Lagarias:93}).
Let $\eta^*_A$ be the uniform approximation exponent of  $A$. We have $\lambda_1(A) > \lambda_2(A)$ and 
\[
\eta^*_A(\mathbf{x}) = 1- \frac{\lambda_2(A)}{\lambda_1(A)}
\]
holds for almost all $\mathbf{x}\in X$. In particular, if $\lambda_2(A) < 0$ then $A$ is a.e.\ strongly convergent.

 The computation of Lyapunov exponents    is   a challenging problem.   For numerical  results, see \cite{berth2021second}.
In practice, ergodic theorems  provide
efficient  ways of   estimating Lyapunov exponents   numerically by following   trajectories and then  taking averages over truncated  trajectories.
This has been developed   e.g.  in \cite{HarKha2,BalNog} for continued fractions   or  in \cite{Zorich} for interval exchanges.
Transfer operators  are   efficient ways to reach them, such as developed  e.g.  in \cite{Pol:2010,JenPoV}, by getting  in some cases  exact computations for the top Lyapunov exponents of random products of matrices using transfer operators.   Indeed, transfer operators 
come with   the   analogue of   Perron--Frobenius    theory for non-negative matrices. They  are  well suited to  provide approximations  by    working   in   finite  dimension, for instance  by    truncating  Taylor expansions of analytic functions \cite{DFV:97,Lhote}. 
  See also  \cite{ST:18}  for  lower and upper bounds for 
    the  Lyapunov exponents for random products of   square matrices  
     with determinant $1$ having real  and   non-negative entries.

\subsection{Classical examples}\label{subsec:zoo}

With this section we want  to illustrate the variety of 
continued fractions algorithms   defined according to the formalism of Section \ref{subsec:dyn}, as well as to focus on   two classical algorithms,  namely the Brun and the Jacobi--Perron algorithms.

Consider an  algorithm based on elementary matrices, i.e.,
on   subtractions, such as described in Section \ref{subsec:dyn}.
 In order to stress the  simple 
rules  that govern  them, we   express them in dimension $d+1=3$. We thus  start  with parameters 
$(\ell_1,\ell_2,\ell_3) \in \mathbb{R}^3_+$. We have to decide which   number has to be subtracted, and with respect  to  which number it   has to be done.
Usually   numbers  $\ell_1,\ell_2,\ell_3 $ are sorted in  increasing (or decreasing)  order. We  stress the subtraction rule  but it is usually preceded and  followed by a sorting operation.

For Jacobi--Perron, we   subtract the second entry  from the  other ones, with the entries being  ordered  in such a way that the first entry is the largest one.
For Brun, we subtract the second largest from the largest one.
 For Poincar\'e,  we subtract  the second largest entry from the largest one,  the third largest from the second largest, etc.   For  Selmer, 
 we subtract the smallest  (positive) entry   from the largest  one. For the  Fully subtractive
 algorithm, 
we subtract the smallest   (positive entry)   from all  the larger  ones.

% For Jacobi--Perron, we   subtract the last one to the two other ones, ordered in a way that the  first entry is the largest one.
% For Brun algorithm, we subtract the second largest to the other.
%  For Poincar\'e,  we subtract  the second largest entry from the largest one,  and   the smallest   entry  from the  second largest one. For  Selmer, 
%  we subtract the smallest  positive entry   from the largest  one. For the  Fully subtractive
%  algorithm, 
% we subtract the smallest   positive entry   from all  the largest  ones.

  % See  \cite{Nog95}  for Poincar\'e 	algorithm. 
   % Similar   intriguing     issues occur in the study of %the  fully subtractive algorithm; they have been  %considered
   % in \cite{MeesterNowicki,KraiMeester}.

 We  also recall that given an algorithm described at the level of 
 a $d+1$-dimensional space, there exist  several possible projectivizations.
 This is the case of the various forms  taken by the  Brun algorithm.  In particular, the Brun algorithm is also called  modified Jacobi--Perron algorithm:  the modified Jacobi--Perron algorithm
introduced by Podsypanin
in \cite{Pod:77} 
 is a two-point extension
of the Brun  algorithm.

Consider now the Jacobi--Perron algorithm (that will be handled in more details  with  respect to the nearest integer part
in Section \ref{sec:ni}). 
The linear form of the Jacobi--Perron
algorithm is defined on $\{(y_0,y_1,y_2)\in\mathbb{R}^3 \setminus\{\mathbf{0}\}:\, \ y_0  \geq  y_1,y_2 \geq 0\}$ by
$$(y_0,y_1,y_2) \mapsto (y_1,y_2-\lfloor y_2/y_1\rfloor y_1,y_0-\lfloor y_0/y_1\rfloor y_1).$$
If we  set 
$$x_1:=y_1/y_0 , \  x_2:= y_2/y_0,$$
 we recover its     {projective}  version defined on
$[0,1]^2$ as
$$(x_1,x_2) \mapsto \left( \{x_2/x_1\}, \{1/x_1\}\right).$$

  Let us stress  the difference  between Brun's  and
 Jacobi--Perron's rule. The Brun algorithm  is  a space-ordering algorithm according to the terminology introduced  in \cite{HK01}. 
 (Note that it is called   ordered Jacobi--Perron  in \cite{HK00}.)  Furthermore, each step  of the Brun   algorithm produces only  one  digit.
 This    helps in  computing the   natural extension and  the invariant measure of the Brun  algorithm
 (see e.g. \cite{ArnouxNogueira93}). Contrary to the Brun  algorithm,  the role played  by  $y_1$ and $y_2$ (in the description above)  is not determined  by  a comparison
between both  parameters in the Jacobi--Perron case; this might  explain the fact that  an explicit  expression of the   natural extension  of this   algorithm is  still not known.
Nevertheless,   the framework of $S$-expansions and the  so-called techniques
of  Insertion and  Singularization  (see \cite{Kraai02}) allow     one to relate both algorithms  as  shown in \cite{Schratz:07}; see also \cite{Schratz:08}.
Both   algorithms (Brun and Jacobi--Perron)
are  known 
 to have an invariant  ergodic  probability measure equivalent 
to the Lebesgue measure (see for instance 
\cite{Schweiger90} and  \cite{SCHWEIGER}).
However, this measure is not known explicitly   for Jacobi--Perron
(the density of the measure is shown to be a piecewise analytic function
in \cite{BROISEJP}), whereas it is known explicitly for  Brun  \cite{ArnouxNogueira93,FUKO}.  Note that  the Burn algorithm can be considered as an additive  algorithm \cite{BLV:18}, where it is proved that partial  quotients tend to be  equal to $1$. Lastly, let us quote \cite{Nogueira:01} for 
  Borel-Bernstein type theorems on the growth of partial quotients.

\section{Lattice reduction  and   rational approximations}\label{sec:LLL}

We now focus  on   the second approach  discussed in 
Section \ref{subsec:approx} based on lattice reduction.
Lattice reduction   methods  induce  indeed a     particularly  fruitful way of  exhibiting   good  simultaneous approximations,  or  else  small values
for linear forms. Algorithms based on  lattice  reduction theory are based on the following idea: lattice  reduction algorithms 
do not produce a priori the smallest vector of a lattice, but a reasonably  small  vector. That is, a vector that is      small enough for  guarantying Diophantine
approximation  properties that can be compared with Dirichlet's quality   up to  an approximation factor exponential in the dimension.
We can thus consider  these  algorithms   as providing  effective  versions  of  Dirichlet's theorem, yielding a satisfying 
 compromise between  efficient computation  and sharpness of the    obtained bounds, that is,     between algorithmic issues and   Diophantine quality.

Let us stress the fact that  the range of applications of lattice reduction     is quite wide for the following reasons.  They play a  central algorithmic role  in
   cryptology, computer algebra, integer linear programming and algorithmic  number theory.  They are particularly   versatile  in terms both  of  existing  variants and  algebraic contexts where they can be developed, see e.g.  \cite{FS:10,Camus:17,Napias:96}.  They  are efficient: LLL     has a    polynomial  runtime   with  respect to the dimension). In particular, in the present context,   they  produce efficient gcd algorithms (see e.g. \cite{HMM:98})  and there exist   promising attemps in order to  devise   continued fractions  upon  them (see \cite{Lagarias:94,BosSmeets:13,Beukers}). 
However, there remains much to understand   concerning   their executions  and  the geometry of the outputs \cite{NguyenStehle:06,YD:18}.

 Lattice reduction  is based on the following elementary  basis transformations:
 they  can be described in terms of   size reduction  (the vector ${\bf b}_i$  of the basis
$({\bf b}_1, \ldots, {\bf b}_{d+1})$  is replaced 
by ${\bf b}_i - \lambda {\bf b}_j$ with $1 \leq j <i$),  and  of exchange steps, also called swaps  (one exchanges ${\bf b}_i$ and ${\bf b}_{i+1}$). These operations are decided
with respect to the Gram-Schmitdt orthogonalization   of the  basis ${\bf b}$. 

More precisely, let $({\bf b}_i^*)$ stand for the basis obtained via the Gram-Schmidt orthogonalization from the basis $({\bf b}_1, \ldots, {\bf b}_{d+1})$, i.e., ${\bf b}_i^*$ is the orthogonal projection of ${\bf b}_i$  on the orthogonal of the space generated by 
${\bf b}_i, \cdots, {\bf b}_{i-1}.$ One writes
${\bf b}_i^*={\bf b}_i -\sum _{k=1}^{i-1}\mu_{ik} {\bf b}_k^*$ with $\mu_{ik}=\frac{\langle {\bf b_i},{\bf b}_k^* \rangle}{\langle  {\bf b_k}^*,{\bf b}_k^* \rangle}$ for $k\leq i-1$. A basis $({\bf b}_1,\cdots, {\bf b}_d)$ is said to be LLL-reduced if
\begin{itemize}
    \item  $|\mu_{ik}| \leq 1/2$  for all  $i,k$ with  $1\leq i \leq d+1$ and $k \leq i-1$ (the basis is said {\em proper});
    \item  $3/4 \Vert {\bf b}_i^*  \Vert \leq  \Vert \mu_{i+1,i}{\bf b}_i^*+{\bf b}_{i+1}^*\Vert$ for all $i$ (this condition is called Lovasz' condition).
\end{itemize}
The factor $3/4$ can be replaced by a parameter $t$ with 
$3/4< t < 1$. The LLL algorithm    consists of  two steps.
\begin{itemize}
    \item First, make the basis proper  by replacing ${\bf b}_i$ by ${\bf b}_i - \Vert \mu_{ij} \Vert {\bf b}_j$,  for $j < i$, where  $\Vert \mu_{ij} \Vert $ stands for the distance to the nearest integer;
    \item if for some $i$, Lovasz' condition is not satisfied, then  swap ${\bf b}_{i}$ and ${\bf b}_{i+1}$ and go the  previous step.
\end{itemize}

Recall that we have  sketched  the   basic  strategy underlying the  use of lattice reduction  in this framework in Section \ref{subsec:approx}. One starts with the lattice  $\Lambda_t$ generated by the columns of the  matrix
$$M_t:=\left[ \begin{matrix}
  1&0& \cdots  & 0 & -\alpha_1\\
  0&1&\cdots &0& -\alpha_2\\
    \vdots&\vdots&\ddots&\vdots&\vdots\\
  0&0&\cdots&1&-\alpha_d\\
  0 &0&\cdots&0&  {t}
  \end{matrix} \right].$$
Note that   $\det (M_t)=t$, hence, 
 the   lattice   $\Lambda_t$  changes  at each step of the  algorithm.  One lets $t$ tend to $0$.
 Let us stress the fact that this strategy differs  from the one discussed  in Section \ref{subsec:zoo} where   one worked with    bases of the fixed   lattice
 $\mathbb{Z}^{d+1}$.
 
Lattice reduction algorithms such as LLL then 
perform a succession of permutations and   subtractions on the matrix $M_t$, i.e., it multiplies  the matrix $M_t$ by elementary matrices and  permutation matrices.
This is a common feature between unimodular continued fraction algorithms and algorithms  based on lattice reduction, namely that they are made of a succession of permutations and subtractions. 
The decisions are taken for classical unimodular continued fractions by comparing entries, whereas lattice reduction involves quadratic comparisons in the sense that they depend on the Gram-Schmidt orthogonalization.

For  more on  the way, lattice reduction provides  best approximations 
of a real number,  see  \cite[p.226,267]{LLLbook}, and 
for a  survey on the overall strategy   for getting constructive type results in Diophantine  approximation based  on LLL, see  \cite[p. 222]{LLLbook}.
Nevertheless, note   that  even in dimension 2,   when using the Gauss algorithm whose efficiency has been largely  proved,
 one has  `little control on the convergent   which is returned; in particular,  this is {\em not}  the largest convergent  with denominator less  than 
 $2 \sqrt{C/3}$',  as quoted  in \cite[p.226 Example 1]{LLLbook};  the     bound   $2 \sqrt{C/3}$  comes from Theorem 7 of   \cite[Chapter 6]{LLLbook}.

Several  attempts  already  exist in order to   use lattice reduction to get  simultaneous  approximations.  Let us quote  
\cite{FF:79,Ferguson:87,Lagarias:82,Lagarias:85,Just:89,Just:90,Just:92} for strongly convergent
algorithms; however they do not  present the same advantages    as more   classical  memory-less  algorithms. Let us quote also     \cite{Lagarias:85} and    \cite{Lagarias:94}  based on  Minkowski lattice  reduction.  This  approach is not  effective,  but it   produces best approximations  (which are known to be NP-hard to locate in an interval  \cite{Lagarias:85}).
This study is  extended in \cite{GrabinerLagarias:01} and in \cite{CHEV}. Let us quote also
  \cite{Beukers,BosSmeets:13}   built   upon LLL.  In  \cite{BosSmeets:13} an iterated LLL algorithm is  designed that is obtained by decreasing the  parameter $t$ from  the lattice $\Lambda_t$, by dividing it  by a given constant (including  also experimental data).  In  \cite{Beukers},  
  the conditions  that occur in LLL are  proved   to be  linear in the parameter $t$  (tending to $0$). The idea at step $k$  is to consider the smallest parameter $t_k$  for which $\Lambda_{t_k}$ is reduced, and then perform a reduction with $t_k- \varepsilon$.

 An efficient way to input some dynamics with this  reduction viewpoint is  to rely  on homogeneous dynamics with the (left) action of the diagonal flow  $(g_t)$
 defined by 
 $$\begin{bmatrix}
e^t I_d &0 \\
0 & e^{-dt}
 \end{bmatrix}$$
 on the space of unimodular lattices, i.e., on  the homogeneous space $SL(d+1,{\mathbb R})/ SL(d+1,{\mathbb  Z})$. This  is  a  very fruitful way  to   combine lattice reduction with the strength of dynamical methods
such as highlighted in the survey \cite{CHEV}.  This  amounts to changing the parameter $t$.
See e.g. \cite{KLM}.  This is particularly relevant for  continued fractions defined in terms of  best approximations.
In \cite{CC:19,CC:23}, Levy's result  about the almost sure growth rate of the denominators of the
convergents from \eqref{eq:levy} is extended to the  best Diophantine approximations 
 (see Definition \ref{def:BAd}). The value of the limit is given by an integral 
over a codimension one submanifold  in the space of lattices $SL(d + 1, {\mathbb R})/ SL(d + 1, {\mathbb Z}).$   An analogue of the  Doeblin-Lenstra
discussed in Section \ref{subsec:dyn}  is also given in \cite{CC:19}. As highlighted in \cite{CC:19},
the section is provided by lattices whose first two minima  of lattices are equal and  the first return map of the geodesic  flow in
the transversal play the role of an invertible extension of the missing Gauss map. The main idea is to relate 
shortest and  so-called minimal  vectors of the lattice $\Lambda_t$ with  best approximations (see \cite[Lemma 8]{CHEV})
together with a  suitable choice of a norm   (see \cite[Section 3.2]{CHEV} and \cite{Cheung:11}).

\hide{

In \cite{DKTWY} a study of the LLL algorithm from the perspective of statistical physics by interpreting LLL as a sandpile model. See also \cite{MV:2010}.

 % Theorem 7.

A  short description of the LLL algorithm is given in  \cite{BosSmeets:13}. The approach of \cite{Beukers} is  with quadratic forms.  \cite{Beukers} and \cite{BosSmeets:13} work with the LLL algorithm in terms of Diophantine approximation.

See \cite{HJLS:89}   for an interesting discussion on the  connections between   the  approximation    algorithms   given in \cite{Ferguson:87}  (see also \cite{FF:79})
and \cite{LLL:82}). 
See also \cite{FF:79}.}

\section{Some applications of continued fractions}\label{sec:applications}

In this section we give a diverse range of  applications  of  continued fractions in arithmetics, cryptography and symbolic dynamics.
In the next section we then   propose  possible hints for  improvements of  dynamical unimodular continued fraction algorithms.

% In this section we  stress the diversity of     applications  of  continued fractions   by focusing  on   completely different  applications, highlighting the variety of ranges  where continued fractions show their relevance.

 We start with arithmetical applications. Continued fractions   play  also an important role in the  arithmetic of algebraic curves.
The relation between the geometry of the elliptic curve $y^2=x^4-6 a x^2-8 b x+c$ and the continued fraction of $y$ with respect to the $x^{-1}$-adic valuation is given in \cite{10.1112/plms/s3-41.3.481} and \cite{Abel+1826+185+221}. The given elliptic curve has two rational points at infinity, call them $P$ and $O$, such that $O$ is the origin of Mordell-Weill group of the elliptic curve. Then the order of  the image of $P$, which is called the infinity divisor of the curve, in the Jacobian is finite if and only if the continued fraction of $y$ is periodic. Moreover, the order depends on the period of the continued fraction. This result has been generalized for the hyperelliptic curves in \cite{Berry}. The torsion order of the infinity divisor is given as the sum of the degrees of all partial fractions from 0  to the period of the continued fraction of $y$ in \cite{10.1112/plms/s3-41.3.481,van2000quasi,Pappalardi2005335}. 
  Furthermore, thanks to the periodic continued fraction of $y$, for a given even degree hyperelliptic curve $y^2=f(x)$,   the Pell equation for $f$  is proved to have a non-trivial solution, i.e., there exist $p, q$ in $k[x], p$ not a constant, such that $p^2-q^2 f=1$, in \cite{10.1112/plms/s3-41.3.481,Berry}.

Cryptography is another field where continued fractions occur in various places. First of all, continued fractions can be used in cryptanalysis, e.g., to attack the RSA cryptosystem. This cryptosystem is based on the mathematical problem of factoring a natural number that is a product of two large prime numbers. It is an asymmetric cryptosystem, i.e., a publicly known key is used for encryption and a secret key is used for decryption. If the secret key is chosen too small, one can use continued fraction expansions to efficiently compute the private key and break the cryptosystem. This attack is known as Wiener's attack (see \cite{wiener}). Apart from cryptanalysis, there is a connection between stream ciphers and continued fraction expansions. Stream ciphers are symmetric cryptosystems, i.e., the same key is used for encryption and decryption. A stream cipher generates a pseudorandom bit string. To encrypt a message, each bit of the pseudorandom bit string is combined with a bit of the secret message (e.g. with the XOR operation). Niederreiter uses continued fraction expansions of generating functions to analyze the randomness of pseudorandom sequences (see \cite{Niederreiter}). In \cite{Kane}, Kane constructs a stream cipher using continued fractions. A rather indirect connection to cryptography is the shared interest in lattice reduction algorithms. As described in Section~\ref{sec:LLL}, these algorithms compute relatively short vectors of lattices. Various cryptosystems such as Frodo, Dilithium and Kyber are based on the problem of finding short vectors in lattices. Hence, the study of lattice reduction algorithms such as the LLL algorithm is important for the security analysis of lattice-based cryptosystems.

We also note that due to its simplicity, the Brun algorithm  appears in various application fields. See \cite{Rooij},  where efficient exponentiation using  addition chains is used.  Note that continued fractions   were already   used for addition chains; see e.g. \cite{BBBD:89}. See also  the  survey \cite{WLR}  for  application of lattice reduction and of the Brun algorithm  in wireless communications and statistical signal processing. 

We now consider  some applications in  (symbolic)  dynamics with the    seminal  example 
of Sturmian dynamical systems, introduced by Morse and  Hedlund   in \cite{Morse2}. There is an impressive    literature devoted to their study
and to their possible generalizations in   word combinatorics, and   in  digital geometry \cite{Rosenfeld}.
This is due to  several factors. 
They provide    symbolic  codings  for  the simplest   arithmetic   systems, namely   the   irrational   translations on the circle, 
they  also code discrete lines and  are  unidimensional   models  of quasicrystals \cite{AperBook}.
Moreover,  the  scale invariance  of Sturmian dynamical  systems
allows their description  using   a renormalization scheme  governed by  usual  continued fractions  via  the geodesic flow acting  on  the modular surface. Renormalization schemes  can  often be interpreted  as continued fractions \cite{Yoccoz}. 
 This was     crystallized later    with   the study  of    interval exchanges in  relation with 
 the Teichm\"uller flow,  through   the work, among others, of Veech,    Masur,  Yoccoz and  Avila. 
Similarly   as  for continued fractions, there is no canonical generalization of  Sturmian  words. 
Episturmian  words, also called Arnoux-Rauzy words,  have  attracted a lot of attention. See in particular
  \cite{Cassaigne-Ferenczi-Messaoudi:08,AHS:16,AHS:16bis}
 for the study of   associated continued fractions  on the Rauzy gasket. 
 It has been a long-standing problem to find good symbolic codings for translations on the $d$-dimensional torus that enjoy the beautiful properties of Sturmian sequences. Symbolic codings in terms of multidimensional continued fraction algorithms  are  defined in \cite{BST22}. In particular, given any exponentially convergent continued fraction algorithm, these sequences lead to renormalization schemes which produce symbolic codings of toral translations. This yields symbolic codings for almost every translation of ${\mathbb T}^2$  \cite{BST22},   and for almost all $3$-dimensional toral translations, paving the way for the development of equidistribution results for the associated  Kronecker sequences.  These results rely on  the   strong properties provided by
 the fact that the  second  Lyapunov exponent is positive;  this explains why   they hold only in small dimensions.

 In view of  the  non-positivity of  the second Lyapunov exponent, there is a  need to design strongly convergent continued fractions  algorithms in high dimensions.
   In fact,   positivity of the second Lyapunov exponent  has only been proved in  dimension $d=2$ for classical algorithms, $d=3$ for the Brun algorithm.
    One  natural approach  for the design of continued fractions consists in trying to derive continued fraction algorithms from  
lattice reduction algorithms using the fact that they compute short vectors and that they reach Dirichlet's bound (up to a constant depending exponentially on the dimension). The design   of continued  fractions  will  go along with   the   dynamical modeling of   lattice reduction algorithms  (and more specifically  LLL) for their probabilistic analysis.  There is  also a need to improve existing algorithms  of a dynamical nature
by taking advantage of their  dynamical  properties. 

For other types of number expansions, in particular for $\beta$-expansions, it turned out to be very fruitful to introduce a random dynamical system to produce the expansions. Where a deterministic system is defined as a pair $(X,T)$, a {\em random dynamical system} makes use of a family of transformations $(T_i: X \to X)_{i \in I}$ (for some index set $I$) all defined on the same domain $X$, where each map is chosen with a certain probability. One then studies the compositions of the form
\[ T_{i_n} \circ \cdots \circ T_{i_1} (x), \quad i_j \in I,\]
instead of $T^n(x)$. For $\beta$-expansions, i.e., expansions of real numbers $x$ of the form $\sum_{k \ge 1} \frac{b_k}{\beta^k}$ with $\beta>1$ a non-integer and each $b_i$ an integer between 0 and $\beta$, a corresponding random dynamical system was first introduced in \cite{DKrandom}. Applications of random $\beta$-expansions with respect to rational approximations of real numbers were then described by Daubechies et al.~\cite{daubechies06}, see also \cite{daubechies10,kohda12}, in relation to analog-to-digital conversion and in \cite{JMKA,JM} in relation to random number generation. Random one-dimensional continued fraction algorithms and their invariant measures were studied in \cite{KKV17,bahsoun20,dajkalmag,taylorcrush21,KMTV22}. Often, in terms of dynamics and the corresponding approximation properties, such a random system performs comparably to the best performing deterministic system present in the family $(T_i: X \to X)_{i \in I}$, but with the added flexibility that real numbers now have many different expansions assigned to them. Even though it is not quite clear at the moment whether placing multidimensional continued fraction algorithms in a random framework could yield similar advantages, it might be of interest to study this further.

\section{Improving Jacobi--Perron algorithm}\label{sec:ni}

In 1981 both Ito and Tanaka in \cite{itotanaka} and Nakada in \cite{Nak} introduced the notion of $\alpha$-continued fractions for $\alpha\in [1/2, 1]$. This was done by replacing $\lfloor \frac{1}{x}\rfloor$ in the definition of the Gauss map $T_G$ by $\lfloor \frac{1}{x}+1-\alpha\rfloor$ in \cite{itotanaka} and by
$\lfloor \frac{1}{|x|}+1-\alpha\rfloor$ in \cite{Nak}. For the setup of Nakada the convergents of the corresponding $\alpha$-continued fraction of a point form a sub equence of its regular convergents, and hence provide better approximations.  In particular, for $\alpha\in [\frac{1}{2}, \frac{\sqrt 5-1}{2}]$, all the corresponding $\alpha$-continued fraction algorithms are isomorphic and provide better approximations than the ones given by the regular continued fractions.
See in particular \cite{KSS}  for results on the entropy. 
One can use a similar idea to
improve the convergence properties of the Jacobi--Perron algorithm.  To keep the exposition simple we will consider the nearest integer case, corresponding to $\alpha=\frac{1}{2}$, and we will not take absolute values inside the floor function (so the digits generated could be negative).

It is well known that when it comes to convergence speed, in one dimension the nearest integer algorithm performs best.  See e.g.~\cite{BDV}. Dynamically the algorithm is given by the map $T(x) = \frac1x - ||| \frac1x |||$, where $||| \cdot |||$ denotes the distance to the nearest integer as before. The map is well defined on the interval $[-\frac12, \frac12]$. The possible continued fraction digits for this algorithm are all integers $n$ with $|n|\ge 2$. In higher dimensions  a nearest integer version of the Jacobi--Perron algorithm has attracted some attention. See  Section  \ref{subsec:expnijpa} for numerical values
for the nearest integer Jacobi--Perron algorithm from  \cite{Steiner:pc}.
We also recall that the Markov conditions  on the digits produced by the classical Jacobi--Perron algorithm have a  very simple form and that the  piecewise analyticity of  the density of its  invariant measure has been proved  in  \cite{BROISEJP}.

\subsection{First dynamical properties of the algorithm}

Let $d \ge 2$. The definition of the {\em $d$-dimensional nearest integer Jacobi--Perron map (NIJP)} on $C=\big[-\frac12, \frac12\big]^d$ is given by 
\[
T_0 (x_1, \ldots, x_d) = \left( \frac{x_2}{x_1} - \left\lfloor \frac{x_2}{x_1} + \frac12 \right\rfloor, \ldots , \frac{x_d}{x_1} - \left\lfloor \frac{x_d}{x_1} + \frac12 \right\rfloor, \frac1{x_1} - \left\lfloor \frac1{x_1} + \frac12 \right\rfloor \right)\]
and can be used to create for each $d$-tuple $(x_1, \ldots, x_d) \in C$ a sequence of continued fraction approximations with the same denominator.

The  matrix version of $T_0$ is therefore defined on 
$$\Lambda  = \{(y_0, y_1, \ldots ,y_d) \in \mathbb{R}^{d+1}\setminus\{\mathbf{0}\}:\, y_0 \ge y_1 , \cdots ,y_d \ge 0\}$$ as

\[
(y_0,y_1,\ldots,y_d) \mapsto \Big(y_1, y_2 - \Big\lfloor\frac{y_2}{y_1}+\frac12\Big\rfloor y_1, \ldots, y_d - \Big\lfloor\frac{y_d}{y_1}+\frac12\Big\rfloor y_1, y_0 - \Big\lfloor\frac{y_0}{y_1}+\frac12\Big\rfloor  y_1\Big),
\]
and we have $$\tra{}{(y_0,y_1, \cdots, y_d)}=  \tra{A_0}(y_0,y_1,\ldots,y_d) \tra{\Big(y_1, y_2 - \Big\lfloor\frac{y_2}{y_1}+\frac12\Big\rfloor y_1, \ldots, y_d - \Big\lfloor\frac{y_d}{y_1}+\frac12\Big\rfloor y_1, y_0 - \Big\lfloor\frac{y_0}{y_1}+\frac12\Big\rfloor  y_1\Big)}$$ with 
\[
\tra{A_0}(y_0,y_1,\ldots,y_d) = \begin{pmatrix}\lfloor\frac{1}{y_1}\rfloor & 1 & \lfloor\frac{y_2}{y_1}+\frac12\rfloor & \cdots & \lfloor\frac{y_{d-1}}{y_d}+\frac12\rfloor \\ 0&0&1&\cdots&0 \\ \vdots&\vdots&\ddots&\ddots&\vdots \\ 0&0&\cdots&\cdots&1 \\ 1&0&\cdots&\cdots&0\end{pmatrix}.
\]
Iterates of $T_0$ produce a matrix
$$
 \begin{pmatrix}q_0^{(n)} & p_{0,1}^{(n)} & \cdots & p_{0,d}^{(n)} \\ q_1^{(n)} & p_{1,1}^{(n)} & \cdots & p_{1,d}^{(n)} \\ \vdots & \vdots & \ddots & \vdots \\ q_d^{(n)} & p_{d,1}^{(n)} & \cdots & p_{d,d}^{(n)}\end{pmatrix}.
$$
Note that this is not  the same matrix as in (\ref{eq:MAT}). This is due to the fact that we use a different normalization  for technical reasons.

For $\mathbf x = (x_1, \ldots, x_d) \in C$ set
\[ a= a (\mathbf x) = \left| \left| \left|  \frac1{x_1} \right| \right| \right|, \quad b^{(i)} = b^{(i)} (\mathbf x) = \left| \left| \left| \frac{x_i}{x_1} \right| \right| \right|, \, 2 \le i \le d.\]
The functions $a$ and $b^{(i)}$ are piecewise constant on $C$. To be precise, for each $\mathbf x \in C$ it holds that $a(\mathbf x) =k$, $k \in \mathbb Z$, if and only if $\frac2{2k+1} < x_1 \le \frac2{2k-1}$ and for each $2 \le i \le d$ it holds that $b^{(i)}(\mathbf x) = k$ if and only if
\[ \begin{cases}
x_1\left(k-\frac12\right) \le x_2 < x_1 \left(k+\frac12 \right), &\text{if } x_1 >0,\\
x_1\left(k+\frac12\right) < x_2 \le x_1 \left(k- \frac12 \right), &\text{if } x_1 <0.
\end{cases}\]
Hence, if we let
\[ C_{a,b_2, \ldots, b_d} = \left\{ \mathbf x \in C \, : \, a(\mathbf x) = a, \, b^{(i)}(\mathbf x) = b_i, \, 2\le i \le d \right\},  \]
then the collection $\mathcal C = \{ C_{a,b_2, \ldots, b_d} \}$ yields a partition of $C$, the elements of which are called the {\em cylinder sets} of $T_0$. Figure~\ref{f:partition2} shows the partition $\mathcal C$ of $C$ for $d=2$. %The cylinder sets are coded by the {\em digits} given by the map $T_0$. These are the numbers in Figure~\ref{f:partition2}, they indicate the values of the floor functions in both coordinates of $T_0$ on that cylinder set. We see that possible digits are $(a,b)$ with $a,b \in \mathbb Z$ and $|a|\ge 2$ and $|b|\le \lceil \frac{|a|}{2}\rceil$. We denote the elements of $\mathcal C$ by
%\[  [(a,b)] \, : \, a,b \in \mathbb Z, \, |a|\ge 2, \, 0\le |b| \le \lceil \frac{a}{2}\rceil \}.\]

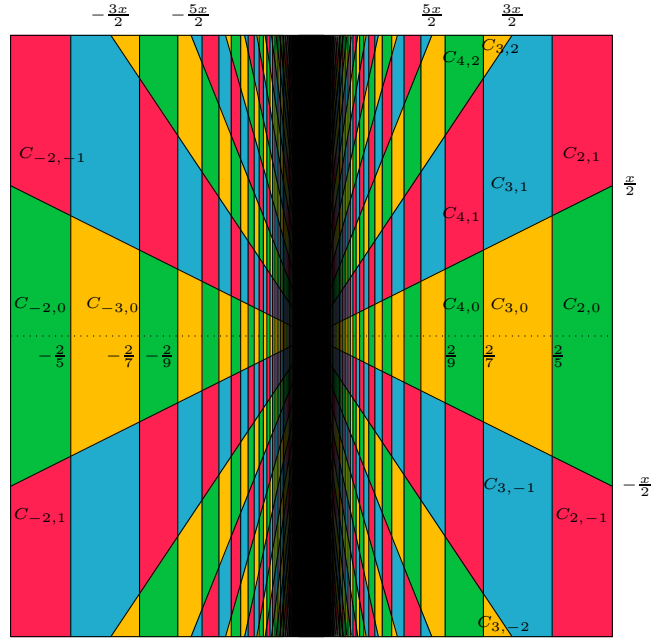
\begin{figure}[H]
\begin{center}
\begin{tikzpicture}[scale=8]

\filldraw[fill=awesome, draw=awesome] (1/2,1/2)--(2/5,1/2)--(2/5,1/5)--(1/2,1/4)--(1/2,1/2);
\filldraw[fill=awesome, draw=awesome] (1/2,-1/2)--(2/5,-1/2)--(2/5,-1/5)--(1/2,-1/4)--(1/2,-1/2);
\filldraw[fill=awesome, draw=awesome] (-1/2,1/2)--(-2/5,1/2)--(-2/5,1/5)--(-1/2,1/4)--(-1/2,1/2);
\filldraw[fill=awesome, draw=awesome] (-1/2,-1/2)--(-2/5,-1/2)--(-2/5,-1/5)--(-1/2,-1/4)--(-1/2,-1/2);
\filldraw[fill=awesome, draw=awesome] (2/11,1/2)--(2/13,1/2)--(2/13,5/13)--(2/11,5/11)--cycle;
\filldraw[fill=awesome, draw=awesome] (2/19,1/2)--(2/21,1/2)--(2/21,9/21)--(2/19,9/19)--cycle;
\filldraw[fill=awesome, draw=awesome] (2/27,1/2)--(2/29,1/2)--(2/29,13/29)--(2/27,13/27)--cycle;
\filldraw[fill=awesome, draw=awesome] (2/35,1/2)--(2/37,1/2)--(2/37,17/37)--(2/35,17/35)--cycle;
\filldraw[fill=awesome, draw=awesome] (2/11,-1/2)--(2/13,-1/2)--(2/13,-5/13)--(2/11,-5/11)--cycle;
\filldraw[fill=awesome, draw=awesome] (2/19,-1/2)--(2/21,-1/2)--(2/21,-9/21)--(2/19,-9/19)--cycle;
\filldraw[fill=awesome, draw=awesome] (2/27,-1/2)--(2/29,-1/2)--(2/29,-13/29)--(2/27,-13/27)--cycle;
\filldraw[fill=awesome, draw=awesome] (2/35,-1/2)--(2/37,-1/2)--(2/37,-17/37)--(2/35,-17/35)--cycle;
\filldraw[fill=awesome, draw=awesome] (-2/11,1/2)--(-2/13,1/2)--(-2/13,5/13)--(-2/11,5/11)--cycle;
\filldraw[fill=awesome, draw=awesome] (-2/19,1/2)--(-2/21,1/2)--(-2/21,9/21)--(-2/19,9/19)--cycle;
\filldraw[fill=awesome, draw=awesome] (-2/27,1/2)--(-2/29,1/2)--(-2/29,13/29)--(-2/27,13/27)--cycle;
\filldraw[fill=awesome, draw=awesome] (-2/35,1/2)--(-2/37,1/2)--(-2/37,17/37)--(-2/35,17/35)--cycle;
\filldraw[fill=awesome, draw=awesome] (-2/11,-1/2)--(-2/13,-1/2)--(-2/13,-5/13)--(-2/11,-5/11)--cycle;
\filldraw[fill=awesome, draw=awesome] (-2/19,-1/2)--(-2/21,-1/2)--(-2/21,-9/21)--(-2/19,-9/19)--cycle;
\filldraw[fill=awesome, draw=awesome] (-2/27,-1/2)--(-2/29,-1/2)--(-2/29,-13/29)--(-2/27,-13/27)--cycle;
\filldraw[fill=awesome, draw=awesome] (-2/35,-1/2)--(-2/37,-1/2)--(-2/37,-17/37)--(-2/35,-17/35)--cycle;

\filldraw[fill=amber, draw=amber] (2/9,3/9)--(2/9,1/2)--(1/5,1/2)--(2/11,5/11)--(2/11,3/11)--(2/9,3/9);
\filldraw[fill=amber, draw=amber] (2/9,-3/9)--(2/9,-1/2)--(1/5,-1/2)--(2/11,-5/11)--(2/11,-3/11)--(2/9,-3/9);
\filldraw[fill=amber, draw=amber] (-2/9,3/9)--(-2/9,1/2)--(-1/5,1/2)--(-2/11,5/11)--(-2/11,3/11)--(-2/9,3/9);
\filldraw[fill=amber, draw=amber] (-2/9,-3/9)--(-2/9,-1/2)--(-1/5,-1/2)--(-2/11,-5/11)--(-2/11,-3/11)--(-2/9,-3/9);
\filldraw[fill=amber, draw=amber] (2/17,7/17)--(2/17,1/2)--(1/9,1/2)--(2/19,9/19)--(2/19,7/19)--cycle;
\filldraw[fill=amber, draw=amber] (2/25,11/25)--(2/25,1/2)--(1/13,1/2)--(2/27,13/27)--(2/27,11/27)--cycle;
\filldraw[fill=amber, draw=amber] (2/33,15/33)--(2/33,1/2)--(1/17,1/2)--(2/35,17/35)--(2/35,15/35)--cycle;
\filldraw[fill=amber, draw=amber] (2/41,19/41)--(2/41,1/2)--(1/21,1/2)--(2/43,21/43)--(2/43,19/43)--cycle;
\filldraw[fill=amber, draw=amber] (2/17,-7/17)--(2/17,-1/2)--(1/9,-1/2)--(2/19,-9/19)--(2/19,-7/19)--cycle;
\filldraw[fill=amber, draw=amber] (2/25,-11/25)--(2/25,-1/2)--(1/13,-1/2)--(2/27,-13/27)--(2/27,-11/27)--cycle;
\filldraw[fill=amber, draw=amber] (2/33,-15/33)--(2/33,-1/2)--(1/17,-1/2)--(2/35,-17/35)--(2/35,-15/35)--cycle;
\filldraw[fill=amber, draw=amber] (2/41,-19/41)--(2/41,-1/2)--(1/21,-1/2)--(2/43,-21/43)--(2/43,-19/43)--cycle;

\filldraw[fill=amber, draw=amber] (-2/17,7/17)--(-2/17,1/2)--(-1/9,1/2)--(-2/19,9/19)--(-2/19,7/19)--cycle;
\filldraw[fill=amber, draw=amber] (-2/25,11/25)--(-2/25,1/2)--(-1/13,1/2)--(-2/27,13/27)--(-2/27,11/27)--cycle;
\filldraw[fill=amber, draw=amber] (-2/33,15/33)--(-2/33,1/2)--(-1/17,1/2)--(-2/35,17/35)--(-2/35,15/35)--cycle;
\filldraw[fill=amber, draw=amber] (-2/41,19/41)--(-2/41,1/2)--(-1/21,1/2)--(-2/43,21/43)--(-2/43,19/43)--cycle;
\filldraw[fill=amber, draw=amber] (-2/17,-7/17)--(-2/17,-1/2)--(-1/9,-1/2)--(-2/19,-9/19)--(-2/19,-7/19)--cycle;
\filldraw[fill=amber, draw=amber] (-2/25,-11/25)--(-2/25,-1/2)--(-1/13,-1/2)--(-2/27,-13/27)--(-2/27,-11/27)--cycle;
\filldraw[fill=amber, draw=amber] (-2/33,-15/33)--(-2/33,-1/2)--(-1/17,-1/2)--(-2/35,-17/35)--(-2/35,-15/35)--cycle;
\filldraw[fill=amber, draw=amber] (-2/41,-19/41)--(-2/41,-1/2)--(-1/21,-1/2)--(-2/43,-21/43)--(-2/43,-19/43)--cycle;

\filldraw[fill=ballblue, draw=ballblue] (2/5,1/2)--(1/3,1/2)--(2/7,3/7)--(2/7,1/7)--(2/5,1/5)--(2/5,1/2);
\filldraw[fill=ballblue, draw=ballblue] (2/5,-1/2)--(1/3,-1/2)--(2/7,-3/7)--(2/7,-1/7)--(2/5,-1/5)--(2/5,-1/2);
\filldraw[fill=ballblue, draw=ballblue] (-2/5,1/2)--(-1/3,1/2)--(-2/7,3/7)--(-2/7,1/7)--(-2/5,1/5)--(-2/5,1/2);
\filldraw[fill=ballblue, draw=ballblue] (-2/5,-1/2)--(-1/3,-1/2)--(-2/7,-3/7)--(-2/7,-1/7)--(-2/5,-1/5)--(-2/5,-1/2);
\filldraw[fill=ballblue, draw=ballblue] (2/13,1/2)--(1/7,1/2)--(2/15,7/15)--(2/15,5/15)--(2/13,5/13)--cycle;
\filldraw[fill=ballblue, draw=ballblue] (2/21,1/2)--(1/11,1/2)--(2/23,11/23)--(2/23,9/23)--(2/21,7/21)--cycle;
\filldraw[fill=ballblue, draw=ballblue] (2/29,1/2)--(1/15,1/2)--(2/31,15/31)--(2/31,13/31)--(2/29,13/29)--cycle;
\filldraw[fill=ballblue, draw=ballblue] (2/13,-1/2)--(1/7,-1/2)--(2/15,-7/15)--(2/15,-5/15)--(2/13,-5/13)--cycle;
\filldraw[fill=ballblue, draw=ballblue] (2/21,-1/2)--(1/11,-1/2)--(2/23,-11/23)--(2/23,-9/23)--(2/21,-7/21)--cycle;
\filldraw[fill=ballblue, draw=ballblue] (2/29,-1/2)--(1/15,-1/2)--(2/31,-15/31)--(2/31,-13/31)--(2/29,-13/29)--cycle;
\filldraw[fill=ballblue, draw=ballblue] (-2/13,1/2)--(-1/7,1/2)--(-2/15,7/15)--(-2/15,5/15)--(-2/13,5/13)--cycle;
\filldraw[fill=ballblue, draw=ballblue] (-2/21,1/2)--(-1/11,1/2)--(-2/23,11/23)--(-2/23,9/23)--(-2/21,7/21)--cycle;
\filldraw[fill=ballblue, draw=ballblue] (-2/29,1/2)--(-1/15,1/2)--(-2/31,15/31)--(-2/31,13/31)--(-2/29,13/29)--cycle;
\filldraw[fill=ballblue, draw=ballblue] (-2/13,-1/2)--(-1/7,-1/2)--(-2/15,-7/15)--(-2/15,-5/15)--(-2/13,-5/13)--cycle;
\filldraw[fill=ballblue, draw=ballblue] (-2/21,-1/2)--(-1/11,-1/2)--(-2/23,-11/23)--(-2/23,-9/23)--(-2/21,-7/21)--cycle;
\filldraw[fill=ballblue, draw=ballblue] (-2/29,-1/2)--(-1/15,-1/2)--(-2/31,-15/31)--(-2/31,-13/31)--(-2/29,-13/29)--cycle;
\filldraw[fill=ballblue, draw=ballblue] (2/37,1/2)--(1/19,1/2)--(2/39,19/39)--(2/39,17/39)--(2/37,17/37)--cycle;
\filldraw[fill=ballblue, draw=ballblue] (2/37,-1/2)--(1/19,-1/2)--(2/39,-19/39)--(2/39,-17/39)--(2/37,-17/37)--cycle;
\filldraw[fill=ballblue, draw=ballblue] (-2/37,1/2)--(-1/19,1/2)--(-2/39,19/39)--(-2/39,17/39)--(-2/37,17/37)--cycle;
\filldraw[fill=ballblue, draw=ballblue] (-2/37,-1/2)--(-1/19,-1/2)--(-2/39,-19/39)--(-2/39,-17/39)--(-2/37,-17/37)--cycle;

\filldraw[fill=amber, draw=amber] (1/3,1/2)--(2/7,1/2)--(2/7,3/7);
\filldraw[fill=amber, draw=amber] (1/3,-1/2)--(2/7,-1/2)--(2/7,-3/7);
\filldraw[fill=amber, draw=amber] (-1/3,1/2)--(-2/7,1/2)--(-2/7,3/7);
\filldraw[fill=amber, draw=amber] (-1/3,-1/2)--(-2/7,-1/2)--(-2/7,-3/7);

\filldraw[fill=pastelgreen, draw=pastelgreen] (1/2,1/4)--(1/2,-1/4)--(2/5,-1/5)--(2/5,1/5)--(1/2,1/4);
\filldraw[fill=pastelgreen, draw=pastelgreen] (-1/2,1/4)--(-1/2,-1/4)--(-2/5,-1/5)--(-2/5,1/5)--(-1/2,1/4);

\filldraw[fill=pastelgreen, draw=pastelgreen] (2/7,1/2)--(2/9,1/2)--(2/9,3/9)--(2/7,3/7)--(2/7,1/2);
\filldraw[fill=pastelgreen, draw=pastelgreen] (2/7,-1/2)--(2/9,-1/2)--(2/9,-3/9)--(2/7,-3/7)--(2/7,-1/2);
\filldraw[fill=pastelgreen, draw=pastelgreen] (-2/7,1/2)--(-2/9,1/2)--(-2/9,3/9)--(-2/7,3/7)--(-2/7,1/2);
\filldraw[fill=pastelgreen, draw=pastelgreen] (-2/7,-1/2)--(-2/9,-1/2)--(-2/9,-3/9)--(-2/7,-3/7)--(-2/7,-1/2);
\filldraw[fill=pastelgreen, draw=pastelgreen] (2/15,1/2)--(2/17,1/2)--(2/17,7/17)--(2/15,7/15)--cycle;
\filldraw[fill=pastelgreen, draw=pastelgreen] (2/23,1/2)--(2/25,1/2)--(2/25,11/25)--(2/23,11/23)--cycle;
\filldraw[fill=pastelgreen, draw=pastelgreen] (2/31,1/2)--(2/33,1/2)--(2/33,15/33)--(2/31,15/31)--cycle;
\filldraw[fill=pastelgreen, draw=pastelgreen] (2/39,1/2)--(2/41,1/2)--(2/41,19/41)--(2/31,19/39)--cycle;
\filldraw[fill=pastelgreen, draw=pastelgreen] (2/15,-1/2)--(2/17,-1/2)--(2/17,-7/17)--(2/15,-7/15)--cycle;
\filldraw[fill=pastelgreen, draw=pastelgreen] (2/23,-1/2)--(2/25,-1/2)--(2/25,-11/25)--(2/23,-11/23)--cycle;
\filldraw[fill=pastelgreen, draw=pastelgreen] (2/31,-1/2)--(2/33,-1/2)--(2/33,-15/33)--(2/31,-15/31)--cycle;
\filldraw[fill=pastelgreen, draw=pastelgreen] (2/39,-1/2)--(2/41,-1/2)--(2/41,-19/41)--(2/31,-19/39)--cycle;
\filldraw[fill=pastelgreen, draw=pastelgreen] (-2/15,1/2)--(-2/17,1/2)--(-2/17,7/17)--(-2/15,7/15)--cycle;
\filldraw[fill=pastelgreen, draw=pastelgreen] (-2/23,1/2)--(-2/25,1/2)--(-2/25,11/25)--(-2/23,11/23)--cycle;
\filldraw[fill=pastelgreen, draw=pastelgreen] (-2/31,1/2)--(-2/33,1/2)--(-2/33,15/33)--(-2/31,15/31)--cycle;
\filldraw[fill=pastelgreen, draw=pastelgreen] (-2/39,1/2)--(-2/41,1/2)--(-2/41,19/41)--(-2/31,19/39)--cycle;
\filldraw[fill=pastelgreen, draw=pastelgreen] (-2/15,-1/2)--(-2/17,-1/2)--(-2/17,-7/17)--(-2/15,-7/15)--cycle;
\filldraw[fill=pastelgreen, draw=pastelgreen] (-2/23,-1/2)--(-2/25,-1/2)--(-2/25,-11/25)--(-2/23,-11/23)--cycle;
\filldraw[fill=pastelgreen, draw=pastelgreen] (-2/31,-1/2)--(-2/33,-1/2)--(-2/33,-15/33)--(-2/31,-15/31)--cycle;
\filldraw[fill=pastelgreen, draw=pastelgreen] (-2/39,-1/2)--(-2/41,-1/2)--(-2/41,-19/41)--(-2/31,-19/39)--cycle;

\filldraw[fill=ballblue, draw=ballblue] (1/5,1/2)--(2/11,1/2)--(2/11,5/11)--cycle;
\filldraw[fill=ballblue, draw=ballblue] (1/5,-1/2)--(2/11,-1/2)--(2/11,-5/11)--cycle;
\filldraw[fill=ballblue, draw=ballblue] (-1/5,1/2)--(-2/11,1/2)--(-2/11,5/11)--cycle;
\filldraw[fill=ballblue, draw=ballblue] (-1/5,-1/2)--(-2/11,-1/2)--(-2/11,-5/11)--cycle;
\filldraw[fill=ballblue, draw=ballblue] (1/9,1/2)--(2/19,1/2)--(2/19,9/19)--cycle;
\filldraw[fill=ballblue, draw=ballblue] (1/9,-1/2)--(2/19,-1/2)--(2/19,-9/19)--cycle;
\filldraw[fill=ballblue, draw=ballblue] (-1/9,1/2)--(-2/19,1/2)--(-2/19,9/19)--cycle;
\filldraw[fill=ballblue, draw=ballblue] (-1/9,-1/2)--(-2/19,-1/2)--(-2/19,-9/19)--cycle;
\filldraw[fill=ballblue, draw=ballblue] (1/13,1/2)--(2/27,1/2)--(2/27,13/27)--cycle;
\filldraw[fill=ballblue, draw=ballblue] (1/13,-1/2)--(2/27,-1/2)--(2/27,-13/27)--cycle;
\filldraw[fill=ballblue, draw=ballblue] (-1/13,1/2)--(-2/27,1/2)--(-2/27,13/27)--cycle;
\filldraw[fill=ballblue, draw=ballblue] (-1/13,-1/2)--(-2/27,-1/2)--(-2/27,-13/27)--cycle;

\filldraw[fill=amber, draw=amber] (1/7,1/2)--(2/15,1/2)--(2/15,7/15)--cycle;
\filldraw[fill=amber, draw=amber] (1/11,1/2)--(2/23,1/2)--(2/23,11/23)--cycle;
\filldraw[fill=amber, draw=amber] (1/15,1/2)--(2/31,1/2)--(2/31,15/31)--cycle;
\filldraw[fill=amber, draw=amber] (1/7,-1/2)--(2/15,-1/2)--(2/15,-7/15)--cycle;
\filldraw[fill=amber, draw=amber] (1/11,-1/2)--(2/23,-1/2)--(2/23,-11/23)--cycle;
\filldraw[fill=amber, draw=amber] (1/15,-1/2)--(2/31,-1/2)--(2/31,-15/31)--cycle;
\filldraw[fill=amber, draw=amber] (-1/7,1/2)--(-2/15,1/2)--(-2/15,7/15)--cycle;
\filldraw[fill=amber, draw=amber] (-1/11,1/2)--(-2/23,1/2)--(-2/23,11/23)--cycle;
\filldraw[fill=amber, draw=amber] (-1/15,1/2)--(-2/31,1/2)--(-2/31,15/31)--cycle;
\filldraw[fill=amber, draw=amber] (-1/7,-1/2)--(-2/15,-1/2)--(-2/15,-7/15)--cycle;
\filldraw[fill=amber, draw=amber] (-1/11,-1/2)--(-2/23,-1/2)--(-2/23,-11/23)--cycle;
\filldraw[fill=amber, draw=amber] (-1/15,-1/2)--(-2/31,-1/2)--(-2/31,-15/31)--cycle;

\foreach \i in {2,...,12}
{
\filldraw[fill=awesome, draw=awesome] ({2/(4*\i-1)},{3/(4*\i-1)})--({2/(4*\i-1)},{1/(4*\i-1)})--({2/(4*\i+1)},{1/(4*\i+1)})--({2/(4*\i+1)},{3/(4*\i+1)})--({2/(4*\i-1)},{3/(4*\i-1)});
\filldraw[fill=awesome, draw=awesome] ({2/(4*\i-1)},{-3/(4*\i-1)})--({2/(4*\i-1)},{-1/(4*\i-1)})--({2/(4*\i+1)},{-1/(4*\i+1)})--({2/(4*\i+1)},{-3/(4*\i+1)})--({2/(4*\i-1)},{-3/(4*\i-1)});
\filldraw[fill=awesome, draw=awesome] ({-2/(4*\i-1)},{3/(4*\i-1)})--({-2/(4*\i-1)},{1/(4*\i-1)})--({-2/(4*\i+1)},{1/(4*\i+1)})--({-2/(4*\i+1)},{3/(4*\i+1)})--({-2/(4*\i-1)},{3/(4*\i-1)});
\filldraw[fill=awesome, draw=awesome] ({-2/(4*\i-1)},{-3/(4*\i-1)})--({-2/(4*\i-1)},{-1/(4*\i-1)})--({-2/(4*\i+1)},{-1/(4*\i+1)})--({-2/(4*\i+1)},{-3/(4*\i+1)})--({-2/(4*\i-1)},{-3/(4*\i-1)});

\filldraw[fill=ballblue, draw=ballblue] ({2/(4*\i+1)},{3/(4*\i+1)})--({2/(4*\i+1)},{1/(4*\i+1)})--({2/(4*\i+3)},{1/(4*\i+3)})--({2/(4*\i+3)},{3/(4*\i+3)})--({2/(4*\i+1)},{3/(4*\i+1)});
\filldraw[fill=ballblue, draw=ballblue] ({2/(4*\i+1)},{-3/(4*\i+1)})--({2/(4*\i+1)},{-1/(4*\i+1)})--({2/(4*\i+3)},{-1/(4*\i+3)})--({2/(4*\i+3)},{-3/(4*\i+3)})--({2/(4*\i+1)},{-3/(4*\i+1)});
\filldraw[fill=ballblue, draw=ballblue] ({-2/(4*\i+1)},{3/(4*\i+1)})--({-2/(4*\i+1)},{1/(4*\i+1)})--({-2/(4*\i+3)},{1/(4*\i+3)})--({-2/(4*\i+3)},{3/(4*\i+3)})--({-2/(4*\i+1)},{3/(4*\i+1)});
\filldraw[fill=ballblue, draw=ballblue] ({-2/(4*\i+1)},{-3/(4*\i+1)})--({-2/(4*\i+1)},{-1/(4*\i+1)})--({-2/(4*\i+3)},{-1/(4*\i+3)})--({-2/(4*\i+3)},{-3/(4*\i+3)})--({-2/(4*\i+1)},{-3/(4*\i+1)});

\filldraw[fill=pastelgreen, draw=pastelgreen] ({2/(4*\i-1)},{-1/(4*\i-1)})--({2/(4*\i-1)},{1/(4*\i-1)})--({2/(4*\i+1)},{1/(4*\i+1)})--({2/(4*\i+1)},{-1/(4*\i+1)})--({2/(4*\i-1)},{-1/(4*\i-1)});
\filldraw[fill=pastelgreen, draw=pastelgreen] ({-2/(4*\i-1)},{-1/(4*\i-1)})--({-2/(4*\i-1)},{1/(4*\i-1)})--({-2/(4*\i+1)},{1/(4*\i+1)})--({-2/(4*\i+1)},{-1/(4*\i+1)})--({-2/(4*\i-1)},{-1/(4*\i-1)});

\filldraw[fill=amber, draw=amber] ({2/(4*\i-1)},{-1/(4*\i-1)})--({2/(4*\i-1)},{1/(4*\i-1)})--({2/(4*\i-3)},{1/(4*\i-3)})--({2/(4*\i-3)},{-1/(4*\i-3)})--({2/(4*\i-1)},{-1/(4*\i-1)});
\filldraw[fill=amber, draw=amber] ({-2/(4*\i-1)},{-1/(4*\i-1)})--({-2/(4*\i-1)},{1/(4*\i-1)})--({-2/(4*\i-3)},{1/(4*\i-3)})--({-2/(4*\i-3)},{-1/(4*\i-3)})--({-2/(4*\i-1)},{-1/(4*\i-1)});

\filldraw[fill=amber, draw=amber] ({2/(4*\i+5)},{3/(4*\i+5)})--({2/(4*\i+5)},{5/(4*\i+5)})--({2/(4*\i+7)},{5/(4*\i+7)})--({2/(4*\i+7)},{3/(4*\i+7)})--({2/(4*\i+5)},{3/(4*\i+5)});
\filldraw[fill=amber, draw=amber] ({2/(4*\i+5)},{-3/(4*\i+5)})--({2/(4*\i+5)},{-5/(4*\i+5)})--({2/(4*\i+7)},{-5/(4*\i+7)})--({2/(4*\i+7)},{-3/(4*\i+7)})--({2/(4*\i+5)},{-3/(4*\i+5)});
\filldraw[fill=amber, draw=amber] ({-2/(4*\i+5)},{3/(4*\i+5)})--({-2/(4*\i+5)},{5/(4*\i+5)})--({-2/(4*\i+7)},{5/(4*\i+7)})--({-2/(4*\i+7)},{3/(4*\i+7)})--({-2/(4*\i+5)},{3/(4*\i+5)});
\filldraw[fill=amber, draw=amber] ({-2/(4*\i+5)},{-3/(4*\i+5)})--({-2/(4*\i+5)},{-5/(4*\i+5)})--({-2/(4*\i+7)},{-5/(4*\i+7)})--({-2/(4*\i+7)},{-3/(4*\i+7)})--({-2/(4*\i+5)},{-3/(4*\i+5)});

\filldraw[fill=pastelgreen, draw=pastelgreen] ({2/(4*\i+3)},{3/(4*\i+3)})--({2/(4*\i+3)},{5/(4*\i+3)})--({2/(4*\i+5)},{5/(4*\i+5)})--({2/(4*\i+5)},{3/(4*\i+5)})--({2/(4*\i+3)},{3/(4*\i+3)});
\filldraw[fill=pastelgreen, draw=pastelgreen] ({2/(4*\i+3)},{-3/(4*\i+3)})--({2/(4*\i+3)},{-5/(4*\i+3)})--({2/(4*\i+5)},{-5/(4*\i+5)})--({2/(4*\i+5)},{-3/(4*\i+5)})--({2/(4*\i+3)},{-3/(4*\i+3)});
\filldraw[fill=pastelgreen, draw=pastelgreen] ({-2/(4*\i+3)},{3/(4*\i+3)})--({-2/(4*\i+3)},{5/(4*\i+3)})--({-2/(4*\i+5)},{5/(4*\i+5)})--({-2/(4*\i+5)},{3/(4*\i+5)})--({-2/(4*\i+3)},{3/(4*\i+3)});
\filldraw[fill=pastelgreen, draw=pastelgreen] ({-2/(4*\i+3)},{-3/(4*\i+3)})--({-2/(4*\i+3)},{-5/(4*\i+3)})--({-2/(4*\i+5)},{-5/(4*\i+5)})--({-2/(4*\i+5)},{-3/(4*\i+5)})--({-2/(4*\i+3)},{-3/(4*\i+3)});

}

\foreach \i in {3,...,12}
{

\filldraw[fill=awesome, draw=awesome] ({2/(4*\i+3)},{5/(4*\i+3)})--({2/(4*\i+3)},{7/(4*\i+3)})--({2/(4*\i+5)},{7/(4*\i+5)})--({2/(4*\i+5)},{5/(4*\i+5)})--cycle;
\filldraw[fill=awesome, draw=awesome] ({2/(4*\i+3)},{-5/(4*\i+3)})--({2/(4*\i+3)},{-7/(4*\i+3)})--({2/(4*\i+5)},{-7/(4*\i+5)})--({2/(4*\i+5)},{-5/(4*\i+5)})--cycle;
\filldraw[fill=awesome, draw=awesome] ({-2/(4*\i+3)},{5/(4*\i+3)})--({-2/(4*\i+3)},{7/(4*\i+3)})--({-2/(4*\i+5)},{7/(4*\i+5)})--({-2/(4*\i+5)},{5/(4*\i+5)})--cycle;
\filldraw[fill=awesome, draw=awesome] ({-2/(4*\i+3)},{-5/(4*\i+3)})--({-2/(4*\i+3)},{-7/(4*\i+3)})--({-2/(4*\i+5)},{-7/(4*\i+5)})--({-2/(4*\i+5)},{-5/(4*\i+5)})--cycle;

}

\foreach \i in {5,...,12}
{

\filldraw[fill=awesome, draw=awesome] ({2/(4*\i+3)},{9/(4*\i+3)})--({2/(4*\i+3)},{11/(4*\i+3)})--({2/(4*\i+5)},{11/(4*\i+5)})--({2/(4*\i+5)},{9/(4*\i+5)})--cycle;
\filldraw[fill=awesome, draw=awesome] ({2/(4*\i+3)},{-9/(4*\i+3)})--({2/(4*\i+3)},{-11/(4*\i+3)})--({2/(4*\i+5)},{-11/(4*\i+5)})--({2/(4*\i+5)},{-9/(4*\i+5)})--cycle;
\filldraw[fill=awesome, draw=awesome] ({-2/(4*\i+3)},{9/(4*\i+3)})--({-2/(4*\i+3)},{11/(4*\i+3)})--({-2/(4*\i+5)},{11/(4*\i+5)})--({-2/(4*\i+5)},{9/(4*\i+5)})--cycle;
\filldraw[fill=awesome, draw=awesome] ({-2/(4*\i+3)},{-9/(4*\i+3)})--({-2/(4*\i+3)},{-11/(4*\i+3)})--({-2/(4*\i+5)},{-11/(4*\i+5)})--({-2/(4*\i+5)},{-9/(4*\i+5)})--cycle;

\filldraw[fill=pastelgreen, draw=pastelgreen] ({2/(4*\i-1)},{7/(4*\i-1)})--({2/(4*\i-1)},{9/(4*\i-1)})--({2/(4*\i+1)},{9/(4*\i+1)})--({2/(4*\i+1)},{7/(4*\i+1)})--cycle;
\filldraw[fill=pastelgreen, draw=pastelgreen] ({2/(4*\i-1)},{-7/(4*\i-1)})--({2/(4*\i-1)},{-9/(4*\i-1)})--({2/(4*\i+1)},{-9/(4*\i+1)})--({2/(4*\i+1)},{-7/(4*\i+1)})--cycle;
\filldraw[fill=pastelgreen, draw=pastelgreen] ({-2/(4*\i-1)},{7/(4*\i-1)})--({-2/(4*\i-1)},{9/(4*\i-1)})--({-2/(4*\i+1)},{9/(4*\i+1)})--({-2/(4*\i+1)},{7/(4*\i+1)})--cycle;
\filldraw[fill=pastelgreen, draw=pastelgreen] ({-2/(4*\i-1)},{-7/(4*\i-1)})--({-2/(4*\i-1)},{-9/(4*\i-1)})--({-2/(4*\i+1)},{-9/(4*\i+1)})--({-2/(4*\i+1)},{-7/(4*\i+1)})--cycle;

\filldraw[fill=ballblue, draw=ballblue] ({2/(4*\i-3)},{5/(4*\i-3)})--({2/(4*\i-3)},{7/(4*\i-3)})--({2/(4*\i-1)},{7/(4*\i-1)})--({2/(4*\i-1)},{5/(4*\i-1)})--cycle;
\filldraw[fill=ballblue, draw=ballblue] ({2/(4*\i-3)},{-5/(4*\i-3)})--({2/(4*\i-3)},{-7/(4*\i-3)})--({2/(4*\i-1)},{-7/(4*\i-1)})--({2/(4*\i-1)},{-5/(4*\i-1)})--cycle;
\filldraw[fill=ballblue, draw=ballblue] ({-2/(4*\i-3)},{5/(4*\i-3)})--({-2/(4*\i-3)},{7/(4*\i-3)})--({-2/(4*\i-1)},{7/(4*\i-1)})--({-2/(4*\i-1)},{5/(4*\i-1)})--cycle;
\filldraw[fill=ballblue, draw=ballblue] ({-2/(4*\i-3)},{-5/(4*\i-3)})--({-2/(4*\i-3)},{-7/(4*\i-3)})--({-2/(4*\i-1)},{-7/(4*\i-1)})--({-2/(4*\i-1)},{-5/(4*\i-1)})--cycle;

\filldraw[fill=amber, draw=amber] ({2/(4*\i+1)},{7/(4*\i+1)})--({2/(4*\i+1)},{9/(4*\i+1)})--({2/(4*\i+3)},{9/(4*\i+3)})--({2/(4*\i+3)},{7/(4*\i+3)})--cycle;
\filldraw[fill=amber, draw=amber] ({2/(4*\i+1)},{-7/(4*\i+1)})--({2/(4*\i+1)},{-9/(4*\i+1)})--({2/(4*\i+3)},{-9/(4*\i+3)})--({2/(4*\i+3)},{-7/(4*\i+3)})--cycle;
\filldraw[fill=amber, draw=amber] ({-2/(4*\i+1)},{7/(4*\i+1)})--({-2/(4*\i+1)},{9/(4*\i+1)})--({-2/(4*\i+3)},{9/(4*\i+3)})--({-2/(4*\i+3)},{7/(4*\i+3)})--cycle;
\filldraw[fill=amber, draw=amber] ({-2/(4*\i+1)},{-7/(4*\i+1)})--({-2/(4*\i+1)},{-9/(4*\i+1)})--({-2/(4*\i+3)},{-9/(4*\i+3)})--({-2/(4*\i+3)},{-7/(4*\i+3)})--cycle;
}

\foreach \i in {7,...,12}
{

\filldraw[fill=awesome, draw=awesome] ({2/(4*\i+3)},{13/(4*\i+3)})--({2/(4*\i+3)},{15/(4*\i+3)})--({2/(4*\i+5)},{15/(4*\i+5)})--({2/(4*\i+5)},{13/(4*\i+5)})--cycle;
\filldraw[fill=awesome, draw=awesome] ({2/(4*\i+3)},{-13/(4*\i+3)})--({2/(4*\i+3)},{-15/(4*\i+3)})--({2/(4*\i+5)},{-15/(4*\i+5)})--({2/(4*\i+5)},{-13/(4*\i+5)})--cycle;
\filldraw[fill=awesome, draw=awesome] ({-2/(4*\i+3)},{13/(4*\i+3)})--({-2/(4*\i+3)},{15/(4*\i+3)})--({-2/(4*\i+5)},{15/(4*\i+5)})--({-2/(4*\i+5)},{13/(4*\i+5)})--cycle;
\filldraw[fill=awesome, draw=awesome] ({-2/(4*\i+3)},{-13/(4*\i+3)})--({-2/(4*\i+3)},{-15/(4*\i+3)})--({-2/(4*\i+5)},{-15/(4*\i+5)})--({-2/(4*\i+5)},{-13/(4*\i+5)})--cycle;

\filldraw[fill=pastelgreen, draw=pastelgreen] ({2/(4*\i-1)},{11/(4*\i-1)})--({2/(4*\i-1)},{13/(4*\i-1)})--({2/(4*\i+1)},{13/(4*\i+1)})--({2/(4*\i+1)},{11/(4*\i+1)})--cycle;
\filldraw[fill=pastelgreen, draw=pastelgreen] ({2/(4*\i-1)},{-11/(4*\i-1)})--({2/(4*\i-1)},{-13/(4*\i-1)})--({2/(4*\i+1)},{-13/(4*\i+1)})--({2/(4*\i+1)},{-11/(4*\i+1)})--cycle;
\filldraw[fill=pastelgreen, draw=pastelgreen] ({-2/(4*\i-1)},{11/(4*\i-1)})--({-2/(4*\i-1)},{13/(4*\i-1)})--({-2/(4*\i+1)},{13/(4*\i+1)})--({-2/(4*\i+1)},{11/(4*\i+1)})--cycle;
\filldraw[fill=pastelgreen, draw=pastelgreen] ({-2/(4*\i-1)},{-11/(4*\i-1)})--({-2/(4*\i-1)},{-13/(4*\i-1)})--({-2/(4*\i+1)},{-13/(4*\i+1)})--({-2/(4*\i+1)},{-11/(4*\i+1)})--cycle;

\filldraw[fill=ballblue, draw=ballblue] ({2/(4*\i-3)},{9/(4*\i-3)})--({2/(4*\i-3)},{11/(4*\i-3)})--({2/(4*\i-1)},{11/(4*\i-1)})--({2/(4*\i-1)},{9/(4*\i-1)})--cycle;
\filldraw[fill=ballblue, draw=ballblue] ({2/(4*\i-3)},{-9/(4*\i-3)})--({2/(4*\i-3)},{-11/(4*\i-3)})--({2/(4*\i-1)},{-11/(4*\i-1)})--({2/(4*\i-1)},{-9/(4*\i-1)})--cycle;
\filldraw[fill=ballblue, draw=ballblue] ({-2/(4*\i-3)},{9/(4*\i-3)})--({-2/(4*\i-3)},{11/(4*\i-3)})--({-2/(4*\i-1)},{11/(4*\i-1)})--({-2/(4*\i-1)},{9/(4*\i-1)})--cycle;
\filldraw[fill=ballblue, draw=ballblue] ({-2/(4*\i-3)},{-9/(4*\i-3)})--({-2/(4*\i-3)},{-11/(4*\i-3)})--({-2/(4*\i-1)},{-11/(4*\i-1)})--({-2/(4*\i-1)},{-9/(4*\i-1)})--cycle;

\filldraw[fill=amber, draw=amber] ({2/(4*\i+1)},{11/(4*\i+1)})--({2/(4*\i+1)},{13/(4*\i+1)})--({2/(4*\i+3)},{13/(4*\i+3)})--({2/(4*\i+3)},{11/(4*\i+3)})--cycle;
\filldraw[fill=amber, draw=amber] ({2/(4*\i+1)},{-11/(4*\i+1)})--({2/(4*\i+1)},{-13/(4*\i+1)})--({2/(4*\i+3)},{-13/(4*\i+3)})--({2/(4*\i+3)},{-11/(4*\i+3)})--cycle;
\filldraw[fill=amber, draw=amber] ({-2/(4*\i+1)},{11/(4*\i+1)})--({-2/(4*\i+1)},{13/(4*\i+1)})--({-2/(4*\i+3)},{13/(4*\i+3)})--({-2/(4*\i+3)},{11/(4*\i+3)})--cycle;
\filldraw[fill=amber, draw=amber] ({-2/(4*\i+1)},{-11/(4*\i+1)})--({-2/(4*\i+1)},{-13/(4*\i+1)})--({-2/(4*\i+3)},{-13/(4*\i+3)})--({-2/(4*\i+3)},{-11/(4*\i+3)})--cycle;
}

\foreach \i in {9,...,12}
{

\filldraw[fill=awesome, draw=awesome] ({2/(4*\i+3)},{17/(4*\i+3)})--({2/(4*\i+3)},{19/(4*\i+3)})--({2/(4*\i+5)},{19/(4*\i+5)})--({2/(4*\i+5)},{17/(4*\i+5)})--cycle;
\filldraw[fill=awesome, draw=awesome] ({2/(4*\i+3)},{-17/(4*\i+3)})--({2/(4*\i+3)},{-19/(4*\i+3)})--({2/(4*\i+5)},{-19/(4*\i+5)})--({2/(4*\i+5)},{-17/(4*\i+5)})--cycle;
\filldraw[fill=awesome, draw=awesome] ({-2/(4*\i+3)},{17/(4*\i+3)})--({-2/(4*\i+3)},{19/(4*\i+3)})--({-2/(4*\i+5)},{19/(4*\i+5)})--({-2/(4*\i+5)},{17/(4*\i+5)})--cycle;
\filldraw[fill=awesome, draw=awesome] ({-2/(4*\i+3)},{-17/(4*\i+3)})--({-2/(4*\i+3)},{-19/(4*\i+3)})--({-2/(4*\i+5)},{-19/(4*\i+5)})--({-2/(4*\i+5)},{-17/(4*\i+5)})--cycle;

\filldraw[fill=pastelgreen, draw=pastelgreen] ({2/(4*\i-1)},{15/(4*\i-1)})--({2/(4*\i-1)},{17/(4*\i-1)})--({2/(4*\i+1)},{17/(4*\i+1)})--({2/(4*\i+1)},{15/(4*\i+1)})--cycle;
\filldraw[fill=pastelgreen, draw=pastelgreen] ({2/(4*\i-1)},{-15/(4*\i-1)})--({2/(4*\i-1)},{-17/(4*\i-1)})--({2/(4*\i+1)},{-17/(4*\i+1)})--({2/(4*\i+1)},{-15/(4*\i+1)})--cycle;
\filldraw[fill=pastelgreen, draw=pastelgreen] ({-2/(4*\i-1)},{15/(4*\i-1)})--({-2/(4*\i-1)},{17/(4*\i-1)})--({-2/(4*\i+1)},{17/(4*\i+1)})--({-2/(4*\i+1)},{15/(4*\i+1)})--cycle;
\filldraw[fill=pastelgreen, draw=pastelgreen] ({-2/(4*\i-1)},{-15/(4*\i-1)})--({-2/(4*\i-1)},{-17/(4*\i-1)})--({-2/(4*\i+1)},{-17/(4*\i+1)})--({-2/(4*\i+1)},{-15/(4*\i+1)})--cycle;

\filldraw[fill=ballblue, draw=ballblue] ({2/(4*\i-3)},{13/(4*\i-3)})--({2/(4*\i-3)},{15/(4*\i-3)})--({2/(4*\i-1)},{15/(4*\i-1)})--({2/(4*\i-1)},{13/(4*\i-1)})--cycle;
\filldraw[fill=ballblue, draw=ballblue] ({2/(4*\i-3)},{-13/(4*\i-3)})--({2/(4*\i-3)},{-15/(4*\i-3)})--({2/(4*\i-1)},{-15/(4*\i-1)})--({2/(4*\i-1)},{-13/(4*\i-1)})--cycle;
\filldraw[fill=ballblue, draw=ballblue] ({-2/(4*\i-3)},{13/(4*\i-3)})--({-2/(4*\i-3)},{15/(4*\i-3)})--({-2/(4*\i-1)},{15/(4*\i-1)})--({-2/(4*\i-1)},{13/(4*\i-1)})--cycle;
\filldraw[fill=ballblue, draw=ballblue] ({-2/(4*\i-3)},{-13/(4*\i-3)})--({-2/(4*\i-3)},{-15/(4*\i-3)})--({-2/(4*\i-1)},{-15/(4*\i-1)})--({-2/(4*\i-1)},{-13/(4*\i-1)})--cycle;

\filldraw[fill=amber, draw=amber] ({2/(4*\i+1)},{15/(4*\i+1)})--({2/(4*\i+1)},{17/(4*\i+1)})--({2/(4*\i+3)},{17/(4*\i+3)})--({2/(4*\i+3)},{15/(4*\i+3)})--cycle;
\filldraw[fill=amber, draw=amber] ({2/(4*\i+1)},{-15/(4*\i+1)})--({2/(4*\i+1)},{-17/(4*\i+1)})--({2/(4*\i+3)},{-17/(4*\i+3)})--({2/(4*\i+3)},{-15/(4*\i+3)})--cycle;
\filldraw[fill=amber, draw=amber] ({-2/(4*\i+1)},{15/(4*\i+1)})--({-2/(4*\i+1)},{17/(4*\i+1)})--({-2/(4*\i+3)},{17/(4*\i+3)})--({-2/(4*\i+3)},{15/(4*\i+3)})--cycle;
\filldraw[fill=amber, draw=amber] ({-2/(4*\i+1)},{-15/(4*\i+1)})--({-2/(4*\i+1)},{-17/(4*\i+1)})--({-2/(4*\i+3)},{-17/(4*\i+3)})--({-2/(4*\i+3)},{-15/(4*\i+3)})--cycle;
}

\foreach \i in {11,...,13}
{

\filldraw[fill=awesome, draw=awesome] ({2/(4*\i+3)},{21/(4*\i+3)})--({2/(4*\i+3)},{23/(4*\i+3)})--({2/(4*\i+5)},{23/(4*\i+5)})--({2/(4*\i+5)},{21/(4*\i+5)})--cycle;
\filldraw[fill=awesome, draw=awesome] ({2/(4*\i+3)},{-21/(4*\i+3)})--({2/(4*\i+3)},{-23/(4*\i+3)})--({2/(4*\i+5)},{-23/(4*\i+5)})--({2/(4*\i+5)},{-21/(4*\i+5)})--cycle;
\filldraw[fill=awesome, draw=awesome] ({-2/(4*\i+3)},{21/(4*\i+3)})--({-2/(4*\i+3)},{23/(4*\i+3)})--({-2/(4*\i+5)},{23/(4*\i+5)})--({-2/(4*\i+5)},{21/(4*\i+5)})--cycle;
\filldraw[fill=awesome, draw=awesome] ({-2/(4*\i+3)},{-21/(4*\i+3)})--({-2/(4*\i+3)},{-23/(4*\i+3)})--({-2/(4*\i+5)},{-23/(4*\i+5)})--({-2/(4*\i+5)},{-21/(4*\i+5)})--cycle;

\filldraw[fill=pastelgreen, draw=pastelgreen] ({2/(4*\i-1)},{19/(4*\i-1)})--({2/(4*\i-1)},{21/(4*\i-1)})--({2/(4*\i+1)},{21/(4*\i+1)})--({2/(4*\i+1)},{19/(4*\i+1)})--cycle;
\filldraw[fill=pastelgreen, draw=pastelgreen] ({2/(4*\i-1)},{-19/(4*\i-1)})--({2/(4*\i-1)},{-21/(4*\i-1)})--({2/(4*\i+1)},{-21/(4*\i+1)})--({2/(4*\i+1)},{-19/(4*\i+1)})--cycle;
\filldraw[fill=pastelgreen, draw=pastelgreen] ({-2/(4*\i-1)},{19/(4*\i-1)})--({-2/(4*\i-1)},{21/(4*\i-1)})--({-2/(4*\i+1)},{21/(4*\i+1)})--({-2/(4*\i+1)},{19/(4*\i+1)})--cycle;
\filldraw[fill=pastelgreen, draw=pastelgreen] ({-2/(4*\i-1)},{-19/(4*\i-1)})--({-2/(4*\i-1)},{-21/(4*\i-1)})--({-2/(4*\i+1)},{-21/(4*\i+1)})--({-2/(4*\i+1)},{-19/(4*\i+1)})--cycle;

\filldraw[fill=ballblue, draw=ballblue] ({2/(4*\i-3)},{17/(4*\i-3)})--({2/(4*\i-3)},{19/(4*\i-3)})--({2/(4*\i-1)},{19/(4*\i-1)})--({2/(4*\i-1)},{17/(4*\i-1)})--cycle;
\filldraw[fill=ballblue, draw=ballblue] ({2/(4*\i-3)},{-17/(4*\i-3)})--({2/(4*\i-3)},{-19/(4*\i-3)})--({2/(4*\i-1)},{-19/(4*\i-1)})--({2/(4*\i-1)},{-17/(4*\i-1)})--cycle;
\filldraw[fill=ballblue, draw=ballblue] ({-2/(4*\i-3)},{17/(4*\i-3)})--({-2/(4*\i-3)},{19/(4*\i-3)})--({-2/(4*\i-1)},{19/(4*\i-1)})--({-2/(4*\i-1)},{17/(4*\i-1)})--cycle;
\filldraw[fill=ballblue, draw=ballblue] ({-2/(4*\i-3)},{-17/(4*\i-3)})--({-2/(4*\i-3)},{-19/(4*\i-3)})--({-2/(4*\i-1)},{-19/(4*\i-1)})--({-2/(4*\i-1)},{-17/(4*\i-1)})--cycle;

\filldraw[fill=amber, draw=amber] ({2/(4*\i+1)},{19/(4*\i+1)})--({2/(4*\i+1)},{21/(4*\i+1)})--({2/(4*\i+3)},{21/(4*\i+3)})--({2/(4*\i+3)},{19/(4*\i+3)})--cycle;
\filldraw[fill=amber, draw=amber] ({2/(4*\i+1)},{-19/(4*\i+1)})--({2/(4*\i+1)},{-21/(4*\i+1)})--({2/(4*\i+3)},{-21/(4*\i+3)})--({2/(4*\i+3)},{-19/(4*\i+3)})--cycle;
\filldraw[fill=amber, draw=amber] ({-2/(4*\i+1)},{19/(4*\i+1)})--({-2/(4*\i+1)},{21/(4*\i+1)})--({-2/(4*\i+3)},{21/(4*\i+3)})--({-2/(4*\i+3)},{19/(4*\i+3)})--cycle;
\filldraw[fill=amber, draw=amber] ({-2/(4*\i+1)},{-19/(4*\i+1)})--({-2/(4*\i+1)},{-21/(4*\i+1)})--({-2/(4*\i+3)},{-21/(4*\i+3)})--({-2/(4*\i+3)},{-19/(4*\i+3)})--cycle;
}

\draw(-.5,-.5)--(.5,-.5)--(.5,.5)--(-.5,.5)--(-.5,-.5);
\draw[dotted](0,-.5)--(0,.5)(-.5,0)--(.5,0);
\draw(2/5,-.5)--(2/5,.5)(2/7,-.5)--(2/7,.5)(2/9,-.5)--(2/9,.5)(2/11,-.5)--(2/11,.5)(2/13,-.5)--(2/13,.5)(2/15,-.5)--(2/15,.5)(2/17,-.5)--(2/17,.5)(2/19,-.5)--(2/19,.5)(2/21,-.5)--(2/21,.5)(2/23,-.5)--(2/23,.5)(2/25,-.5)--(2/25,.5)(2/27,-.5)--(2/27,.5)(2/29,-.5)--(2/29,.5)(2/31,-.5)--(2/31,.5)(2/33,-.5)--(2/33,.5)(2/35,-.5)--(2/35,.5)(2/37,-.5)--(2/37,.5)(2/39,-.5)--(2/39,.5);
\draw(-2/5,-.5)--(-2/5,.5)(-2/7,-.5)--(-2/7,.5)(-2/9,-.5)--(-2/9,.5)(-2/11,-.5)--(-2/11,.5)(-2/13,-.5)--(-2/13,.5)(-2/15,-.5)--(-2/15,.5)(-2/17,-.5)--(-2/17,.5)(-2/19,-.5)--(-2/19,.5)(-2/21,-.5)--(-2/21,.5)(-2/23,-.5)--(-2/23,.5)(-2/25,-.5)--(-2/25,.5)(-2/27,-.5)--(-2/27,.5)(-2/29,-.5)--(-2/29,.5)(-2/31,-.5)--(-2/31,.5)(-2/33,-.5)--(-2/33,.5)(-2/35,-.5)--(-2/35,.5)(-2/37,-.5)--(-2/37,.5)(-2/39,-.5)--(-2/39,.5);
%\draw(-.5,-.25)--(.5,.25);
\draw[smooth, samples =20, domain=-.5:.5] plot(\x,{\x / 2});
\draw[smooth, samples =20, domain=-1/3:1/3] plot(\x,{3*\x / 2});
\draw[smooth, samples =20, domain=-1/5:1/5] plot(\x,{5*\x / 2});
\draw[smooth, samples =20, domain=-1/7:1/7] plot(\x,{7*\x / 2});
\draw[smooth, samples =20, domain=-1/9:1/9] plot(\x,{9*\x / 2});
\draw[smooth, samples =20, domain=-1/11:1/11] plot(\x,{11*\x / 2});
\draw[smooth, samples =20, domain=-1/13:1/13] plot(\x,{13*\x / 2});
\draw[smooth, samples =20, domain=-1/15:1/15] plot(\x,{15*\x / 2});
\draw[smooth, samples =20, domain=-1/17:1/17] plot(\x,{17*\x / 2});
\draw[smooth, samples =20, domain=-1/19:1/19] plot(\x,{19*\x / 2});

\draw[smooth, samples =20, domain=-.5:.5] plot(\x,{-\x / 2});
\draw[smooth, samples =20, domain=-1/3:1/3] plot(\x,{-3*\x / 2});
\draw[smooth, samples =20, domain=-1/5:1/5] plot(\x,{-5*\x / 2});
\draw[smooth, samples =20, domain=-1/7:1/7] plot(\x,{-7*\x / 2});
\draw[smooth, samples =20, domain=-1/9:1/9] plot(\x,{-9*\x / 2});
\draw[smooth, samples =20, domain=-1/11:1/11] plot(\x,{-11*\x / 2});
\draw[smooth, samples =20, domain=-1/13:1/13] plot(\x,{-13*\x / 2});
\draw[smooth, samples =20, domain=-1/15:1/15] plot(\x,{-15*\x / 2});
\draw[smooth, samples =20, domain=-1/17:1/17] plot(\x,{-17*\x / 2});
\draw[smooth, samples =20, domain=-1/19:1/19] plot(\x,{-19*\x / 2});

\foreach \i in {10,...,50}
{
 \draw[smooth, samples =20, domain=-1/(2*\i-1):1/(2*\i-1)] plot(\x,{(2*\i-1)*\x / 2});
 \draw[smooth, samples =20, domain=-1/(2*\i-1):1/(2*\i-1)] plot(\x,{-(2*\i-1)*\x / 2});
}

\node[below] at (.41,0) {\tiny $\frac25$};
\node[below] at (2/7+0.01,0) {\tiny $\frac27$};
\node[below] at (2/9+0.01,0) {\tiny $\frac29$};
\node[right] at (.5,.25) {\tiny $\frac{x}{2}$};
\node[right] at (.5,-.25) {\tiny $-\frac{x}{2}$};
\node[above] at (1/3,1/2) {\tiny $\frac{3x}{2}$};
\node[above] at (-1/3,1/2) {\tiny $-\frac{3x}{2}$};

\node[below] at (-.43,0) {\tiny $-\frac25$};
\node[below] at (-2/7-0.03,0) {\tiny $-\frac27$};
\node[below] at (-2/9-0.03,0) {\tiny $-\frac29$};
\node[above] at (1/5,1/2) {\tiny $\frac{5x}{2}$};
\node[above] at (-1/5,1/2) {\tiny $-\frac{5x}{2}$};

\node at (.45,.05) {\tiny $C_{2,0}$};
\node at (-.45,.05) {\tiny $C_{-2,0}$};
\node at (.45,.3) {\tiny $C_{2,1}$};
\node at (-.45,-.3) {\tiny $C_{-2,1}$};
\node at (.45,-.3) {\tiny $C_{2,-1}$};
\node at (-.43,.3) {\tiny $C_{-2,-1}$};
\node at (.33,.05) {\tiny $C_{3,0}$};
\node at (-.33,.05) {\tiny $C_{-3,0}$};
\node at (.33,.25) {\tiny $C_{3,1}$};
\node at (.33,-.25) {\tiny $C_{3,-1}$};
\node at (.315,.48) {\tiny $C_{3,2}$};
\node at (.32,-.48) {\tiny $C_{3,-2}$};
\node at (.25,.05) {\tiny $C_{4,0}$};
\node at (.25,.2) {\tiny $C_{4,1}$};
\node at (.25,.46) {\tiny $C_{4,2}$};
%\node at (.25,-.2) {\tiny $C_{4,-1}$};
%\node at (.253,-.46) {\tiny $C_{4,-2}$};

\foreach \i in {20,...,55}
{
    \draw( {2/(2*\i+1)}, -.5)--({2/(2*\i+1)},.5);
    \draw( {-2/(2*\i+1)}, -.5)--({-2/(2*\i+1)},.5);
 }
 
\filldraw[fill=black, draw=black] (-.02,-.5) rectangle (.02,.5);
\end{tikzpicture}
\label{f:partition2}\caption{The cylinder sets of the NIJP. The colors just serve to distinguish the different cylinder sets.}
\end{center}
\end{figure}

\vskip .2cm
Note that $C_{a,b_2, \ldots, b_d} \neq \emptyset$ if and only if $|a|\ge 2$ and $0\le |b_i| \le \lceil \frac{a}{2}\rceil$ for each $2\le i \le d$. This implies that for the linear version of the NIJP algorithm not all matrices of the form
\begin{equation}\label{q:matrixnijp}
A_0= \begin{pmatrix}
a & 0 & 0 & \cdots & 0 & 1\\
1 & 0 & 0 & \cdots & 0 & 0\\
b_2 & 1& 0 & \cdots & 0 & 0\\
b_3 & 0 & 1 & \cdots & 0 & 0\\
\vdots & \vdots & \vdots & \ddots & \vdots & \vdots\\
b_d & 0 & 0 & \cdots & 1 & 0
\end{pmatrix},
\end{equation}
are allowed. It needs to hold that $|a|\ge 2$ and $0\le |b_i| \le \lceil \frac{a}{2}\rceil$ for each $2\le i \le d$, but there are more restrictions. Below we describe these restrictions in detail for the case $d=2$ by showing that in this case $T_0$ admits a {\em Markov partition}, i.e., there exists a finite collection $\mathcal P$ of polygonal subsets of $[-\frac12,\frac12]^2$ that have the property that for any set $P \cap A$ with $P \in \mathcal P$ and $A \in \mathcal C$ there exists $P_1, \ldots, P_N \in \mathcal P$, such that $T_0 (P \cap A) = \bigcup_{i=1}^N P_i$ up to sets of zero Lebesgue measure. Figure~\ref{f:mp} shows the 20 sets that are in $\mathcal P$ for the map $T_0$.

%For all $i$ with $1 \leq i \leq d$, one has $\lfloor \frac{x_i}{x_1} +\frac12 \rfloor \leq  \lfloor \frac{1}{x_1} +\frac12 \rfloor$, hence $b^{(i)} \leq a.$ If  $\lfloor \frac{x_i}{x_1} +1/2 \rfloor \leq  \lfloor \frac{1}{x_1} +1/2 \rfloor$.

\begin{figure}[H]
\begin{center}
\begin{tikzpicture}[scale=3]
\draw(-.5,-.5)--(.5,-.5)--(.5,.5)--(-.5,.5)--(-.5,-.5);
\draw(0,-.5)--(0,.5)(-.5,0)--(.5,0)(-.5,-.5)--(.5,.5)(-.5,.5)--(.5,-.5);
\draw (-1/2,0)--(-1/4,-1/2)--(1/4,1/2)--(1/2,0)--(1/4,-1/2)--(-1/4,1/2)--cycle;
\node[left] at (-.5,0) {\footnotesize $-\frac12$};
\node[right] at (.5,0) {\footnotesize $\frac12$};
\node[below] at (0,-.5) {\footnotesize $-\frac12$};
\node[above] at (0,.5) {\footnotesize $\frac12$};
\node at (.08,.35) {\footnotesize $E_1$};
\node at (.25,.35) {\footnotesize $F_1$};
\node at (.35,.45) {\footnotesize $G_1$};
\node at (.25,.1) {\footnotesize $H_1$};
\node at (.43,.3) {\footnotesize $J_1$};

\node at (-.08,.35) {\footnotesize $E_2$};
\node at (-.25,.35) {\footnotesize $F_2$};
\node at (-.35,.45) {\footnotesize $G_2$};
\node at (-.25,.1) {\footnotesize $H_2$};
\node at (-.43,.3) {\footnotesize $J_2$};

\node at (-.08,-.35) {\footnotesize $E_3$};
\node at (-.25,-.35) {\footnotesize $F_3$};
\node at (-.35,-.45) {\footnotesize $G_3$};
\node at (-.25,-.1) {\footnotesize $H_3$};
\node at (-.43,-.3) {\footnotesize $J_3$};

\node at (.08,-.35) {\footnotesize $E_4$};
\node at (.25,-.35) {\footnotesize $F_4$};
\node at (.35,-.45) {\footnotesize $G_4$};
\node at (.25,-.1) {\footnotesize $H_4$};
\node at (.43,-.3) {\footnotesize $J_4$};

\end{tikzpicture}
\caption{The Markov partition for the NIJP map for $d=2$.}
\label{f:mp}
\end{center}
\end{figure}
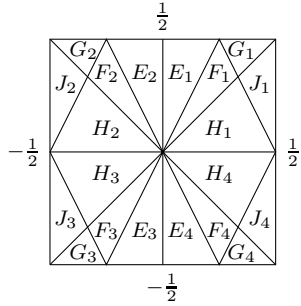

\begin{theorem}
Let $d=2$. Let $\mathcal P$ be the collection of disjoint subsets of $C$ bounded by the ten lines $x_1=0$, $x_2 = 0$, $x_2=\pm x_1$, $x_2=\pm 2x_1$, $x_2=\pm 1 \pm 2x_1$. Then $\mathcal P$ is a Markov partition for $T_0$.
\end{theorem}

\begin{proof}
To show that $\mathcal P$ from Figure~\ref{f:mp} is a Markov partition, we need to consider the image under $T_0$ of all sets $P \cap C_{a,b}$ for $P \in \mathcal P$ and cylinder set $C_{a,b} \in \mathcal C$. Due to symmetry, it is enough to only consider the sets in the first quadrant. We label the sets of $\mathcal P$ by $E_i,F_i,G_i,H_i,J_i$, $1 \le i \le 4$, as shown in Figure~\ref{f:mp} and use $C_{a,b,P}$ to denote the set $P \cap C_{a,b}$ with $P \in \mathcal P$ and $C_{a,b} \in \mathcal C$. Figure~\ref{f:types} shows several of these sets in the first quadrant. Based on the images of the sets $C_{a,b,P}$ we distinguish 15 types, indicated by the different colors in Figure~\ref{f:types}. Hence, to prove that $T_0$ admits a Markov partition for $d=2$ we need to compute the image under $T_0$ of each of these types of sets. We describe what happens to the set $C_{2,0,H_1}$.

\vskip .2cm
On $C_{2,0,H_1}$ the map $T_0$ is given by $T_0 (x_1,x_2) = ( \frac{x_2}{x_1}, \frac1{x_1}-2 )$. The boundary of $C_{2,0,H_1}$ is the union of three sets (we take all intervals open, because the endpoints of the intervals have no impact on the Lebesgue measure of the sets):
\[ \begin{split}
\partial_1 =\ & \left\{ (x_1,x_2) \in C \, : \, x_2=0, \, x_1 \in \left(\frac25, \frac12 \right) \right\},\\
\partial_2 =\ & \left\{ (x_1,x_2) \in C \, : \, x_1=\frac25, \, x_2 \in \left(0, \frac15 \right) \right\},\\
\partial_3 =\ & \left\{ (x_1, x_2) \in C\, : \, x_1 \in \left(\frac25, \frac12 \right), \, x_2 = 1-2x_1 \right\}.
\end{split}\]
Then
\[ \begin{split}
T_0(\partial_1) =\ & \left\{ (x_1,x_2) \in C \, : \, x_1=0, \, x_2 \in \left(0, \frac12 \right) \right\},\\
T_0(\partial_2) =\ & \left\{ (x_1,x_2) \in C \, : \, x_2=\frac12, \, x_1 \in \left(0, \frac12 \right) \right\},\\
T_0(\partial_3) =\ & \left\{ (x_1, x_2) \in C\, : \, x_2=x_1, \, x_1 \in \left(0, \frac12 \right) \right\}.
\end{split}\]
From this we can conclude that $T_0 (C_{2,0,H_1}) = E_1 \cup F_1 \cup G_1$. A similar computation can be done for each of the sets $C_{a,b,P}$. Table~\ref{t:images} lists the images of each type of set $C_{a,b,P}$ in the first quadrant. See also Figure~\ref{f:images}. By symmetry, similar results are obtained for sets $C_{a,b,P}$ in the other quadrants, from which we can deduce that the collection $\mathcal P$ is a Markov partition for $T_0$ with $d=2$.
\end{proof}

{\footnotesize
\begin{table}[h]
    \centering
    \begin{tabular}{|l|l|l|}
\hline
& $C_{a,b,P}$ & $T(C_{a,b,P})$\\
\hline
type 1 & $C_{2,0,H_1}$, $C_{3,1,F_1}$ & $E_1 \cup F_1 \cup G_1$  \\
\hline
type 2 & $C_{2,0,J_1}$ & $H_1 \cup J_1$\\
\hline
type 3 & $C_{2,1,J_1}$ & $E_2 \cup F_2 \cup G_2 \cup H_2 \cup J_2$\\
\hline
type 4 & $C_{2,1,G_1}$, $C_{4,2,E_1}$ & $E_1$\\
\hline
type 5 & $C_{a,0,H_1}$ for $a \ge 3$, $C_{a,1,F_1}$ for $a \ge 4$, & $\bigcup_{i=1,4} (E_i \cup F_i \cup G_i \cup H_i \cup J_i)$\\
& $C_{a,2,E_1}$ for $a \ge 6$ &\\
\hline
type 6 & $C_{3,1,H_1}$, $C_{4,2,F_1}$ & $E_2\cup F_2 \cup G_2 \cup H_2 \cup J_2 \cup H_3 \cup J_3$\\
\hline
type 7 & $C_{3,1,J_1}$ & $E_3 \cup F_3 \cup G_3$\\
\hline
type 8 & $C_{3,1,G_1}$ & $H_1 \cup J_1 \cup E_4 \cup F_4 \cup H_4$\\
\hline
type 9 & $C_{3,2,G_1}$, $C_{2k-1,k,E_1}$ for $k \ge 3$ & $G_2 \cup J_2$\\
\hline
type 10 & $C_{a,1,H_1}$ for $a \ge 4$, $C_{a,2,F_1}$ for $a \ge 5$ & $\bigcup_{i=2,3} (E_i \cup F_i \cup G_i \cup H_i \cup J_i)$\\
\hline
type 11 & $C_{4,2,G_1}$ & $F_3 \cup G_3$\\
\hline
type 12 & $C_{5,2,E_1}$ & $E_1 \cup F_1 \cup G_1 \cup H_1 \cup J_1 \cup E_4 \cup F_4 \cup H_4$\\
\hline
type 13 & $C_{2k,k,E_1}$ for $k \ge 3$ & $E_1 \cup E_2 \cup F_2 \cup G_2 \cup H_2 \cup J_2 \cup F_3 \cup G_3 \cup H_3 \cup J_3$\\
\hline
type 14 & $C_{2k+1,k,E_1}$ for $k \ge 3$ & $C \setminus (G_4 \cup J_4)$\\
\hline
type 15 & $C_{a,b,E_1}$ for $b \ge 3$ and $a \ge 2b+2$ & $C$\\
\hline
    \end{tabular}
    \vspace{.2cm}
    \caption{Images of the sets $C_{a,b,P}$ for $a,b\ge 0$ and $P \in \mathcal P$ under $T_0$.}
    \label{t:images}
\end{table}
}

\begin{figure}[H]
\begin{center}
\begin{tikzpicture}[scale=20]

\filldraw[fill=awesome, draw=awesome] (1/2,1/2)--(2/5,2/5)--(2/5,1/5)--(1/2,1/4)--(1/2,1/2);
\filldraw[fill=turquoise, draw=turquoise] (1/2,1/2)--(2/5,2/5)--(2/5,1/2)--(1/2,1/2);
\filldraw[fill=turquoise, draw=turquoise] (1/4,1/2)--(2/9,1/2)--(2/9,4/9)--cycle;

%\filldraw[fill=amber, draw=amber] (2/9,3/9)--(2/9,4/9)--(2/11,4/11)--(2/11,3/11)--cycle;

\filldraw[fill=lighttaupe, draw=lighttaupe] (2/9,4/9)--(2/9,1/2)--(1/5,1/2)--(2/11,5/11)--(2/11,4/11)--cycle;

\filldraw[fill=ballblue, draw=ballblue] (2/5,1/5)--(2/5,2/5)--(1/3,1/3)--cycle;
%\filldraw[fill=powderblue(web), draw=powderblue(web)] (2/7,3/7)--(2/7,2/7)--(1/3,1/3)--cycle;
\filldraw[fill=tangelo, draw=tangelo] (2/7,2/7)--(2/7,1/7)--(2/5,1/5)--(1/3,1/3)--cycle;
\filldraw[fill=tangelo, draw=tangelo] (2/9,1/3)--(2/9,4/9)--(1/4,1/2)--(2/7,3/7)--cycle;

\filldraw[fill=mauvelous, draw=mauvelous] (2/5,1/2)--(1/3,1/2)--(2/7,3/7)--(1/3,1/3)--(2/5,2/5)--cycle;

%\filldraw[fill=pinegreen, draw=pinegreen] (2/9,1/2)--(2/9,4/9)--(1/4,1/2)--cycle;

\filldraw[fill=jazzberryjam, draw=jazzberryjam] (1/3,1/2)--(2/7,1/2)--(2/7,3/7);

\filldraw[fill=jazzberryjam, draw=jazzberryjam] (1/5,1/2)--(2/11,1/2)--(2/11,5/11)--cycle;
\filldraw[fill=jazzberryjam, draw=jazzberryjam] (1/7,1/2)--(2/15,1/2)--(2/15,7/15)--cycle;
\filldraw[fill=jazzberryjam, draw=jazzberryjam] (1/9,1/2)--(2/19,1/2)--(2/19,9/19)--cycle;
\filldraw[fill=jazzberryjam, draw=jazzberryjam] (1/11,1/2)--(2/23,1/2)--(2/23,11/23)--cycle;
\filldraw[fill=jazzberryjam, draw=jazzberryjam] (1/13,1/2)--(2/27,1/2)--(2/27,13/27)--cycle;
\filldraw[fill=jazzberryjam, draw=jazzberryjam] (1/15,1/2)--(2/31,1/2)--(2/31,15/31)--cycle;

\filldraw[fill=lava, draw=lava] (0,0)--(2/5,0)--(2/5,1/5)--cycle;
\filldraw[fill=lava, draw=lava] (0,0)--(2/7,2/7)--(2/7,3/7)--cycle;
\filldraw[fill=lava, draw=lava] (0,0)--(2/11,4/11)--(2/11,5/11)--cycle;

\filldraw[fill=pastelgreen, draw=pastelgreen] (1/2,1/4)--(1/2,0)--(2/5,1/5)--(1/2,1/4);
\filldraw[fill=teagreen, draw=teagreen] (1/2,0)--(2/5,0)--(2/5,1/5)--cycle;
\filldraw[fill=teagreen, draw=teagreen] (2/7,2/7)--(2/7,3/7)--(1/3,1/3)--cycle;

\filldraw[fill=amber, draw=amber] (0,0)--(2/7,1/7)--(2/7,2/7)--cycle;
\filldraw[fill=amber, draw=amber] (0,0)--(2/9,1/3)--(2/9,4/9)--cycle;

%\filldraw[fill=lavenderpurple, draw=lavenderpurple] (2/7,1/2)--(1/4,1/2)--(2/9,4/9)--(2/9,3/9)--(2/7,3/7)--cycle;
\filldraw[fill=lavenderpurple, draw=lavenderpurple] (2/7,1/2)--(1/4,1/2)--(2/7,3/7)--cycle;

\filldraw[fill=bleudefrance, draw=bleudefrance] (0,0)--(2/15,5/15)--(2/15,7/15)--(2/19,7/19)--(2/19,9/19)--(2/23,9/23)--(2/23,11/23)--(2/27,11/27)--(2/27,13/27)--(2/31,13/31)--(2/31,15/31)--(2/35,15/35)--(2/35,17/35)--(2/39,17/39)--(2/39,19/39)--(2/43,19/43)--(2/43,21/43)--(2/47,21/47)--(2/47,23/47)--(2/51,23/51)--(2/51,25/51)--(2/55,25/55)--(2/55,27/55)--(2/59,27/59)--(2/59,29/59)--(2/63, 29/63)--(2/63,31/63)--(2/67,31/67)--(2/67,33/67)--(2/71,33/71)--(2/71,35/71)--cycle;

\filldraw[fill=pinegreen, draw=pinegreen] (2/15,5/15)--(2/13,5/13)--(2/13,1/2)--(1/7,1/2)--(2/15,7/15)--cycle;
\filldraw[fill=pinegreen, draw=pinegreen] (2/19,7/19)--(2/17,7/17)--(2/17,1/2)--(1/9,1/2)--(2/19,9/19)--cycle;
\filldraw[fill=pinegreen, draw=pinegreen] (2/23,9/23)--(2/21,9/21)--(2/21,1/2)--(1/11,1/2)--(2/23,11/23)--cycle;
\filldraw[fill=pinegreen, draw=pinegreen] (2/27,11/27)--(2/25,11/25)--(2/25,1/2)--(1/13,1/2)--(2/27,13/27)--cycle;
\filldraw[fill=pinegreen, draw=pinegreen] (2/31,13/31)--(2/29,13/29)--(2/29,1/2)--(1/15,1/2)--(2/31,15/31)--cycle;
\filldraw[fill=pinegreen, draw=pinegreen] (2/35,15/35)--(2/33,15/33)--(2/33,1/2)--(1/17,1/2)--(2/35,17/35)--cycle;
\filldraw[fill=pinegreen, draw=pinegreen] (2/39,17/39)--(2/37,17/37)--(2/37,1/2)--(1/19,1/2)--(2/39,19/39)--cycle;
\filldraw[fill=pinegreen, draw=pinegreen] (2/43,19/43)--(2/41,19/41)--(2/41,1/2)--(1/21,1/2)--(2/43,21/43)--cycle;
\filldraw[fill=pinegreen, draw=pinegreen] (2/47,21/47)--(2/45,21/45)--(2/45,1/2)--(1/23,1/2)--(2/47,23/47)--cycle;
\filldraw[fill=pinegreen, draw=pinegreen] (2/51,23/51)--(2/49,23/49)--(2/49,1/2)--(1/25,1/2)--(2/51,25/51)--cycle;
\filldraw[fill=pinegreen, draw=pinegreen] (2/55,25/55)--(2/53,25/53)--(2/53,1/2)--(1/27,1/2)--(2/55,27/55)--cycle;
\filldraw[fill=pinegreen, draw=pinegreen] (2/59,27/59)--(2/57,27/57)--(2/57,1/2)--(1/29,1/2)--(2/59,29/59)--cycle;

\filldraw[fill=bananamania, draw=bananamania] (2/13,5/13)--(2/11,5/11)--(2/11,1/2)--(2/13,1/2)--cycle;
\filldraw[fill=bananamania, draw=bananamania] (2/17,7/17)--(2/15,7/15)--(2/15,1/2)--(2/17,1/2)--cycle;
\filldraw[fill=bananamania, draw=bananamania] (2/21,9/21)--(2/19,9/19)--(2/19,1/2)--(2/21,1/2)--cycle;
\filldraw[fill=bananamania, draw=bananamania] (2/25,11/25)--(2/23,11/23)--(2/23,1/2)--(2/25,1/2)--cycle;
\filldraw[fill=bananamania, draw=bananamania] (2/29,13/29)--(2/27,13/27)--(2/27,1/2)--(2/29,1/2)--cycle;
\filldraw[fill=bananamania, draw=bananamania] (2/33,15/33)--(2/31,15/31)--(2/31,1/2)--(2/33,1/2)--cycle;
\filldraw[fill=bananamania, draw=bananamania] (2/37,17/37)--(2/35,17/35)--(2/35,1/2)--(2/37,1/2)--cycle;
\filldraw[fill=bananamania, draw=bananamania] (2/41,19/41)--(2/39,19/39)--(2/39,1/2)--(2/41,1/2)--cycle;
\filldraw[fill=bananamania, draw=bananamania] (2/45,21/45)--(2/43,21/43)--(2/43,1/2)--(2/45,1/2)--cycle;
\filldraw[fill=bananamania, draw=bananamania] (2/49,23/49)--(2/47,23/47)--(2/47,1/2)--(2/49,1/2)--cycle;
\filldraw[fill=bananamania, draw=bananamania] (2/53,25/53)--(2/51,25/51)--(2/51,1/2)--(2/53,1/2)--cycle;
\filldraw[fill=bananamania, draw=bananamania] (2/57,27/57)--(2/55,27/55)--(2/55,1/2)--(2/57,1/2)--cycle;

%\draw(0,-.5)--(0,.5);
\draw(2/5,0)--(2/5,.5)(2/7,0)--(2/7,.5)(2/9,0)--(2/9,.5)(2/11,0)--(2/11,.5)(2/13,0)--(2/13,.5)(2/15,0)--(2/15,.5)(2/17,0)--(2/17,.5)(2/19,0)--(2/19,.5)(2/21,0)--(2/21,.5)(2/23,0)--(2/23,.5)(2/25,0)--(2/25,.5)(2/27,0)--(2/27,.5)(2/29,0)--(2/29,.5)(2/31,0)--(2/31,.5)(2/33,0)--(2/33,.5)(2/35,0)--(2/35,.5)(2/37,0)--(2/37,.5)(2/39,0)--(2/39,.5);

\draw[smooth, samples =20, domain=0:.5] plot(\x,{\x / 2});
\draw[smooth, samples =20, domain=0:1/3] plot(\x,{3*\x / 2});
\draw[smooth, samples =20, domain=0:1/5] plot(\x,{5*\x / 2});
\draw[smooth, samples =20, domain=0:1/7] plot(\x,{7*\x / 2});
\draw[smooth, samples =20, domain=0:1/9] plot(\x,{9*\x / 2});
\draw[smooth, samples =20, domain=0:1/11] plot(\x,{11*\x / 2});
\draw[smooth, samples =20, domain=0:1/13] plot(\x,{13*\x / 2});
\draw[smooth, samples =20, domain=0:1/15] plot(\x,{15*\x / 2});
\draw[smooth, samples =20, domain=0:1/17] plot(\x,{17*\x / 2});
\draw[smooth, samples =20, domain=0:1/19] plot(\x,{19*\x / 2});

\foreach \i in {10,...,50}
{
 \draw[smooth, samples =20, domain=0:1/(2*\i-1)] plot(\x,{(2*\i-1)*\x / 2});
}

\foreach \i in {20,...,55}
{
    \draw( {2/(2*\i+1)}, 0)--({2/(2*\i+1)},.5);
 }

\filldraw[fill=black, draw=black] (0,0) rectangle (.02,.5);
 \draw[thick,spirodiscoball] (0,0)--(1/4,1/2)(1/4,1/2)--(1/2,0)(0,0)--(1/2,1/2);
\draw[thick, spirodiscoball](0,0)--(.5,0)--(.5,.5)--(0,.5)--(0,0);
 
 \node at (.465,.15) {\tiny $C_{2,0,J_1}$};
  \node at (.435,.03) {\tiny $C_{2,0,H_1}$};
\node at (.45,.3) {\tiny $C_{2,1,J_1}$};
\node at (.435,.485) {\tiny $C_{2,1,G_1}$};

\node at (.33,.1) {\tiny $C_{3,0,H_1}$};
\node at (.37,.32) {\tiny $C_{3,1,J_1}$};
\node at (.35,.4) {\tiny $C_{3,1,G_1}$};
\node at (.33,.25) {\tiny $C_{3,1,H_1}$};

\node[rotate=55] at (.305,.483) {\tiny $C_{3,2,G_1}$};

\node[rotate=90] at (.165,.46) {\tiny $C_{6,3,E_1}$};
\node[rotate=55] at (.2,.44) {\tiny $C_{5,2,E_1}$};

\node at (.255,.2) {\tiny $C_{4,1,H_1}$};
\node[rotate=-55] at (.27,.475) {\tiny $C_{4,2,G_1}$};
 \node[rotate=90] at (.145,.42) {\tiny $C_{7,3,E_1}$};
\node[rotate=90] at (.125,.38) {\tiny $C_{8,3,E_1}$};

\node[below] at (0,0) {\tiny $0$};
\node[below] at (1/2,0) {\tiny $\frac12$};
\node[left] at (0,1/2) {\tiny $\frac12$};

\end{tikzpicture}
\caption{The black lines indicate the boundaries of cylinder sets. The blue lines indicate the boundaries of the Markov partition elements. The colors of the polygons indicate the different types of sets $C_{a,b,P}$ according to their images under the NIJP map. Set with the same color have the same type.}
\label{f:types}
\end{center}
\end{figure}
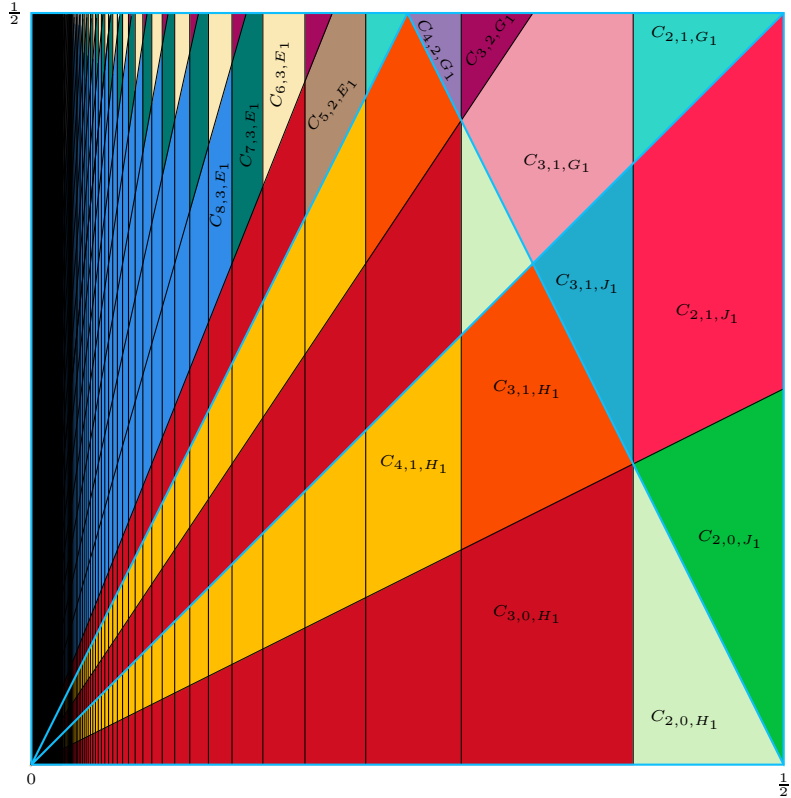

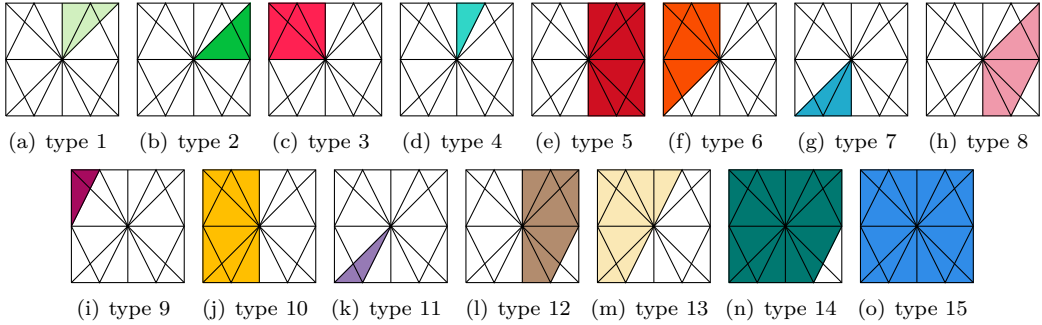
\begin{figure}[H]
\begin{center}
\subfigure[type 1]{
\begin{tikzpicture}[scale=1.5]

\filldraw[fill=teagreen, draw=teagreen] (0,0)--(0,1/2)--(1/2,1/2)--cycle;

\draw(-.5,-.5)--(.5,-.5)--(.5,.5)--(-.5,.5)--(-.5,-.5);
\draw(0,-.5)--(0,.5)(-.5,0)--(.5,0)(-.5,-.5)--(.5,.5)(-.5,.5)--(.5,-.5);
\draw (-1/2,0)--(-1/4,-1/2)--(1/4,1/2)--(1/2,0)--(1/4,-1/2)--(-1/4,1/2)--cycle;

\end{tikzpicture}}
\subfigure[type 2]{
\begin{tikzpicture}[scale=1.5]

\filldraw[fill=pastelgreen, draw=pastelgreen] (0,0)--(1/2,1/2)--(1/2,0)--cycle;

\draw(-.5,-.5)--(.5,-.5)--(.5,.5)--(-.5,.5)--(-.5,-.5);
\draw(0,-.5)--(0,.5)(-.5,0)--(.5,0)(-.5,-.5)--(.5,.5)(-.5,.5)--(.5,-.5);
\draw (-1/2,0)--(-1/4,-1/2)--(1/4,1/2)--(1/2,0)--(1/4,-1/2)--(-1/4,1/2)--cycle;

\end{tikzpicture}}
\subfigure[type 3]{
\begin{tikzpicture}[scale=1.5]

\filldraw[fill=awesome, draw=awesome] (-1/2,0)--(0,0)--(0,1/2)--(-1/2,1/2)--cycle;

\draw(-.5,-.5)--(.5,-.5)--(.5,.5)--(-.5,.5)--(-.5,-.5);
\draw(0,-.5)--(0,.5)(-.5,0)--(.5,0)(-.5,-.5)--(.5,.5)(-.5,.5)--(.5,-.5);
\draw (-1/2,0)--(-1/4,-1/2)--(1/4,1/2)--(1/2,0)--(1/4,-1/2)--(-1/4,1/2)--cycle;

\end{tikzpicture}}
\subfigure[type 4]{
\begin{tikzpicture}[scale=1.5]

\filldraw[fill=turquoise, draw=turquoise] (0,0)--(1/4,1/2)--(0,1/2)--cycle;

\draw(-.5,-.5)--(.5,-.5)--(.5,.5)--(-.5,.5)--(-.5,-.5);
\draw(0,-.5)--(0,.5)(-.5,0)--(.5,0)(-.5,-.5)--(.5,.5)(-.5,.5)--(.5,-.5);
\draw (-1/2,0)--(-1/4,-1/2)--(1/4,1/2)--(1/2,0)--(1/4,-1/2)--(-1/4,1/2)--cycle;

\end{tikzpicture}}
\subfigure[type 5]{
\begin{tikzpicture}[scale=1.5]

\filldraw[fill=lava, draw=lava] (0,-1/2)--(1/2,-1/2)--(1/2,1/2)--(0,1/2)--cycle;

\draw(-.5,-.5)--(.5,-.5)--(.5,.5)--(-.5,.5)--(-.5,-.5);
\draw(0,-.5)--(0,.5)(-.5,0)--(.5,0)(-.5,-.5)--(.5,.5)(-.5,.5)--(.5,-.5);
\draw (-1/2,0)--(-1/4,-1/2)--(1/4,1/2)--(1/2,0)--(1/4,-1/2)--(-1/4,1/2)--cycle;

\end{tikzpicture}}
\subfigure[type 6]{
\begin{tikzpicture}[scale=1.5]

\filldraw[fill=tangelo, draw=tangelo] (-1/2,-1/2)--(0,0)--(0,1/2)--(-1/2,1/2)--cycle;

\draw(-.5,-.5)--(.5,-.5)--(.5,.5)--(-.5,.5)--(-.5,-.5);
\draw(0,-.5)--(0,.5)(-.5,0)--(.5,0)(-.5,-.5)--(.5,.5)(-.5,.5)--(.5,-.5);
\draw (-1/2,0)--(-1/4,-1/2)--(1/4,1/2)--(1/2,0)--(1/4,-1/2)--(-1/4,1/2)--cycle;

\end{tikzpicture}}
\subfigure[type 7]{
\begin{tikzpicture}[scale=1.5]

\filldraw[fill=ballblue, draw=ballblue] (-1/2,-1/2)--(0,-1/2)--(0,0)--cycle;

\draw(-.5,-.5)--(.5,-.5)--(.5,.5)--(-.5,.5)--(-.5,-.5);
\draw(0,-.5)--(0,.5)(-.5,0)--(.5,0)(-.5,-.5)--(.5,.5)(-.5,.5)--(.5,-.5);
\draw (-1/2,0)--(-1/4,-1/2)--(1/4,1/2)--(1/2,0)--(1/4,-1/2)--(-1/4,1/2)--cycle;

\end{tikzpicture}}
\subfigure[type 8]{
\begin{tikzpicture}[scale=1.5]

\filldraw[fill=mauvelous, draw=mauvelous] (0,-1/2)--(1/4,-1/2)--(1/2,0)--(1/2,1/2)--(0,0)--cycle;

\draw(-.5,-.5)--(.5,-.5)--(.5,.5)--(-.5,.5)--(-.5,-.5);
\draw(0,-.5)--(0,.5)(-.5,0)--(.5,0)(-.5,-.5)--(.5,.5)(-.5,.5)--(.5,-.5);
\draw (-1/2,0)--(-1/4,-1/2)--(1/4,1/2)--(1/2,0)--(1/4,-1/2)--(-1/4,1/2)--cycle;

\end{tikzpicture}}
\subfigure[type 9]{
\begin{tikzpicture}[scale=1.5]

\filldraw[fill=jazzberryjam, draw=jazzberryjam] (-1/2,0)--(-1/4,1/2)--(-1/2,1/2)--cycle;

\draw(-.5,-.5)--(.5,-.5)--(.5,.5)--(-.5,.5)--(-.5,-.5);
\draw(0,-.5)--(0,.5)(-.5,0)--(.5,0)(-.5,-.5)--(.5,.5)(-.5,.5)--(.5,-.5);
\draw (-1/2,0)--(-1/4,-1/2)--(1/4,1/2)--(1/2,0)--(1/4,-1/2)--(-1/4,1/2)--cycle;

\end{tikzpicture}}
\subfigure[type 10]{
\begin{tikzpicture}[scale=1.5]

\filldraw[fill=amber, draw=amber] (-1/2,-1/2)--(0,-1/2)--(0,1/2)--(-1/2,1/2)--cycle;

\draw(-.5,-.5)--(.5,-.5)--(.5,.5)--(-.5,.5)--(-.5,-.5);
\draw(0,-.5)--(0,.5)(-.5,0)--(.5,0)(-.5,-.5)--(.5,.5)(-.5,.5)--(.5,-.5);
\draw (-1/2,0)--(-1/4,-1/2)--(1/4,1/2)--(1/2,0)--(1/4,-1/2)--(-1/4,1/2)--cycle;

\end{tikzpicture}}
\subfigure[type 11]{
\begin{tikzpicture}[scale=1.5]

\filldraw[fill=lavenderpurple, draw=lavenderpurple] (-1/2,-1/2)--(-1/4,-1/2)--(0,0)--cycle;

\draw(-.5,-.5)--(.5,-.5)--(.5,.5)--(-.5,.5)--(-.5,-.5);
\draw(0,-.5)--(0,.5)(-.5,0)--(.5,0)(-.5,-.5)--(.5,.5)(-.5,.5)--(.5,-.5);
\draw (-1/2,0)--(-1/4,-1/2)--(1/4,1/2)--(1/2,0)--(1/4,-1/2)--(-1/4,1/2)--cycle;

\end{tikzpicture}}
\subfigure[type 12]{
\begin{tikzpicture}[scale=1.5]

\filldraw[fill=lighttaupe, draw=lighttaupe] (0,-1/2)--(1/4,-1/2)--(1/2,0)--(1/2,1/2)--(0,1/2)--cycle;

\draw(-.5,-.5)--(.5,-.5)--(.5,.5)--(-.5,.5)--(-.5,-.5);
\draw(0,-.5)--(0,.5)(-.5,0)--(.5,0)(-.5,-.5)--(.5,.5)(-.5,.5)--(.5,-.5);
\draw (-1/2,0)--(-1/4,-1/2)--(1/4,1/2)--(1/2,0)--(1/4,-1/2)--(-1/4,1/2)--cycle;

\end{tikzpicture}}
\subfigure[type 13]{
\begin{tikzpicture}[scale=1.5]

\filldraw[fill=bananamania, draw=bananamania] (-1/2,-1/2)--(-1/4,-1/2)--(1/4,1/2)--(-1/2,1/2)--cycle;

\draw(-.5,-.5)--(.5,-.5)--(.5,.5)--(-.5,.5)--(-.5,-.5);
\draw(0,-.5)--(0,.5)(-.5,0)--(.5,0)(-.5,-.5)--(.5,.5)(-.5,.5)--(.5,-.5);
\draw (-1/2,0)--(-1/4,-1/2)--(1/4,1/2)--(1/2,0)--(1/4,-1/2)--(-1/4,1/2)--cycle;

\end{tikzpicture}}
\subfigure[type 14]{
\begin{tikzpicture}[scale=1.5]

\filldraw[fill=pinegreen, draw=pinegreen] (-1/2,-1/2)--(1/4,-1/2)--(1/2,0)--(1/2,1/2)--(-1/2,1/2)--cycle;

\draw(-.5,-.5)--(.5,-.5)--(.5,.5)--(-.5,.5)--(-.5,-.5);
\draw(0,-.5)--(0,.5)(-.5,0)--(.5,0)(-.5,-.5)--(.5,.5)(-.5,.5)--(.5,-.5);
\draw (-1/2,0)--(-1/4,-1/2)--(1/4,1/2)--(1/2,0)--(1/4,-1/2)--(-1/4,1/2)--cycle;

\end{tikzpicture}}
\subfigure[type 15]{
\begin{tikzpicture}[scale=1.5]

\filldraw[fill=bleudefrance, draw=bleudefrance] (-1/2,-1/2)--(1/2,-1/2)--(1/2,1/2)--(-1/2,1/2)--cycle;

\draw(-.5,-.5)--(.5,-.5)--(.5,.5)--(-.5,.5)--(-.5,-.5);
\draw(0,-.5)--(0,.5)(-.5,0)--(.5,0)(-.5,-.5)--(.5,.5)(-.5,.5)--(.5,-.5);
\draw (-1/2,0)--(-1/4,-1/2)--(1/4,1/2)--(1/2,0)--(1/4,-1/2)--(-1/4,1/2)--cycle;

\end{tikzpicture}}
\caption{The images of the sets from Figure~\ref{f:types}.}
\label{f:images}
\end{center}
\end{figure}

From Table~\ref{t:images} we can deduce the restrictions that apply to applications of the matrices from \eqref{q:matrixnijp}. For example, from the first two lines we read that the digit $(2,0)$ can only be followed by a digit $(a,b)$ with $a,b\ge 0$. Some restrictions are more complicated to describe and carry further. For example, the digit $(3,2)$ can be followed by the digits $(-2,0)$, $(-2,-1)$, $(-3,-1)$, $(-3,-2)$ and $(-4,-2)$. The digit $(-2,0)$ can in principle be followed by any digit $(a,b)$ with $a,b \le 0$, but if one sees $(3,2)$ followed by $(-2,0)$ then this can only be followed by those digit $(a,b)$ with $a,b \le 0$ such that $C_{a,b} \in H_3 \cup J_3$. Knowing which sequences of digits are allowed, tells us which matrix products involving matrices of the form \eqref{q:matrixnijp} we have to analyze in order to obtaining numerical information about the approximation properties of the nearest integer Jacobi--Perron algorithm. Giving a fuller description of the allowed digit sequences from the results in Table~\ref{t:images} would be a first step in this direction.

\vskip .2cm
Having a Markov partition can also help in finding invariant measures for the dynamical system given by $T_0$. Here we describe a different approach that might lead to an invariant measure that is absolutely continuous with respect to the Lebesgue measure. To prove the existence of an a.c.i.m, we follow the book of Schweiger \cite[Chapter 4]{SCHWEIGER}, where he analyzed the ergodic properties of the Jacobi--Perron algorithm (see also \cite{mayer,Broise:01}). The analysis is similar, we will be brief. First notice that only cylinders of type $15$ are full, i.e. $T_0([a,b)])=C$ up to sets of zero Lebesgue measure. All other types are non-full. However, after at most three iterations, any non-full cylinder is mapped to a region containing at least a fixed positive proportion, say $q$, of full cylinders. This implies that every non-full cylinder can be written as a countable union of disjoint full cylinders of higher rank (i.e. where more digits are specified), so that the collection of full cylinders generate the Borel $\sigma$-algebra. This then allows one to define a jump transformation with full cylinders and satisfying the conditions of R\'enyi \cite{Renyi} (see also Theorem 8 in  \cite{SCHWEIGER}). From this one concludes that the jump transformation admits an absolutely continuous invariant ergodic measure. Then, Theorem 11 (see also Theorem 18) of  \cite{SCHWEIGER} implies that $T_0$ admits a finite absolutely continuous invariant ergodic measure.

 %\subsection{More on the Markov partition} \label{subsec:nijpa}

\subsection{Experimental data}\label{subsec:expnijpa}
We now provide some experimental data due to Steiner \cite{Steiner:pc}
the Jacobi--Perron and the nearest integer Jacobi--Perron algorithms
indicating a better behaviour for the  nearest integer Jacobi--Perron algorithm in terms of Lyapunov exponents than  for its   usual  version.

For the nearest integer Jacobi--Perron algorithm (on the left) and  the  usual Jacobi--Perron algorithm (on the right), one gets the following
experimental data.

\begin{center}
\begin{tabular}{c|c|c||c|c|c}
$d$ & $\lambda_1$ & $\lambda_2$  & $1-\frac{\lambda_2}{\lambda_1}$  \\ \hline
d=2 & 1.72241 &-0.691444 &1.40144\\
d=3 & 1.72394 &-0.388217& 1.22520\\
d=4 & 1.72400& -0.24779 &1.14372\\
d=5 & 1.72408 &-0.16873 &1.09786\\
d=6 & 1.72417& -0.11892 &1.06897\\
d=7 & 1.72413 &-0.08522 &1.049430\\
d=8 & 1.72417 &-0.06122 &1.03551\\
d=9 & 1.72409 &-0.04347 &1.02522\\
d=10 &1.72414& -0.02995& 1.01737\\
d=11 &1.72413 &-0.01939& 1.01125\\
d=12& 1.72414 &-0.01102 & 1.00640\\
d=13 &1.72409& -0.00425 & 1.00247\\
d=14 &1.72414 & 0.001304 & 0.99924
\end{tabular}
\quad 
\begin{tabular}{c|c|c||c|c|c}
$d$ & $\lambda_1 $&$ \lambda_2$ &$1-\frac{\lambda_2}{\lambda_1}$  \\ \hline
d=2  &1.20052 &-0.448404 &1.37351\\
d=3  &1.18560 &-0.227877 &1.19220\\
d=4  &1.17295 &-0.13064& 1.11138\\
d=5 & 1.16579 &-0.07882& 1.067614\\
d=6 & 1.16224& -0.04798 &1.041279\\
d=7 & 1.16068 &-0.028202 &1.02430\\
d=8 & 1.15992 &-0.014708& 1.01268\\
d=9 & 1.15956 &-0.00505 &1.004358\\
d=10 &1.159476 & 0.00217& 0.99813\\
d=11 &1.159360& 0.00776& 0.993308\\
d=12 &1.159364 & 0.01221 &0.98946\\
d=13& 1.159401& 0.01586 &0.986320\\
d=14 & 1.15930 & 0.01889 &0.983705
\end{tabular}
\end{center}

\bibliographystyle{amsalpha}
\bibliography{Wine}
\end{document}